\documentclass[aap]{imsart}
\RequirePackage{natbib}
\usepackage{booktabs}
\usepackage{multirow}
\usepackage{etex}
\usepackage{amsmath}
\usepackage{amstext}
\usepackage{amsfonts}
\usepackage{amsthm}
\usepackage{verbatim}
\usepackage{amssymb}
\usepackage{graphicx}
\usepackage{bm}
\usepackage{xcolor}
\usepackage{hyperref}
\usepackage{dsfont}
\usepackage{tikz}
\usepackage{pgfplots,filecontents}
\usetikzlibrary{spy}
\usetikzlibrary{shapes}
\usetikzlibrary{plotmarks}

\tikzset{new spy style/.style={spy scope={%
	magnification=5,
	size=1.25cm,
	connect spies,
	every spy on node/.style={
		rectangle,
		draw,
	},
	every spy in node/.style={
		draw,
		rectangle,
		fill=gray!40,
	}
}}}

\usepackage{optional}

\newcommand{\RR}{{\mathbb{R}}}
\newcommand{\NN}{{\mathbb{N}}}

\newcommand{\CC}{{\mathbb{C}}}

\newcommand{\trans}{{\sf T}}

\DeclareMathOperator{\acos}{acos}
\DeclareMathOperator{\asin}{asin}

\newcommand{\EE}{{\rm E}}

\DeclareMathOperator{\tr}{tr}

\DeclareMathOperator{\diag}{\rm diag}

\newcommand{\RED}{\color[rgb]{0.70,0,0}}
\newcommand{\BLUE}{\color[rgb]{0,0,0.69}}
\newcommand{\GREEN}{\color[rgb]{0,0.6,0}}

\newcounter{ctheorem}
\newtheorem{theorem}[ctheorem]{Theorem}
\newcounter{cassumption}
\newtheorem{assumption}[cassumption]{Assumption}
\newcounter{cproposition}
\newtheorem{proposition}[cproposition]{Proposition}
\newcounter{ccorollary}
\newtheorem{corollary}[ccorollary]{Corollary}
\newcounter{clemma}
\newtheorem{lemma}[clemma]{Lemma}
\newcounter{cconj}
\newtheorem{conj}[cconj]{Conjecture}

\newcounter{cremark}
\newtheorem{remark}[cremark]{Remark}
\begin{document}

\begin{frontmatter}

\title{A Random Matrix Approach\\ to Neural Networks}
\runtitle{A Random Matrix Approach to Neural Networks}

\begin{aug}
	\author{\fnms{Cosme} \snm{Louart}\ead[label=e1]{cosme.louart@ens.fr}},
	\author{\fnms{Zhenyu} \snm{Liao}\ead[label=e2]{zhenyu.liao@centralesupelec.fr}},
	\and
	\author{\fnms{Romain} \snm{Couillet}\corref{}\thanksref{t3}\ead[label=e3]{romain.couillet@centralesupelec.fr}}

	\thankstext{t3}{Couillet's work is supported by the ANR Project RMT4GRAPH (ANR-14-CE28-0006).}

	\runauthor{C.\@ Louart et al.}
	
	\affiliation{CentraleSup\'elec, University of Paris--Saclay, France.}
\end{aug}



\begin{abstract}
	This article studies the Gram random matrix model $G=\frac1T\Sigma^\trans\Sigma$, $\Sigma=\sigma(WX)$, classically found in \textcolor{black}{the analysis of random feature maps and} random neural networks, where $X=[x_1,\ldots,x_T]\in\RR^{p\times T}$ is a (data) matrix of bounded norm, $W\in\RR^{n\times p}$ is a matrix of independent zero-mean unit variance entries, and $\sigma:\RR\to\RR$ is a Lipschitz continuous (activation) function --- $\sigma(WX)$ being understood entry-wise. \textcolor{black}{By means of a key concentration of measure lemma arising from non-asymptotic random matrix arguments,} we prove that, as $n,p,T$ grow large at the same rate, the resolvent $Q=(G+\gamma I_T)^{-1}$, for $\gamma>0$, has a similar behavior as that met in sample covariance matrix models, involving notably the moment $\Phi=\frac{T}n\EE[G]$, which provides in passing a deterministic equivalent for the empirical spectral measure of $G$. Application-wise, this result enables the estimation of the asymptotic performance of single-layer random neural networks. This in turn provides practical insights into the underlying mechanisms into play in random neural networks, entailing several unexpected consequences, as well as a fast practical means to tune the network hyperparameters.
\end{abstract}	

\begin{keyword}[class=MSC]
\kwd[Primary ]{60B20}
\kwd[; secondary ]{62M45}
\end{keyword}

\end{frontmatter}

\section{Introduction}

Artificial neural networks, developed in the late fifties \citep{ROS58} in an attempt to develop machines capable of brain-like behaviors, know today an unprecedented research interest, notably in its applications to computer vision and machine learning at large \citep{KRI12,SCH15} where superhuman performances on specific tasks are now commonly achieved. Recent progress in neural network performances however find their source in the processing power of modern computers as well as in the availability of large datasets rather than in the development of new mathematics. In fact, for lack of appropriate tools to understand the theoretical behavior of the non-linear activations and deterministic data dependence underlying these networks, the discrepancy between mathematical and practical (heuristic) studies of neural networks has kept widening. A first salient problem in harnessing neural networks lies in their being completely designed upon a deterministic training dataset $X=[x_1,\ldots,x_T]\in\RR^{p\times T}$, so that their resulting performances intricately depend first and foremost on $X$. Recent works have nonetheless established that, when smartly designed, mere randomly connected neural networks can achieve performances close to those reached by entirely data-driven network designs \citep{RAH07,SAX11}. As a matter of fact, to handle gigantic databases, the computationally expensive learning phase (the so-called backpropagation of the error method) typical of deep neural network structures becomes impractical, while it was recently shown that smartly designed single-layer random networks (as studied presently) can already reach superhuman capabilities \citep{CAM15} and beat expert knowledge in specific fields \citep{JAE04}. These various findings have opened the road to the study of neural networks by means of statistical and probabilistic tools \citep{CHO14,GIR15}. The second problem relates to the non-linear activation functions present at each neuron, which have long been known (as opposed to linear activations) to help design universal approximators for any input-output target map \citep{HOR89}.

In this work, we propose an original random matrix-based approach to understand the end-to-end regression performance of single-layer random artificial neural networks, sometimes referred to as extreme learning machines \citep{HUA06,HUA12}, when the number $T$ and size $p$ of the input dataset are large and scale proportionally with the number $n$ of neurons in the network. \textcolor{black}{These networks can also be seen, from a more immediate statistical viewpoint, as a mere linear ridge-regressor relating a {\it random feature map} $\sigma(WX)\in\RR^{n\times T}$ of explanatory variables $X=[x_1,\ldots,x_T]\in\RR^{p\times T}$ and target variables $y=[y_1,\ldots,y_T]\in\RR^{d\times T}$, for $W\in\RR^{n\times p}$ a randomly designed matrix and $\sigma(\cdot)$ a non-linear $\RR\to\RR$ function (applied component-wise)}. Our approach has several interesting features both for theoretical and practical considerations. It is first one of the few known attempts to move the random matrix realm away from matrices with independent or linearly dependent entries. Notable exceptions are the line of works surrounding kernel random matrices \citep{ELK10,COU16} as well as large dimensional robust statistics models \citep{COU13b,KAR13,ZHA14}. Here, to alleviate the non-linear difficulty, we exploit concentration of measure arguments \citep{LED05} \textcolor{black}{for non-asymptotic random matrices}, thereby pushing further the original ideas of \citep{ELK09,VER12} established for simpler random matrix models. While we believe that more powerful, albeit more computational intensive, tools (such as an appropriate adaptation of the Gaussian tools advocated in \citep{PAS11}) cannot be avoided to handle advanced considerations in neural networks, we demonstrate here that the concentration of measure phenomenon allows one to fully characterize the main quantities at the heart of the single-layer regression problem at hand. 

In terms of practical applications, our findings shed light on the already incompletely understood extreme learning machines which have proved extremely efficient in handling machine learning problems involving large to huge datasets \citep{HUA12,CAM15} at a computationally affordable cost. But our objective is also to pave to path to the understanding of more involved neural network structures, featuring notably multiple layers and some steps of learning by means of backpropagation of the error.

\bigskip

Our main contribution is twofold. \textcolor{black}{From a theoretical perspective, we first obtain a key lemma, Lemma~\ref{lem:concentration_quadform}, on the concentration of quadratic forms of the type $\sigma(w^\trans X)A\sigma(X^\trans w)$ where $w=\varphi(\tilde{w})$, $\tilde w\sim \mathcal N(0,I_p)$, with $\varphi:\RR\to\RR$ and $\sigma:\RR\to\RR$ Lipschitz functions, and $X\in\RR^{p\times T}$, $A\in\RR^{n\times n}$ are deterministic matrices. This {\it non-asymptotic} result (valid for all $n,p,T$) is then exploited under a simultaneous growth regime for $n,p,T$ and boundedness conditions on $\|X\|$ and $\|A\|$ to obtain, in Theorem~\ref{th:EQ}, a deterministic approximation $\bar Q$ of the resolvent $\EE[Q]$, where $Q=(\frac1T\Sigma^\trans\Sigma+\gamma I_T)^{-1}$, $\gamma>0$, $\Sigma=\sigma(WX)$, for some $W=\varphi(\tilde{W})$, $\tilde W\in\RR^{n\times p}$ having independent $\mathcal N(0,1)$ entries. As the resolvent of a matrix (or operator) is an important proxy for the characterization of its spectrum (see e.g., \citep{PAS11,AKH93}), this result therefore allows for the characterization of the asymptotic spectral properties of $\frac1T\Sigma^\trans\Sigma$, such as its limiting spectral measure in Theorem~\ref{th:lsd}.}

\textcolor{black}{
	Application-wise, the theoretical findings are an important preliminary step for the understanding and improvement of various statistical methods based on random features in the large dimensional regime. Specifically, here, we consider the question of linear ridge-regression from random feature maps, which coincides with the aforementioned single hidden-layer random neural network known as extreme learning machine. We show that, under mild conditions, both the {\it training} $E_{\rm train}$ and {\it testing} $E_{\rm test}$ mean-square errors, respectively corresponding to the regression errors on known input-output pairs $(x_1,y_1),\ldots,(x_T,y_T)$ (with $x_i\in\RR^p$, $y_i\in\RR^d$) and unknown pairings $(\hat x_1,\hat y_1),\ldots,(\hat x_{\hat T},\hat y_{\hat T})$, almost surely converge to deterministic limiting values as $n,p,T$ grow large at the same rate (while $d$ is kept constant) for every fixed ridge-regression parameter $\gamma>0$. 
}
Simulations on real image datasets are provided that corroborate our results.

\medskip

These findings provide new insights into the roles played by the activation function $\sigma(\cdot)$ and the random distribution of the entries of $W$ \textcolor{black}{in random feature maps as well as by} the ridge-regression parameter $\gamma$ in the neural network performance. We notably exhibit and prove some peculiar behaviors, such as the impossibility for the network to carry out elementary Gaussian mixture classification tasks, when either the activation function or the random weights distribution are ill chosen.

Besides, for the practitioner, the theoretical formulas retrieved in this work allow for a fast offline tuning of the aforementioned hyperparameters of the neural network, notably when $T$ is not too large compared to $p$. The graphical results provided in the course of the article were particularly obtained within a $100$- to $500$-fold gain in computation time between theory and simulations.

\bigskip

The remainder of the article is structured as follows: in Section~\ref{sec:model}, we introduce the mathematical model of the system under investigation. Our main results are then described and discussed in Section~\ref{sec:results}, the proofs of which are deferred to Section~\ref{sec:proofs}. Section~\ref{sec:discussion} discusses our main findings. The article closes on concluding remarks on envisioned extensions of the present work in Section~\ref{sec:conclusion}. The appendix provides some intermediary lemmas of constant use throughout the proof section.

\bigskip

{\it Reproducibility:} Python~3 codes used to produce the results of Section~\ref{sec:discussion} are available at \href{https://github.com/Zhenyu-LIAO/RMT4ELM}{https://github.com/Zhenyu-LIAO/RMT4ELM}

\bigskip

{\it Notations:} The norm $\|\cdot\|$ is understood as the Euclidean norm for vectors and the operator norm for matrices, while the norm $\|\cdot\|_F$ is the Frobenius norm for matrices. All vectors in the article are understood as column vectors.

\section{System Model}
\label{sec:model}

We consider a \textcolor{black}{ridge-regression task on random feature maps} defined as follows. Each input data $x\in\RR^p$ is multiplied by a matrix $W\in\RR^{n\times p}$; a non-linear function $\sigma:\RR\to\RR$ is then applied entry-wise to the vector $Wx$\textcolor{black}{, thereby providing a set of $n$ random features $\sigma(Wx)\in\RR^n$ for each datum $x\in\RR^p$}. The output $z\in\RR^d$ of the \textcolor{black}{linear regression} is the inner product $z=\beta^\trans \sigma(Wx)$ for some matrix $\beta\in\RR^{n\times d}$ to be designed.

\textcolor{black}{From a neural network viewpoint,} the $n$ neurons of the network are the virtual units operating the mapping $W_{i\cdot}x\mapsto \sigma(W_{i\cdot}x)$ \textcolor{black}{($W_{i\cdot}$ being the $i$-th row of $W$)}, for $1\leq i\leq n$. The neural network then operates in two phases: a training phase where the regression matrix $\beta$ is learned based on a known input-output dataset pair $(X,Y)$ and a testing phase where, for $\beta$ now fixed, the network operates on a new input dataset $\hat X$ with corresponding unknown output $\hat Y$.

\bigskip

During the training phase, based on a set of known input $X=[x_1,\ldots,x_T]\in\RR^{p\times T}$ and output $Y=[y_1,\ldots,y_T]\in\RR^{d\times T}$ datasets, the matrix $\beta$ is chosen so as to minimize the mean square error $\frac1T\sum_{i=1}^T\|z_i-y_i\|^2+\gamma \|\beta\|_F^2$, where $z_i=\beta^\trans \sigma(Wx_i)$ and $\gamma>0$ is some regularization factor. \textcolor{black}{Solving for $\beta$,} this leads to the explicit {\it ridge-regressor}
\begin{align*}
	\beta &= \frac1T\Sigma \left( \frac1T\Sigma^\trans\Sigma + \gamma I_T \right)^{-1}Y^\trans
\end{align*}
where we defined $\Sigma\equiv \sigma(WX)$.
\textcolor{black}{This follows from differentiating the mean square error along $\beta$ to obtain $0 = \gamma \beta + \frac1T\sum_{i=1}^T \sigma(Wx_i)(\beta^\trans \sigma(Wx_i)-y_i)^\trans$, so that $(\frac1T\Sigma\Sigma^\trans + \gamma I_n )\beta = \frac1T \Sigma Y^\trans$ which, along with $(\frac1T\Sigma\Sigma^\trans + \gamma I_n )^{-1}\Sigma=\Sigma(\frac1T\Sigma^\trans\Sigma + \gamma I_T )^{-1}$, gives the result.
}

In the remainder, we will also denote 
\begin{align*}
	Q\equiv \left( \frac1T\Sigma^\trans\Sigma + \gamma I_T\right)^{-1}
\end{align*}
the {\it resolvent} of $\frac1T\Sigma^\trans\Sigma$.
\textcolor{black}{The matrix $Q$ naturally appears as a key quantity in the performance analysis of the neural network.} Notably, the mean-square error $E_{\rm train}$ on the training dataset $X$ is given by
\begin{align}
	\label{eq:Etrain}
	E_{\rm train} &= \frac1T\left\| Y^\trans - \Sigma^\trans \beta \right\|_F^2 = \frac{\gamma^2}T \tr Y^\trans Y Q^2.
\end{align}
Under the growth rate assumptions on $n,p,T$ taken below, it shall appear that the random variable $E_{\rm train}$ concentrates around its mean, letting then appear $\EE[Q^2]$ as a central object in the asymptotic evaluation of $E_{\rm train}$.

\bigskip

The testing phase of the neural network is more interesting in practice as it unveils the actual performance of neural networks. For a test dataset $\hat{X}\in\RR^{p\times \hat T}$ of length $\hat T$, with unknown output $\hat Y\in\RR^{d\times \hat T}$, the test mean-square error is defined by
\begin{align*}
	E_{\rm test} &= \frac1T\left\| \hat{Y}^\trans - \hat{\Sigma}^\trans \beta \right\|_F^2
\end{align*}
where $\hat{\Sigma}=\sigma(W\hat X)$ and $\beta$ is the same as used in \eqref{eq:Etrain} (and thus only depends on $(X,Y)$ and $\gamma$).
One of the key questions in the analysis of such an elementary neural network lies in the determination of $\gamma$ which minimizes $E_{\rm test}$ (and is thus said to have good {\it generalization} performance). Notably, small $\gamma$ values are known to reduce $E_{\rm train}$ but to induce the popular {\it overfitting} issue which generally increases $E_{\rm test}$, while large $\gamma$ values engender both large values for $E_{\rm train}$ and $E_{\rm test}$.

From a mathematical standpoint though, the study of $E_{\rm test}$ brings forward some technical difficulties that do not allow for as a simple treatment through the present concentration of measure methodology as the study of $E_{\rm train}$. Nonetheless, the analysis of $E_{\rm train}$ allows at least for heuristic approaches to become available, which we shall exploit to propose an asymptotic deterministic approximation for $E_{\rm test}$.

\bigskip


\textcolor{black}{From a technical standpoint, we shall make the following set of assumptions on the mapping $x\mapsto \sigma(Wx)$.}

\begin{assumption}[Subgaussian $W$]
	\label{ass:W}
	The matrix $W$ is defined by 
	\begin{align*}
		W=\varphi(\tilde{W})
	\end{align*}
(understood entry-wise), where $\tilde{W}$ has independent and identically distributed $\mathcal N(0,1)$ entries and $\varphi(\cdot)$ is $\lambda_{\varphi}$-Lipschitz. 
\end{assumption}
For $a=\varphi(b)\in\RR^\ell$, $\ell\geq 1$, with $b\sim\mathcal N(0,I_\ell)$, we shall subsequently denote $a\sim \mathcal N_{\varphi}(0,I_\ell)$.

Under the notations of Assumption~\ref{ass:W}, we have in particular $W_{ij}\sim \mathcal N(0,1)$ if $\varphi(t)=t$ and $W_{ij}\sim \mathcal U(-1,1)$ (the uniform distribution on $[-1,1]$) if $\varphi(t)=-1+2\frac1{\sqrt{2\pi}}\int_t^\infty e^{-x^2}dx$ ($\varphi$ is here a $\sqrt{2/\pi}$-Lipschitz map).

\medskip

We further need the following regularity condition on the function $\sigma$.
\begin{assumption}[Function $\sigma$]
	\label{ass:sigma}
	The function $\sigma$ is Lipschitz continuous with parameter $\lambda_\sigma$.
\end{assumption}

This assumption holds for many of the activation functions traditionally considered in neural networks, such as sigmoid functions, the rectified linear unit $\sigma(t)=\max(t,0)$, or the absolute value operator.

\medskip 

\textcolor{black}{
	When considering the interesting case of simultaneously large data and random features (or neurons), we shall then make the following growth rate assumptions.
\begin{assumption}[Growth Rate]
	\label{ass:growth}
	As $n\to\infty$, 
	\begin{align*}
		0<\liminf_n \min\{p/n,T/n\}\leq \limsup_n \max\{p/n,T/n\}<\infty
	\end{align*}
	while $\gamma,\lambda_\sigma,\lambda_\varphi>0$ and $d$ are kept constant. In addition, 
	\begin{align*}
		\limsup_n \|X\| &<\infty \\
		\limsup_n \max_{ij} |Y_{ij}| &< \infty. 
	\end{align*}
\end{assumption}
}

\section{Main Results}
\label{sec:results}

\subsection{Main technical results and training performance}

\textcolor{black}{As a standard preliminary step in the {\it asymptotic} random matrix analysis of the expectation $\EE[Q]$ of the resolvent $Q=(\frac1T\Sigma^\trans \Sigma+\gamma I_T)^{-1}$, a convergence of quadratic forms based on the row vectors of $\Sigma$ is necessary (see e.g., \citep{MAR67,SIL95}). Such results are usually obtained by exploiting the independence (or linear dependence) in the vector entries. This not being the case here, as the entries of the vector $\sigma(X^\trans w)$ are in general not independent, we resort to a concentration of measure approach, as advocated in \citep{ELK09}. The following lemma, stated here in a {\it non-asymptotic} random matrix regime (that is, without necessarily resorting to Assumption~\ref{ass:growth}), and thus of independent interest, provides this concentration result.}
\textcolor{black}{For this lemma, we need first to define the following key matrix}
\begin{align}
	\label{eq:Phi}
	\Phi &= \EE\left[ \sigma(w^\trans X)^\trans \sigma(w^\trans X) \right]
\end{align}
of size $T\times T$, where $w\sim \mathcal N_\varphi(0,I_p)$.

\begin{lemma}[Concentration of quadratic forms]
	\label{lem:concentration_quadform}
	Let \textcolor{black}{Assumptions~\ref{ass:W}--\ref{ass:sigma}} hold. \textcolor{black}{Let also $A \in \RR^{T\times T}$ such that $\| A \| \leq 1$} and, for $X\in\RR^{p\times T}$ and $w\sim \mathcal N_\varphi(0,I_p)$, define the random vector $\sigma\equiv \sigma(w^\trans X)^\trans\in\RR^T$. Then, 
	\begin{align*}
		\textcolor{black}{P\left( \left| \frac1T\sigma^\trans A \sigma - \frac1T\tr \Phi A  \right| > t \right) } & \textcolor{black}{\leq C e^{-\frac{cT}{\|X \|^2 \lambda_{\varphi}^2\lambda_{\sigma}^2} \min\left(\frac{t^2}{t_0^2}, t\right)}}
	\end{align*}
	\textcolor{black}{for $t_0\equiv |\sigma(0)| + \lambda_\varphi\lambda_{\sigma} \|X \| \sqrt{\frac{p}T}$ and $C,c>0$ independent of all other parameters. In particular, under the additional Assumption~\ref{ass:growth},
	\begin{align*}
		P\left( \left| \frac1T\sigma^\trans A \sigma - \frac1T\tr \Phi A  \right| > t \right) \leq C e^{-c n \min(t,t^2)}
	\end{align*}
	for some $C,c>0$.
	}
\end{lemma}

\textcolor{black}{Note that this lemma partially extends concentration of measure results involving quadratic forms, see e.g., \cite[Theorem~1.1]{RUD13}, to non-linear vectors.}

\textcolor{black}{With this result in place, the standard resolvent approaches of random matrix theory apply, providing our main theoretical finding as follows.}
\begin{theorem}[{Asymptotic equivalent for $\EE[Q]$}]
	\label{th:EQ}
	Let Assumptions~\ref{ass:W}--\ref{ass:growth} hold and define $\bar{Q}$ as
	\begin{align*}
		\bar{Q} &\equiv \left( \frac{n}T \frac{\Phi}{1+\delta} + \gamma I_T \right)^{-1}
	\end{align*}
	where $\delta$ is implicitly defined as the unique positive solution to $\delta=\frac1T\tr \Phi\bar{Q}$. Then, for all $\varepsilon>0$, there exists $c>0$ such that
	\begin{align*}
		\left\| \EE[Q] - \bar{Q} \right\| &\leq cn^{-\frac12+\varepsilon}.
	\end{align*}
\end{theorem}

As a corollary of Theorem~\ref{th:EQ} along with a concentration argument on $\frac1T\tr Q$, we have the following result \textcolor{black}{on the spectral measure of $\frac1T\Sigma^\trans\Sigma$, which may be seen as a non-linear extension of \citep{SIL95} for which $\sigma(t)=t$.}
\begin{theorem}[{Limiting spectral measure of $\frac1T\Sigma^\trans\Sigma$}]
	\label{th:lsd}
	Let Assumptions~\ref{ass:W}--\ref{ass:growth} hold and, for $\lambda_1,\ldots,\lambda_T$ the eigenvalues of $\frac1T\Sigma^\trans\Sigma$, define $\mu_n = \frac1T\sum_{i=1}^T {\bm\delta}_{\lambda_i}$. 
	\textcolor{black}{Then, for every bounded continuous function $f$, with probability one
	\begin{align*}
		\int f d\mu_n - \int f d\bar{\mu}_n &\to 0.
	\end{align*}}%
%
%
where $\bar{\mu}_n$ is the measure defined through its Stieltjes transform $m_{\bar{\mu}_n}(z)\equiv \int (t-z)^{-1}d\bar{\mu}_n(t)$ given, for $z\in\{w\in\CC,~\Im[w]>0\}$, by
	\begin{align*}
		m_{\bar{\mu}_n}(z) &= \frac1T\tr \left( \frac{n}T \frac{\Phi}{1+\delta_z} - z I_T \right)^{-1}
	\end{align*}
	with $\delta_z$ the unique solution in $\{w\in\CC,~\Im[w]>0\}$ of
	\begin{align*}
		\delta_z &= \frac1T\tr \Phi \left( \frac{n}T \frac{\Phi}{1+\delta_z} - z I_T \right)^{-1}.
	\end{align*}
\end{theorem}

Note that $\bar \mu_n$ has a well-known form, already met in early random matrix works (e.g., \citep{SIL95}) on sample covariance matrix models. Notably, $\bar\mu_n$ is also the deterministic equivalent of the empirical spectral measure of $\frac1TP^\trans W^\trans WP$ for any deterministic matrix $P\in\RR^{p\times T}$ such that $P^\trans P=\Phi$. As such, to some extent, the results above provide a consistent asymptotic {\it linearization} of $\frac1T\Sigma^\trans\Sigma$.
From standard spiked model arguments (see e.g., \citep{BEN12}), the result $\|\EE[Q]-\bar Q\|\to 0$ further suggests that also the eigenvectors associated to isolated eigenvalues of $\frac1T\Sigma^\trans\Sigma$ (if any) behave similarly to those of $\frac1TP^\trans W^\trans WP$, a remark that has fundamental importance in the neural network performance understanding.  

However, as shall be shown in Section~\ref{sec:Phi}, and contrary to empirical covariance matrix models of the type $P^\trans W^\trans WP$, $\Phi$ explicitly depends on the distribution of $W_{ij}$ (that is, beyond its first two moments). Thus, the aforementioned {\it linearization} of $\frac1T\Sigma^\trans\Sigma$, and subsequently the deterministic equivalent for $\mu_n$, are not universal with respect to the distribution of zero-mean unit variance $W_{ij}$. This is in striking contrast to the many linear random matrix models studied to date which often exhibit such universal behaviors. This property too will have deep consequences in the performance of neural networks as shall be shown through Figure~\ref{fig:poly2} in Section~\ref{sec:discussion} for an example where inappropriate choices for the law of $W$ lead to network failure to fulfill the regression task.

\bigskip

For convenience in the following, letting $\delta$ and $\Phi$ be defined as in Theorem~\ref{th:EQ}, we shall denote
\begin{align}
	\label{eq:Psi}
	\Psi &= \frac{n}T \frac{\Phi}{1+\delta}.
\end{align}

\bigskip

Theorem~\ref{th:EQ} provides the central step in the evaluation of $E_{\rm train}$, for which not only $\EE[Q]$ but also $\EE[Q^2]$ needs be estimated. This last ingredient is provided in the following proposition.
\begin{proposition}[{Asymptotic equivalent for $\EE[QAQ]$}]
	\label{prop:QAQ}
	Let Assumptions~\ref{ass:W}--\ref{ass:growth} hold and $A\in\RR^{T\times T}$ be a symmetric non-negative definite matrix which is either $\Phi$ or a matrix with uniformly bounded operator norm (with respect to $T$). Then, for all $\varepsilon>0$, there exists $c>0$ such that, for all $n$,
	\begin{align*}
		\left\| \EE[QAQ] - \left( \bar{Q}A\bar{Q} + \frac{\frac1n\tr \left(\Psi\bar QA\bar Q\right)}{1-\frac1n\tr \Psi^2\bar Q^2} \bar Q\Psi \bar Q \right) \right\| &\leq cn^{-\frac12+\varepsilon}.
	\end{align*}
\end{proposition}

As an immediate consequence of Proposition~\ref{prop:QAQ}, we have the following result on the training mean-square error of single-layer random neural networks.
\begin{theorem}[Asymptotic training mean-square error]
	\label{th:Etrain}
	Let Assumptions~\ref{ass:W}--\ref{ass:growth} hold and $\bar{Q}$, $\Psi$ be defined as in Theorem~\ref{th:EQ} and \eqref{eq:Psi}. Then, for all $\varepsilon>0$,
	\begin{align*}
		n^{\frac12-\varepsilon} \left(E_{\rm train} - \bar{E}_{\rm train}\right) &\to 0
	\end{align*}
	almost surely, where 
	\begin{align*}
		E_{\rm train} &= \frac1T \left\| Y^\trans - \Sigma^\trans \beta \right\|_F^2 = \frac{\gamma^2}T \tr Y^\trans Y Q^2 \\
		\bar{E}_{\rm train}  &= \frac{\gamma^2}T\tr Y^\trans Y \bar{Q} \left[ \frac{\frac1n\tr \Psi\bar{Q}^2}{1-\frac1n\tr (\Psi\bar{Q})^2} \Psi + I_T \right] \bar{Q}.
	\end{align*}
\end{theorem}

Since $\bar Q$ and $\Phi$ share the same orthogonal eigenvector basis, it appears that $E_{\rm train}$ depends on the alignment between the right singular vectors of $Y$ and the eigenvectors of $\Phi$, with weighting coefficients
\begin{align*}
	\left(\frac{\gamma}{\lambda_i+\gamma}\right)^2 \left( 1 + \lambda_i \frac{\frac1n \sum_{j=1}^T \lambda_j(\lambda_j+\gamma)^{-2}}{1-\frac1n \sum_{j=1}^T \lambda_j^2(\lambda_j+\gamma)^{-2}} \right),~1\leq i\leq T
\end{align*}
where we denoted $\lambda_i=\lambda_i(\Psi)$, $1\leq i\leq T$, the eigenvalues of $\Psi$ (which depend on $\gamma$ through $\lambda_i(\Psi)=\frac{n}{T(1+\delta)}\lambda_i(\Phi)$). If $\liminf_n n/T>1$, it is easily seen that $\delta\to 0$ as $\gamma\to 0$, in which case $E_{\rm train}\to 0$ almost surely. However, in the more interesting case in practice where $\limsup_n n/T<1$, $\delta\to\infty$ as $\gamma\to 0$ and $E_{\rm train}$ consequently does not have a simple limit (see Section~\ref{sec:limiting_cases} for more discussion on this aspect).

Theorem~\ref{th:Etrain} is also reminiscent of applied random matrix works on empirical covariance matrix models, such as \citep{BAI07,KAM09}, then further emphasizing the strong connection between the non-linear matrix $\sigma(WX)$ and its linear counterpart $W\Phi^\frac12$. 

\bigskip

As a side note, observe that, to obtain Theorem~\ref{th:Etrain}, we could have used the fact that $\tr Y^\trans Y Q^2=-\frac{\partial}{\partial \gamma}\tr Y^\trans Y Q$ which, along with some analyticity arguments (for instance when extending the definition of $Q=Q(\gamma)$ to $Q(z)$, $z\in\CC$), would have directly ensured that $\frac{\partial \bar Q}{\partial \gamma}$ is an asymptotic equivalent for $-\EE[Q^2]$, without the need for the explicit derivation of Proposition~\ref{prop:QAQ}. Nonetheless, as shall appear subsequently, Proposition~\ref{prop:QAQ} is also a proxy to the asymptotic analysis of $E_{\rm test}$. Besides, the technical proof of Proposition~\ref{prop:QAQ} quite interestingly showcases the strength of the concentration of measure tools under study here.

\subsection{Testing performance}

As previously mentioned, harnessing the asymptotic testing performance $E_{\rm test}$ seems, to the best of the authors' knowledge, out of current reach with the sole concentration of measure arguments used for the proof of the previous main results. Nonetheless, if not fully effective, these arguments allow for an intuitive derivation of a deterministic equivalent for $E_{\rm test}$, which is strongly supported by simulation results. We provide this result below under the form of a yet unproven claim, a heuristic derivation of which is provided at the end of Section~\ref{sec:proofs}.

\bigskip

To introduce this result, let $\hat{X}=[\hat x_1,\ldots,\hat x_{\hat T}]\in\RR^{p\times \hat T}$ be a set of input data with corresponding output $\hat Y=[\hat y_1,\ldots,\hat y_{\hat T}]\in\RR^{d\times \hat T}$. We also define $\hat \Sigma=\sigma(W\hat X)\in\RR^{p\times\hat T}$. We assume that $\hat X$ and $\hat Y$ satisfy the same growth rate conditions as $X$ and $Y$ in Assumption~\ref{ass:growth}. To introduce our claim, we need to extend the definition of $\Phi$ in \eqref{eq:Phi} and $\Psi$ in \eqref{eq:Psi} to the following notations: for all pair of matrices $(A,B)$ of appropriate dimensions,
\begin{align*}
	\Phi_{AB} &= \EE\left[ \sigma(w^\trans A)^\trans \sigma(w^\trans B) \right] \\
	\Psi_{AB} &= \frac{n}T\frac{\Phi_{AB}}{1+\delta}
\end{align*}
where $w\sim\mathcal N_\varphi(0,I_p)$. In particular, $\Phi=\Phi_{XX}$ and $\Psi=\Psi_{XX}$. 

With these notations in place, we are in position to state our claimed result.

\begin{conj}[{Deterministic equivalent for $E_{\rm test}$}]
	\label{claim:Etest}
	Let Assumptions~\ref{ass:W}--\ref{ass:sigma} hold and $\hat X,\hat Y$ satisfy the same conditions as $X,Y$ in Assumption~\ref{ass:growth}. Then, for all $\varepsilon>0$,
	\begin{align*}
		n^{\frac12-\varepsilon}\left(E_{\rm test} - \bar E_{\rm test}\right) &\to 0
	\end{align*}
	almost surely, where
	\begin{align*}
		E_{\rm test} &= \frac1{\hat T} \left\| \hat Y^\trans - \hat \Sigma^\trans\beta \right\|_F^2 \\
		\bar E_{\rm test} &= \frac1{\hat T} \left\| \hat Y^\trans - \Psi_{X\hat X}^\trans \bar Q Y^\trans \right\|_F^2\nonumber \\
		&+ \frac{\frac1n\tr Y^\trans Y\bar Q\Psi \bar Q}{1-\frac1n\tr (\Psi\bar Q)^2} \left[ \frac1{\hat T}\tr \Psi_{\hat X\hat X} - \frac1{\hat T}\tr (I_T+\gamma \bar Q)(\Psi_{X\hat X}\Psi_{\hat XX}\bar Q)\right].
	\end{align*}
\end{conj}

While not immediate at first sight, one can confirm (using notably the relation $\Psi\bar Q+\gamma\bar Q=I_T$) that, for $(\hat X,\hat Y)=(X,Y)$, $\bar E_{\rm train}=\bar E_{\rm test}$, as expected.

In order to evaluate practically the results of Theorem~\ref{th:Etrain} and Conjecture~\ref{claim:Etest}, it is a first step to be capable of estimating the values of $\Phi_{AB}$ for various $\sigma(\cdot)$ activation functions of practical interest. Such results, which call for completely different mathematical tools (mostly based on integration tricks), are provided in the subsequent section.

\subsection{Evaluation of $\Phi_{AB}$}
\label{sec:Phi}

The evaluation of $\Phi_{AB}=\EE[\sigma(w^\trans A)^\trans\sigma(w^\trans B)]$ for arbitrary matrices $A,B$ naturally boils down to the evaluation of its individual entries and thus to the calculus, for arbitrary vectors $a,b\in\RR^p$, of
\begin{align}
	\label{eq:formula_Phi_ab}
	\textcolor{black}{\Phi_{ab}\equiv} \EE[\sigma( w^\trans a)\sigma(w^\trans b)] &= (2\pi)^{-\frac{p}2} \int \sigma(\varphi(\tilde w)^\trans a)\sigma( \varphi(\tilde w)^\trans b) e^{-\frac12 \textcolor{black}{\|\tilde w\|^2}}d\tilde w.
\end{align}
The evaluation of \eqref{eq:formula_Phi_ab} can be obtained through various integration tricks for a wide family of mappings $\varphi(\cdot)$ and activation functions $\sigma(\cdot)$. The most popular activation functions in neural networks are sigmoid functions, such as $\sigma(t)={\rm erf}(t)\equiv\frac2{\sqrt{\pi}}\int_0^t e^{-u^2}du$, as well as the so-called rectified linear unit (ReLU) defined by $\sigma(t)=\max(t,0)$ which has been recently popularized as a result of its robust behavior in deep neural networks. In physical artificial neural networks implemented using light projections, $\sigma(t)=|t|$ is the preferred choice. Note that all aforementioned functions are Lipschitz continuous and therefore in accordance with Assumption~\ref{ass:sigma}.

Despite their not abiding by the prescription of Assumptions~\ref{ass:W} and \ref{ass:sigma}, we believe that the results of this article could be extended to more general settings, as discussed in Section~\ref{sec:discussion}. In particular, \textcolor{black}{since the key ingredient} in the proof of all our results is that the vector $\sigma(w^\trans X)$ follows a concentration of measure phenomenon, induced by the Gaussianity of $\tilde w$ (if $w=\varphi(\tilde w)$), the Lipschitz character of $\sigma$ and the norm boundedness of $X$, it is likely, although not necessarily simple to prove, that $\sigma(w^\trans X)$ may still concentrate under relaxed assumptions. This is likely the case for more generic vectors $w$ than $\mathcal N_\varphi(0,I_p)$ as well as for a larger class of activation functions, such as polynomial or piece-wise Lipschitz continuous functions. 

In anticipation of these likely generalizations, we provide in Table~\ref{tab:Phi} the values of $\Phi_{ab}$ for $w\sim\mathcal N(0,I_p)$ (i.e., for $\varphi(t)=t$) and for a set of functions $\sigma(\cdot)$ not necessarily satisfying Assumption~\ref{ass:sigma}. Denoting $\Phi\equiv \Phi(\sigma(t))$, it is interesting to remark that, since $\arccos(x)=-\arcsin(x)+\frac\pi2$, $\Phi(\max(t,0))=\Phi(\frac12t)+\Phi(\frac12|t|)$. Also, $[\Phi(\cos(t))+\Phi(\sin(t))]_{a,b}=\exp(-\frac12\|a-b\|^2)$, a result reminiscent of \citep{RAH07}.\footnote{It is in particular not difficult to prove, based on our framework, that, as $n/T\to\infty$, a random neural network composed of $n/2$ neurons with activation function $\sigma(t)=\cos(t)$ and $n/2$ neurons with activation function $\sigma(t)=\sin(t)$ implements a Gaussian difference kernel.} Finally, note that $\Phi( {\rm erf}(\kappa t))\to \Phi( {\rm sign}(t) )$ as $\kappa\to\infty$, inducing that the extension by continuity of ${\rm erf}(\kappa t)$ to ${\rm sign}(t)$ propagates to their associated kernels.

\begin{table}
	\renewcommand{\arraystretch}{1.2}
\centering
\begin{tabular}{cc}
	\toprule
	$\sigma(t)$ & $\Phi_{ab}$ \\ \hline
	\noalign{\smallskip}
	$t$ & $a^\trans b$ \\
	$\max(t,0)$ & $\frac1{2\pi}\|a\| \|b\| \left( \angle(a,b) \acos(-\angle(a,b)) + \sqrt{1-\angle(a,b)^2} \right)$\\
	$|t|$ & $\frac2{\pi}\|a\| \|b\| \left( \angle(a,b) \asin(\angle(a,b)) + \sqrt{1-\angle(a,b)^2} \right)$\\
	${\rm erf}(t)$ & $\frac2{\pi}\asin \left( \frac{2a^\trans b}{\sqrt{ (1+2\|a\|^2)(1+2\|b\|^2) }} \right)$ \\
	$1_{\{t>0\}}$ & $\frac12-\frac1{2\pi} \acos(\angle(a,b))$ \\
	${\rm sign}(t)$ & $\frac2\pi \asin(\angle(a,b))$ \\
	$\cos(t)$ & $\exp(-\frac12(\|a\|^2+\|b\|^2))\cosh(a^\trans b)$ \\
	$\sin(t)$ & $\exp(-\frac12(\|a\|^2+\|b\|^2))\sinh(a^\trans b)$. \\
	\bottomrule
\end{tabular}

\caption{Values of $\Phi_{ab}$ for $w\sim \mathcal N(0,I_p)$, $\angle(a,b)\equiv \frac{a^\trans b}{\|a\|\|b\|}$.}
\label{tab:Phi}
\end{table}

\bigskip

In addition to these results for $w\sim \mathcal N(0,I_p)$, we also evaluated $\Phi_{ab}=\EE[\sigma(w^\trans a)\sigma(w^\trans b)]$ for $\sigma(t)=\zeta_2 t^2+\zeta_1 t+\zeta_0$ and $w\in\RR^p$ a vector of independent and identically distributed entries of zero mean and moments of order $k$ equal to $m_k$ (so $m_1=0$); $w$ is not restricted here to satisfy $w\sim \mathcal N_\varphi(0,I_p)$. In this case, we find
\begin{align}
	\label{eq:Phi_poly2}
	\Phi_{ab} &= \zeta_2^2\left[m_2^2\left(2(a^\trans b)^2+\|a\|^2\|b\|^2\right)+(m_4-3m_2^2)(a^2)^\trans (b^2) \right]+ \zeta_1^2 m_2 a^\trans b \nonumber \\
	&+ \zeta_2\zeta_1 m_3 \left[(a^2)^\trans b+a^\trans(b^2) \right] + \zeta_2\zeta_0 m_2 \left[ \|a\|^2+\|b\|^2 \right] + \zeta_0^2
\end{align}
where we defined $(a^2)\equiv [a_1^2,\ldots,a_p^2]^\trans$. 

It is already interesting to remark that, while classical random matrix models exhibit a well-known universality property --- in the sense that their limiting spectral distribution is independent of the moments (higher than two) of the entries of the involved random matrix, here $W$ ---, for $\sigma(\cdot)$ a polynomial of order two, $\Phi$ and thus $\mu_n$ strongly depend on $\EE[W_{ij}^k]$ for $k=3,4$. We shall see in Section~\ref{sec:discussion} that this remark has troubling consequences. We will notably infer (and confirm via simulations) that the studied neural network may provably fail to fulfill a specific task if the $W_{ij}$ are Bernoulli with zero mean and unit variance but succeed with possibly high performance if the $W_{ij}$ are standard Gaussian (which is explained by the disappearance or not of the term $(a^\trans b)^2$ and $(a^2)^\trans (b^2)$ in \eqref{eq:Phi_poly2} if $m_4=m_2^2$).

\section{Practical Outcomes}
\label{sec:discussion}

We discuss in this section the outcomes of our main results in terms of neural network application. The technical discussions on Theorem~\ref{th:EQ} and Proposition~\ref{prop:QAQ} will be made in the course of their respective proofs in Section~\ref{sec:proofs}.

\subsection{Simulation Results}

We first provide in this section a simulation corroborating the findings of Theorem~\ref{th:Etrain} and suggesting the validity of Conjecture~\ref{claim:Etest}. To this end, we consider the task of classifying the popular MNIST image database \citep{MNIST}, composed of grayscale handwritten digits of size $28\times 28$, with a neural network composed of $n=512$ units and standard Gaussian $W$. We represent here each image as a $p=784$-size vector; $1\,024$ images of sevens and $1\,024$ images of nines were extracted from the database and were evenly split in $512$ training and test images, respectively. The database images were jointly centered and scaled so to fall close to the setting of Assumption~\ref{ass:growth} on $X$ and $\hat X$ (an admissible preprocessing intervention). The columns of the output values $Y$ and $\hat Y$ were taken as unidimensional ($d=1$) with $Y_{1j},\hat Y_{1j}\in\{-1,1\}$ depending on the image class. Figure~\ref{fig:perf} displays the simulated (averaged over $100$ realizations of $W$) versus theoretical values of $E_{\rm train}$ and $E_{\rm test}$ for three choices of Lipschitz continuous functions $\sigma(\cdot)$, as a function of $\gamma$. 

Note that a perfect match between theory and practice is observed, for both $E_{\rm train}$ and $E_{\rm test}$, which is a strong indicator of both the validity of Conjecture~\ref{claim:Etest} and the adequacy of Assumption~\ref{ass:growth} to the MNIST dataset.

\begin{figure}[h!]
  \centering
  \begin{tikzpicture}[font=\footnotesize]
    \renewcommand{\axisdefaulttryminticks}{4} 
    \tikzstyle{every major grid}+=[style=densely dashed] 
    \tikzstyle{every axis y label}+=[yshift=-10pt] 
    \tikzstyle{every axis x label}+=[yshift=5pt]
    \tikzstyle{every axis legend}+=[cells={anchor=west},fill=white,
        at={(0.02,0.98)}, anchor=north west, font=\scriptsize ]
    \begin{loglogaxis}[
      width=.8\linewidth,
      xmin=1e-4,
      ymin=5e-2,
      xmax=1e2,
      ymax=1,
      bar width=1.5pt,
      grid=major,
      ymajorgrids=false,
      scaled ticks=true,
      xlabel={$\gamma$},
      ylabel={MSE}
      ]
      \addplot[black,smooth,line width=0.5pt] coordinates{
	      (100,100)
      };
      \addplot[black,densely dashed,smooth,line width=0.5pt] coordinates{
	      (100,100)
      };
      \addplot[black,only marks,mark=o,line width=0.5pt] coordinates{
	      (100,100)
      };
      \addplot[black,only marks,mark=x,line width=0.5pt] coordinates{
	      (100,100)
      };
      \addplot[blue,smooth,line width=0.5pt] coordinates{
	      (0.000100,0.060032)(0.000178,0.060037)(0.000316,0.060051)(0.000562,0.060092)(0.001000,0.060210)(0.001778,0.060520)(0.003162,0.061264)(0.005623,0.062850)(0.010000,0.065838)(0.017783,0.070852)(0.031623,0.078496)(0.056234,0.089318)(0.100000,0.103786)(0.177828,0.122225)(0.316228,0.144733)(0.562341,0.171222)(1.000000,0.201947)(1.778279,0.238929)(3.162278,0.287718)(5.623413,0.357078)(10.000000,0.453269)(17.782794,0.570838)(31.622777,0.691251)(56.234133,0.794663)(100.000000,0.871479)	      
      };
      \addplot[blue,densely dashed,smooth,line width=0.5pt] coordinates{
	      (0.000100,0.369989)(0.000178,0.368107)(0.000316,0.364890)(0.000562,0.359540)(0.001000,0.351041)(0.001778,0.338439)(0.003162,0.321477)(0.005623,0.301251)(0.010000,0.280183)(0.017783,0.261069)(0.031623,0.246054)(0.056234,0.236308)(0.100000,0.232295)(0.177828,0.234146)(0.316228,0.241874)(0.562341,0.255459)(1.000000,0.275160)(1.778279,0.302571)(3.162278,0.342330)(5.623413,0.402283)(10.000000,0.488500)(17.782794,0.596315)(31.622777,0.708353)(56.234133,0.805456)(100.000000,0.877995)
      };
      \addplot[blue,only marks,mark=o,line width=0.5pt] coordinates{
	      (0.000100,0.060169)(0.000178,0.059654)(0.000316,0.059800)(0.000562,0.060305)(0.001000,0.060212)(0.001778,0.060921)(0.003162,0.061980)(0.005623,0.063356)(0.010000,0.065644)(0.017783,0.070753)(0.031623,0.079214)(0.056234,0.090096)(0.100000,0.105253)(0.177828,0.122725)(0.316228,0.145592)(0.562341,0.171732)(1.000000,0.202450)(1.778279,0.240647)(3.162278,0.288083)(5.623413,0.357426)(10.000000,0.454285)(17.782794,0.570757)(31.622777,0.691082)(56.234133,0.793788)(100.000000,0.872311)
      };
      \addplot[blue,only marks,mark=x,line width=0.5pt] coordinates{
	      (0.000100,0.371541)(0.000178,0.371174)(0.000316,0.364297)(0.000562,0.359050)(0.001000,0.349780)(0.001778,0.338103)(0.003162,0.321833)(0.005623,0.299803)(0.010000,0.278571)(0.017783,0.259929)(0.031623,0.245208)(0.056234,0.238587)(0.100000,0.232204)(0.177828,0.235324)(0.316228,0.242129)(0.562341,0.255795)(1.000000,0.275866)(1.778279,0.304368)(3.162278,0.342947)(5.623413,0.401261)(10.000000,0.489728)(17.782794,0.596495)(31.622777,0.708384)(56.234133,0.805123)(100.000000,0.879209)
      };
      \addplot[green!60!black,smooth,line width=0.5pt] coordinates{
	      (0.000100,0.074310)(0.000178,0.074319)(0.000316,0.074346)(0.000562,0.074426)(0.001000,0.074645)(0.001778,0.075201)(0.003162,0.076470)(0.005623,0.079030)(0.010000,0.083576)(0.017783,0.090766)(0.031623,0.101061)(0.056234,0.114591)(0.100000,0.131052)(0.177828,0.149735)(0.316228,0.169772)(0.562341,0.190553)(1.000000,0.212193)(1.778279,0.236175)(3.162278,0.266308)(5.623413,0.309628)(10.000000,0.375483)(17.782794,0.469683)(31.622777,0.585529)(56.234133,0.703494)(100.000000,0.803855)
      };
      \addplot[green!60!black,densely dashed,smooth,line width=0.5pt] coordinates{
	      (0.000100,0.440652)(0.000178,0.437896)(0.000316,0.433230)(0.000562,0.425596)(0.001000,0.413766)(0.001778,0.396844)(0.003162,0.375087)(0.005623,0.350429)(0.010000,0.325926)(0.017783,0.304414)(0.031623,0.287559)(0.056234,0.275806)(0.100000,0.268843)(0.177828,0.266139)(0.316228,0.267331)(0.562341,0.272381)(1.000000,0.281629)(1.778279,0.296124)(3.162278,0.318617)(5.623413,0.354854)(10.000000,0.413192)(17.782794,0.499231)(31.622777,0.606965)(56.234133,0.717905)(100.000000,0.812953)
      };
      \addplot[green!60!black,only marks,mark=o,line width=0.5pt] coordinates{
	      (0.000100,0.074604)(0.000178,0.074787)(0.000316,0.074943)(0.000562,0.074212)(0.001000,0.074328)(0.001778,0.074718)(0.003162,0.076553)(0.005623,0.079204)(0.010000,0.083747)(0.017783,0.091309)(0.031623,0.101327)(0.056234,0.115821)(0.100000,0.131169)(0.177828,0.150731)(0.316228,0.169360)(0.562341,0.190909)(1.000000,0.212579)(1.778279,0.236194)(3.162278,0.266393)(5.623413,0.309525)(10.000000,0.374092)(17.782794,0.468339)(31.622777,0.586155)(56.234133,0.703854)(100.000000,0.802731)
      };
      \addplot[green!60!black,only marks,mark=x,line width=0.5pt] coordinates{
	      (0.000100,0.434155)(0.000178,0.441159)(0.000316,0.432215)(0.000562,0.426099)(0.001000,0.411614)(0.001778,0.395781)(0.003162,0.376977)(0.005623,0.349195)(0.010000,0.325091)(0.017783,0.304871)(0.031623,0.286733)(0.056234,0.277579)(0.100000,0.268198)(0.177828,0.265720)(0.316228,0.266617)(0.562341,0.272183)(1.000000,0.281899)(1.778279,0.296846)(3.162278,0.318317)(5.623413,0.353809)(10.000000,0.411592)(17.782794,0.497774)(31.622777,0.606980)(56.234133,0.717485)(100.000000,0.811980)
      };
      \addplot[red,smooth,line width=0.5pt] coordinates{
	      (0.000100,0.091075)(0.000178,0.093112)(0.000316,0.095535)(0.000562,0.098376)(0.001000,0.101672)(0.001778,0.105475)(0.003162,0.109864)(0.005623,0.114953)(0.010000,0.120889)(0.017783,0.127834)(0.031623,0.135932)(0.056234,0.145272)(0.100000,0.155848)(0.177828,0.167557)(0.316228,0.180262)(0.562341,0.193894)(1.000000,0.208523)(1.778279,0.224475)(3.162278,0.242766)(5.623413,0.266077)(10.000000,0.300001)(17.782794,0.353233)(31.622777,0.433903)(56.234133,0.541025)(100.000000,0.659354)
      };
      \addplot[red,smooth,densely dashed,line width=0.5pt] coordinates{
	      (0.000100,0.589942)(0.000178,0.554335)(0.000316,0.522265)(0.000562,0.492952)(0.001000,0.465919)(0.001778,0.440920)(0.003162,0.417855)(0.005623,0.396675)(0.010000,0.377307)(0.017783,0.359647)(0.031623,0.343607)(0.056234,0.329199)(0.100000,0.316591)(0.177828,0.306125)(0.316228,0.298314)(0.562341,0.293788)(1.000000,0.293182)(1.778279,0.297025)(3.162278,0.305964)(5.623413,0.321683)(10.000000,0.348461)(17.782794,0.393936)(31.622777,0.465999)(56.234133,0.564414)(100.000000,0.675108)
      };
      \addplot[red,only marks,mark=o,line width=0.5pt] coordinates{
	      (0.000100,0.091099)(0.000178,0.093085)(0.000316,0.095631)(0.000562,0.098558)(0.001000,0.101701)(0.001778,0.105571)(0.003162,0.110049)(0.005623,0.115043)(0.010000,0.121013)(0.017783,0.127955)(0.031623,0.136139)(0.056234,0.145400)(0.100000,0.155874)(0.177828,0.167709)(0.316228,0.180108)(0.562341,0.194161)(1.000000,0.208499)(1.778279,0.224535)(3.162278,0.243232)(5.623413,0.266925)(10.000000,0.300952)(17.782794,0.355041)(31.622777,0.434164)(56.234133,0.539804)(100.000000,0.659702)
      };
      \addplot[red,only marks,mark=x,line width=0.5pt] coordinates{
	      (0.000100,0.592670)(0.000178,0.553248)(0.000316,0.522034)(0.000562,0.489985)(0.001000,0.463881)(0.001778,0.438415)(0.003162,0.415474)(0.005623,0.396566)(0.010000,0.377298)(0.017783,0.359021)(0.031623,0.342439)(0.056234,0.328641)(0.100000,0.317819)(0.177828,0.305806)(0.316228,0.298247)(0.562341,0.294082)(1.000000,0.293358)(1.778279,0.297068)(3.162278,0.305307)(5.623413,0.322055)(10.000000,0.349827)(17.782794,0.394786)(31.622777,0.466276)(56.234133,0.563912)(100.000000,0.675019)
      };
      \addplot[black,smooth,line width=0.5pt] coordinates{
	      (0.000100,0.118994)(0.000178,0.118995)(0.000316,0.118999)(0.000562,0.119009)(0.001000,0.119042)(0.001778,0.119136)(0.003162,0.119403)(0.005623,0.120100)(0.010000,0.121759)(0.017783,0.125275)(0.031623,0.131881)(0.056234,0.143027)(0.100000,0.160327)(0.177828,0.185704)(0.316228,0.221697)(0.562341,0.271730)(1.000000,0.339752)(1.778279,0.428248)(3.162278,0.534452)(5.623413,0.647826)(10.000000,0.753218)(17.782794,0.838557)(31.622777,0.899920)(56.234133,0.940250)(100.000000,0.965180)
      };
      \addplot[black,smooth,densely dashed,line width=0.5pt] coordinates{
	      (0.000100,0.660724)(0.000178,0.659433)(0.000316,0.657168)(0.000562,0.653236)(0.001000,0.646531)(0.001778,0.635433)(0.003162,0.617930)(0.005623,0.592268)(0.010000,0.558284)(0.017783,0.518675)(0.031623,0.478766)(0.056234,0.444558)(0.100000,0.420852)(0.177828,0.410866)(0.316228,0.416966)(0.562341,0.441401)(1.000000,0.486219)(1.778279,0.551764)(3.162278,0.633990)(5.623413,0.723019)(10.000000,0.806020)(17.782794,0.873195)(31.622777,0.921446)(56.234133,0.953129)(100.000000,0.972701)
      };
      \addplot[black,only marks,mark=o,line width=0.5pt] coordinates{
	      (0.000100,0.121227)(0.000178,0.120282)(0.000316,0.120928)(0.000562,0.121864)(0.001000,0.119653)(0.001778,0.119860)(0.003162,0.120877)(0.005623,0.120728)(0.010000,0.122730)(0.017783,0.126662)(0.031623,0.132465)(0.056234,0.145740)(0.100000,0.163505)(0.177828,0.188107)(0.316228,0.224932)(0.562341,0.274258)(1.000000,0.339900)(1.778279,0.432975)(3.162278,0.538123)(5.623413,0.649557)(10.000000,0.755071)(17.782794,0.839693)(31.622777,0.899472)(56.234133,0.940602)(100.000000,0.965418)
      };
      \addplot[black,only marks,mark=x,line width=0.5pt] coordinates{
	      (0.000100,0.662443)(0.000178,0.666947)(0.000316,0.652540)(0.000562,0.650891)(0.001000,0.656302)(0.001778,0.636791)(0.003162,0.619482)(0.005623,0.595861)(0.010000,0.568546)(0.017783,0.523688)(0.031623,0.480737)(0.056234,0.445712)(0.100000,0.421055)(0.177828,0.414472)(0.316228,0.421412)(0.562341,0.444476)(1.000000,0.487235)(1.778279,0.554794)(3.162278,0.638481)(5.623413,0.725179)(10.000000,0.807130)(17.782794,0.873597)(31.622777,0.920899)(56.234133,0.953378)(100.000000,0.972782)
      };
      \draw[->,thick] (axis cs:.3,.08) -- (axis cs:.04,.08) node [right,pos=0,font=\footnotesize] { \BLUE $\sigma(t)=\max(t,0)$ };
      \draw[->,thick] (axis cs:.3,.105) -- (axis cs:.04,.105) node [right,pos=0,font=\footnotesize] { \GREEN $\sigma(t)={\rm erf(t)}$ };
      \draw[->,thick] (axis cs:.0004,.15) -- (axis cs:.0004,.095) node [above,pos=0,font=\footnotesize] { \RED $\sigma(t)=t$ };
      \draw[->,thick] (axis cs:.007,.15) -- (axis cs:.007,.12) node [above,pos=0,font=\footnotesize] { $\sigma(t)=|t|$ };
      \legend{ { $\bar E_{\rm train}$ }, { $\bar E_{\rm test}$ }, { $E_{\rm train}$ }, { $E_{\rm test}$ } }
    \end{loglogaxis}
  \end{tikzpicture}
  \caption{ Neural network performance for Lipschitz continuous $\sigma(\cdot)$, $W_{ij}\sim\mathcal N(0,1)$, as a function of $\gamma$, for 2-class MNIST data (sevens, nines), $n=512$, $T=\hat T=1024$, $p=784$.}
  \label{fig:perf}
\end{figure}
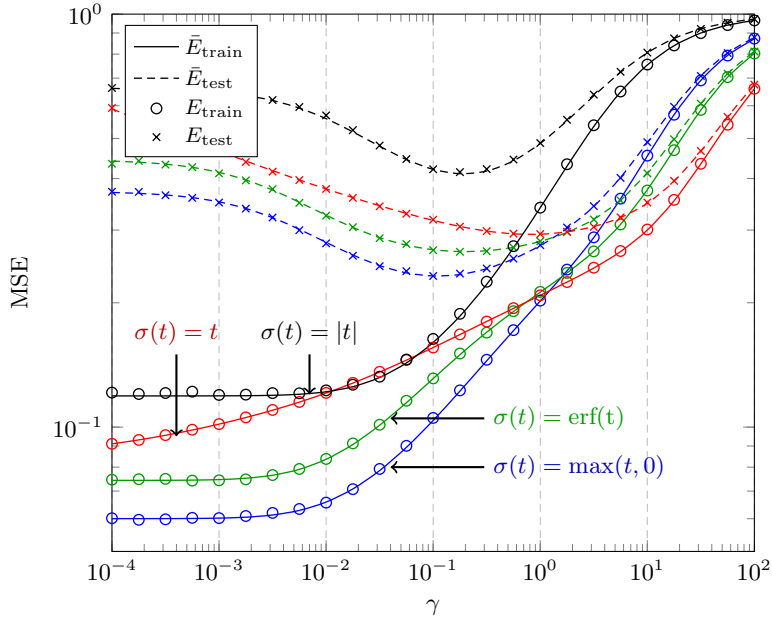

\bigskip

We subsequently provide in Figure~\ref{fig:perf2} the comparison between theoretical formulas and practical simulations for a set of functions $\sigma(\cdot)$ which do not satisfy Assumption~\ref{ass:sigma}, i.e., either discontinuous or non-Lipschitz maps. The closeness between both sets of curves is again remarkably good, although to a lesser extent than for the Lipschitz continuous functions of Figure~\ref{fig:perf}. Also, the achieved performances are generally worse than those observed in Figure~\ref{fig:perf}.

\bigskip

It should be noted that the performance estimates provided by Theorem~\ref{th:Etrain} and Conjecture~\ref{claim:Etest} can be efficiently implemented at low computational cost in practice. Indeed, by diagonalizing $\Phi$ (which is a marginal cost independent of $\gamma$), $\bar E_{\rm train}$ can be computed for all $\gamma$ through mere vector operations; similarly $\bar E_{\rm test}$ is obtained by the marginal cost of a basis change of $\Phi_{\hat XX}$ and the matrix product $\Phi_{X\hat X}\Phi_{\hat XX}$, all remaining operations being accessible through vector operations. 
As a consequence, the simulation durations to generate the aforementioned theoretical curves using the linked \href{https://github.com/Zhenyu-LIAO/RMT4ELM}{Python script} were found to be $100$ to $500$ times faster than to generate the simulated network performances. Beyond their theoretical interest, the provided formulas therefore allow for an efficient offline tuning of the network hyperparameters, notably the choice of an appropriate value for the ridge-regression parameter $\gamma$.

\begin{figure}[h!]
  \centering
  \begin{tikzpicture}[font=\footnotesize]
    \renewcommand{\axisdefaulttryminticks}{4} 
    \tikzstyle{every major grid}+=[style=densely dashed] 
    \tikzstyle{every axis y label}+=[yshift=-10pt] 
    \tikzstyle{every axis x label}+=[yshift=5pt]
    \tikzstyle{every axis legend}+=[cells={anchor=west},fill=white,
        at={(0.02,0.98)}, anchor=north west, font=\scriptsize ]
    \begin{loglogaxis}[
      width=.8\linewidth,
      xmin=1e-4,
      ymin=9e-2,
      xmax=1e2,
      ymax=1,
      bar width=1.5pt,
      grid=major,
      ymajorgrids=false,
      scaled ticks=true,
      xlabel={$\gamma$},
      ylabel={MSE}
      ]
      \addplot[black,smooth,line width=0.5pt] coordinates{
	      (100,100)
      };
      \addplot[black,densely dashed,smooth,line width=0.5pt] coordinates{
	      (100,100)
      };
      \addplot[black,only marks,mark=o,line width=0.5pt] coordinates{
	      (100,100)
      };
      \addplot[black,only marks,mark=x,line width=0.5pt] coordinates{
	      (100,100)
      };
      \addplot[blue,smooth,line width=0.5pt] coordinates{
	      (0.000100,0.101631)(0.000178,0.101631)(0.000316,0.101632)(0.000562,0.101633)(0.001000,0.101636)(0.001778,0.101648)(0.003162,0.101682)(0.005623,0.101782)(0.010000,0.102059)(0.017783,0.102768)(0.031623,0.104403)(0.056234,0.107742)(0.100000,0.113740)(0.177828,0.123312)(0.316228,0.137084)(0.562341,0.155186)(1.000000,0.177208)(1.778279,0.202623)(3.162278,0.231905)(5.623413,0.268115)(10.000000,0.318033)(17.782794,0.390671)(31.622777,0.490380)(56.234133,0.608227)(100.000000,0.723980)
      };
      \addplot[blue,densely dashed,smooth,line width=0.5pt] coordinates{
	      (0.000100,0.481805)(0.000178,0.481469)(0.000316,0.480873)(0.000562,0.479823)(0.001000,0.477986)(0.001778,0.474805)(0.003162,0.469408)(0.005623,0.460552)(0.010000,0.446765)(0.017783,0.426921)(0.031623,0.401243)(0.056234,0.372033)(0.100000,0.343155)(0.177828,0.318452)(0.316228,0.300486)(0.562341,0.290339)(1.000000,0.288079)(1.778279,0.293456)(3.162278,0.306858)(5.623413,0.330612)(10.000000,0.370212)(17.782794,0.433324)(31.622777,0.523583)(56.234133,0.632350)(100.000000,0.740274)
      };
      \addplot[blue,only marks,mark=o,line width=0.5pt] coordinates{
	      (0.000100,0.101380)(0.000178,0.101823)(0.000316,0.099922)(0.000562,0.100597)(0.001000,0.100515)(0.001778,0.100991)(0.003162,0.101582)(0.005623,0.099692)(0.010000,0.101135)(0.017783,0.101268)(0.031623,0.103566)(0.056234,0.106612)(0.100000,0.112760)(0.177828,0.122626)(0.316228,0.137018)(0.562341,0.154291)(1.000000,0.177041)(1.778279,0.201662)(3.162278,0.231028)(5.623413,0.267380)(10.000000,0.318315)(17.782794,0.389261)(31.622777,0.490441)(56.234133,0.607919)(100.000000,0.721985)
      };
      \addplot[blue,only marks,mark=x,line width=0.5pt] coordinates{
	      (0.000100,0.478121)(0.000178,0.475364)(0.000316,0.478287)(0.000562,0.482397)(0.001000,0.476415)(0.001778,0.473307)(0.003162,0.470096)(0.005623,0.460438)(0.010000,0.447509)(0.017783,0.423552)(0.031623,0.397944)(0.056234,0.370366)(0.100000,0.342543)(0.177828,0.314459)(0.316228,0.300897)(0.562341,0.290034)(1.000000,0.288515)(1.778279,0.293273)(3.162278,0.305632)(5.623413,0.329891)(10.000000,0.369766)(17.782794,0.432267)(31.622777,0.523854)(56.234133,0.632386)(100.000000,0.738391)
      };
      \addplot[green!60!black,smooth,line width=0.5pt] coordinates{
	      (0.000100,0.101701)(0.000178,0.101702)(0.000316,0.101708)(0.000562,0.101726)(0.001000,0.101779)(0.001778,0.101932)(0.003162,0.102343)(0.005623,0.103353)(0.010000,0.105575)(0.017783,0.109880)(0.031623,0.117238)(0.056234,0.128474)(0.100000,0.144032)(0.177828,0.163809)(0.316228,0.187257)(0.562341,0.214087)(1.000000,0.245645)(1.778279,0.286466)(3.162278,0.344694)(5.623413,0.428571)(10.000000,0.537696)(17.782794,0.657339)(31.622777,0.766400)(56.234133,0.851145)(100.000000,0.909586)
      };
      \addplot[green!60!black,densely dashed,smooth,line width=0.5pt] coordinates{
	      (0.000100,0.480760)(0.000178,0.479440)(0.000316,0.477137)(0.000562,0.473178)(0.001000,0.466539)(0.001778,0.455846)(0.003162,0.439667)(0.005623,0.417297)(0.010000,0.389769)(0.017783,0.360170)(0.031623,0.332563)(0.056234,0.310365)(0.100000,0.295510)(0.177828,0.288591)(0.316228,0.289452)(0.562341,0.297971)(1.000000,0.315159)(1.778279,0.344555)(3.162278,0.392917)(5.623413,0.467323)(10.000000,0.567064)(17.782794,0.678037)(31.622777,0.780004)(56.234133,0.859602)(100.000000,0.914642)
      };
      \addplot[green!60!black,only marks,mark=o,line width=0.5pt] coordinates{
	      (0.000100,0.100741)(0.000178,0.101310)(0.000316,0.100981)(0.000562,0.100913)(0.001000,0.101062)(0.001778,0.101863)(0.003162,0.101262)(0.005623,0.102705)(0.010000,0.104395)(0.017783,0.108961)(0.031623,0.117437)(0.056234,0.127768)(0.100000,0.143755)(0.177828,0.163859)(0.316228,0.187180)(0.562341,0.212634)(1.000000,0.245899)(1.778279,0.286137)(3.162278,0.345357)(5.623413,0.431214)(10.000000,0.535742)(17.782794,0.656330)(31.622777,0.767266)(56.234133,0.850659)(100.000000,0.910035)
      };
      \addplot[green!60!black,only marks,mark=x,line width=0.5pt] coordinates{
	      (0.000100,0.477432)(0.000178,0.478196)(0.000316,0.481604)(0.000562,0.472204)(0.001000,0.467143)(0.001778,0.456010)(0.003162,0.444912)(0.005623,0.413607)(0.010000,0.386393)(0.017783,0.359360)(0.031623,0.329923)(0.056234,0.311097)(0.100000,0.295346)(0.177828,0.287666)(0.316228,0.287885)(0.562341,0.296445)(1.000000,0.315249)(1.778279,0.343831)(3.162278,0.392565)(5.623413,0.469413)(10.000000,0.564588)(17.782794,0.676932)(31.622777,0.780833)(56.234133,0.859157)(100.000000,0.915229)
      };
      \addplot[red,smooth,line width=0.5pt] coordinates{
	      (0.000100,0.109390)(0.000178,0.109391)(0.000316,0.109394)(0.000562,0.109404)(0.001000,0.109434)(0.001778,0.109523)(0.003162,0.109773)(0.005623,0.110425)(0.010000,0.111966)(0.017783,0.115211)(0.031623,0.121265)(0.056234,0.131404)(0.100000,0.147025)(0.177828,0.169765)(0.316228,0.201729)(0.562341,0.245656)(1.000000,0.304683)(1.778279,0.381201)(3.162278,0.474594)(5.623413,0.579125)(10.000000,0.684455)(17.782794,0.779147)(31.622777,0.855003)(56.234133,0.909665)(100.000000,0.945820)
      };
      \addplot[red,smooth,densely dashed,line width=0.5pt] coordinates{
	      (0.000100,0.730571)(0.000178,0.729008)(0.000316,0.726265)(0.000562,0.721505)(0.001000,0.713390)(0.001778,0.699965)(0.003162,0.678800)(0.005623,0.647772)(0.010000,0.606627)(0.017783,0.558420)(0.031623,0.509157)(0.056234,0.465453)(0.100000,0.432356)(0.177828,0.412942)(0.316228,0.409171)(0.562341,0.422792)(1.000000,0.455524)(1.778279,0.508155)(3.162278,0.578741)(5.623413,0.661142)(10.000000,0.745685)(17.782794,0.822170)(31.622777,0.883465)(56.234133,0.927547)(100.000000,0.956630)
      };
      \addplot[red,only marks,mark=o,line width=0.5pt] coordinates{
	      (0.000100,0.112432)(0.000178,0.112563)(0.000316,0.111459)(0.000562,0.112554)(0.001000,0.114305)(0.001778,0.112503)(0.003162,0.113161)(0.005623,0.115196)(0.010000,0.117623)(0.017783,0.120830)(0.031623,0.127196)(0.056234,0.138830)(0.100000,0.156935)(0.177828,0.182148)(0.316228,0.215164)(0.562341,0.260786)(1.000000,0.320765)(1.778279,0.392710)(3.162278,0.488268)(5.623413,0.587657)(10.000000,0.688738)(17.782794,0.782422)(31.622777,0.858345)(56.234133,0.909954)(100.000000,0.946666)
      };
      \addplot[red,only marks,mark=x,line width=0.5pt] coordinates{
	      (0.000100,0.759612)(0.000178,0.767696)(0.000316,0.743818)(0.000562,0.741448)(0.001000,0.740763)(0.001778,0.707827)(0.003162,0.693511)(0.005623,0.664490)(0.010000,0.624933)(0.017783,0.562682)(0.031623,0.509503)(0.056234,0.478548)(0.100000,0.441849)(0.177828,0.427179)(0.316228,0.421262)(0.562341,0.436005)(1.000000,0.469829)(1.778279,0.516654)(3.162278,0.591679)(5.623413,0.670545)(10.000000,0.751154)(17.782794,0.824556)(31.622777,0.886383)(56.234133,0.927872)(100.000000,0.957088)
      };
      \draw[->,thick] (axis cs:2,.12) -- (axis cs:.2,.12) node [right,pos=0,font=\footnotesize] { \BLUE $\sigma(t)={\rm sign}(t)$ };
      \draw[->,thick] (axis cs:2,.15) -- (axis cs:.12,.15) node [right,pos=0,font=\footnotesize] { \GREEN $\sigma(t)=1_{\{t>0\}}$ };
      \draw[->,thick] (axis cs:2,.2) -- (axis cs:.3,.2) node [right,pos=0,font=\footnotesize] { \RED $\sigma(t)=1-\frac12t^2$ };
      \legend{ { $\bar E_{\rm train}$ }, { $\bar E_{\rm test}$ }, { $E_{\rm train}$ }, { $E_{\rm test}$ } }
    \end{loglogaxis}
  \end{tikzpicture}
  \caption{Neural network performance for $\sigma(\cdot)$ either discontinuous or non Lipschitz, $W_{ij}\sim\mathcal N(0,1)$, as a function of $\gamma$, for 2-class MNIST data (sevens, nines), $n=512$, $T=\hat T=1024$, $p=784$.}
  \label{fig:perf2}
\end{figure}
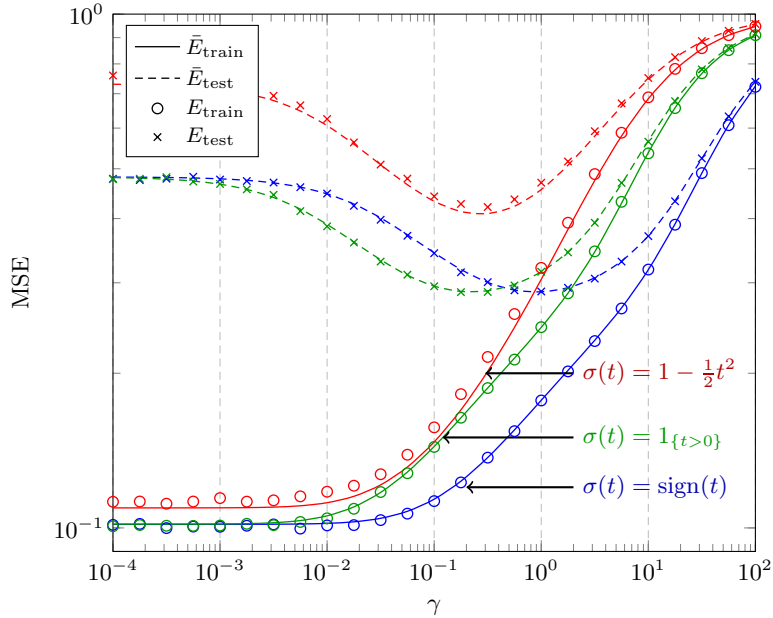

\subsection{The underlying kernel}

Theorem~\ref{th:EQ} and the subsequent theoretical findings importantly reveal that the neural network performances are directly related to the Gram matrix $\Phi$, which acts as a deterministic kernel on the dataset $X$. This is in fact a well-known result found e.g., in 
\citep{WIL98} where it is shown that, as $n\to\infty$ alone, the neural network behaves as a mere kernel operator \textcolor{black}{(this observation is retrieved here in the subsequent Section~\ref{sec:limiting_cases})}. This remark was then put at an advantage in \citep{RAH07} and subsequent works, where \textcolor{black}{random feature maps of the type $x\mapsto \sigma(Wx)$ are proposed as a computationally efficient proxy to evaluate kernels $(x,y)\mapsto \Phi(x,y)$}. 

As discussed previously, the formulas for $\bar E_{\rm train}$ and $\bar E_{\rm test}$ suggest that good performances are achieved if the dominant eigenvectors of $\Phi$ show a good alignment to $Y$ (and similarly for $\Phi_{X\hat X}$ and $\hat Y$). This naturally drives us to finding a priori simple regression tasks where ill-choices of $\Phi$ may annihilate the neural network performance. Following recent works on the asymptotic performance analysis of kernel methods for Gaussian mixture models \citep{COU16,LIA17b,MAI17} and \citep{COU17}, we describe here such a task. 

Let $x_1,\ldots,x_{T/2}\sim \mathcal N(0,\frac1pC_1)$ and $x_{T/2+1},\ldots,x_T\sim \mathcal N(0,\frac1pC_2)$ where $C_1$ and $C_2$ are such that $\tr C_1=\tr C_2$, $\|C_1\|,\|C_2\|$ are bounded, and $\tr(C_1-C_2)^2=O(p)$. Accordingly, $y_1,\ldots,y_{T/2+1}=-1$ and $y_{T/2+1},\ldots,y_T=1$. It is proved in the aforementioned articles that, under these conditions, it is theoretically possible, in the large $p,T$ limit, to classify the data using a kernel least-square support vector machine (that is, with a training dataset) or with a kernel spectral clustering method (that is, in a completely unsupervised manner) with a non-trivial limiting error probability (i.e., neither zero nor one). This scenario has the interesting feature that $x_i^\trans x_j\to 0$ almost surely for all $i\neq j$ while $\|x_i\|^2-\frac1p\tr (\frac12C_1+\frac12C_2)\to 0$, almost surely, irrespective of the class of $x_i$, thereby allowing for a Taylor expansion of the non-linear kernels as early proposed in \citep{ELK10}. 

Transposed to our present setting, the aforementioned Taylor expansion allows for a consistent approximation $\tilde{\Phi}$ of $\Phi$ by an {\it information-plus-noise} (spiked) random matrix model (see e.g., \citep{LOU10b,BEN12}). In the present Gaussian mixture context, it is shown in \citep{COU16} that data classification is (asymptotically at least) only possible if $\tilde\Phi_{ij}$ explicitly contains the quadratic term $(x_i^\trans x_j)^2$ (or combinations of $(x_i^2)^\trans x_j$, $(x_j^2)^\trans x_i$, and $(x_i^2)^\trans (x_j^2)$). \textcolor{black}{In particular, letting $a,b\sim \mathcal N(0,C_i)$ with $i=1,2$, it is easily seen from Table~\ref{tab:Phi} that only $\max(t,0)$, $|t|$, and $\cos(t)$ can realize the task}. Indeed, we have the following Taylor expansions around $x=0$:
\begin{align*}
\asin(x) &= x+O(x^3) \\
\sinh(x) &= x+O(x^3) \\
\acos(x) &= \frac{\pi}2-x + O(x^3) \\
\cosh(x) &= 1+\frac{x^2}2 + O(x^3) \\
x\acos(-x)+\sqrt{1-x^2} &= 1+\frac{\pi x}2+\frac{x^2}2+O(x^3) \\
x\asin(x)+\sqrt{1-x^2} &= 1+\frac{x^2}2+O(x^3)
\end{align*}
where only the last three functions \textcolor{black}{(only found in the expression of $\Phi_{ab}$ corresponding to $\sigma(t)=\max(t,0)$, $|t|$, or $\cos(t)$)} exhibit a quadratic term.

More surprisingly maybe, recalling now Equation~\eqref{eq:Phi_poly2} which considers non-necessarily Gaussian $W_{ij}$ with moments $m_k$ of order $k$, a more refined analysis shows that the aforementioned Gaussian mixture classification task will fail if $m_3=0$ and $m_4=m_2^2$, so for instance for $W_{ij}\in\{-1,1\}$ Bernoulli with parameter $\frac12$. The performance comparison of this scenario is shown in the top part of Figure~\ref{fig:poly2} for $\sigma(t)=-\frac12t^2+1$ and $C_1=\diag(I_{p/2},4I_{p/2})$, $C_2=\diag(4I_{p/2},I_{p/2})$, for $W_{ij}\sim \mathcal N(0,1)$ and $W_{ij}\sim {\rm Bern}$ (that is, Bernoulli $\{(-1,\frac12),(1,\frac12)\}$). The choice of $\sigma(t)=\zeta_2 t^2+\zeta_1t+\zeta_0$ with $\zeta_1=0$ is motivated by \citep{COU16,COU17} where it is shown, in a somewhat different setting, that this choice is optimal for class recovery. Note that, while the test performances are overall rather weak in this setting, for $W_{ij}\sim \mathcal N(0,1)$, $E_{\rm test}$ drops below one (the amplitude of the $\hat Y_{ij}$)\textcolor{black}{, thereby indicating that non-trivial classification is performed}. This is not so for the Bernoulli $W_{ij}\sim {\rm Bern}$ case where $E_{\rm test}$ is systematically greater than $|\hat Y_{ij}|\textcolor{black}{=1}$. This is theoretically explained by the fact that, from Equation~\eqref{eq:Phi_poly2}, $\Phi_{ij}$ contains structural information about the data classes through the term $2m_2^2(x_i^\trans x_j)^2+(m_4-3m_2^2)(x_i^2)^\trans(x_j^2)$ which induces an information-plus-noise model for $\Phi$ as long as $2m_2^2+(m_4-3m_2^2)\neq 0$, i.e., $m_4\neq m_2^2$ (see \citep{COU16} for details). This is visually seen in the bottom part of Figure~\ref{fig:poly2} where the Gaussian scenario presents an isolated eigenvalue for $\Phi$ with corresponding structured eigenvector, which is not the case of the Bernoulli scenario.
To complete this discussion, it appears relevant in the present setting to choose $W_{ij}$ in such a way that $m_4-m_2^2$ is far from zero, thus suggesting the interest of heavy-tailed distributions. To confirm this prediction, Figure~\ref{fig:poly2} additionally displays the performance achieved and the spectrum of $\Phi$ observed for $W_{ij}\sim {\rm Stud}$, that is, following a Student-t distribution with degree of freedom $\nu=7$ normalized to unit variance (in this case $m_2=1$ and $m_4=5$). Figure~\ref{fig:poly2} confirms the large superiority of this choice over the Gaussian case (note nonetheless the slight inaccuracy of our theoretical formulas in this case, which is likely due to too small values of $p,n,T$ to accommodate $W_{ij}$ with higher order moments, an observation which is confirmed in simulations when letting $\nu$ be even smaller).

\begin{figure}[h!]
  \centering
  \begin{tikzpicture}[font=\footnotesize]
    \renewcommand{\axisdefaulttryminticks}{4} 
    \tikzstyle{every major grid}+=[style=densely dashed] 
    \tikzstyle{every axis y label}+=[yshift=-10pt] 
    \tikzstyle{every axis x label}+=[yshift=5pt]
    \tikzstyle{every axis legend}+=[cells={anchor=west},fill=white,
        at={(0.02,0.98)}, anchor=north west, font=\scriptsize ]
    \begin{loglogaxis}[
      width=.8\linewidth,
      xmin=1e-4,
      ymin=1.8e-1,
      xmax=1e2,
      ymax=3,
      bar width=1.5pt,
      grid=major,
      ymajorgrids=false,
      scaled ticks=true,
      xlabel={$\gamma$},
      ylabel={MSE}
      ]
      \addplot[black,smooth,line width=0.5pt] coordinates{
	      (100,100)
      };
      \addplot[black,densely dashed,smooth,line width=0.5pt] coordinates{
	      (100,100)
      };
      \addplot[black,only marks,mark=o,line width=0.5pt] coordinates{
	      (100,100)
      };
      \addplot[black,only marks,mark=x,line width=0.5pt] coordinates{
	      (100,100)
      };
      \addplot[black,densely dashed,line width=1pt] coordinates{(1e-4,1)(1e2,1)};
      \addplot[blue,smooth,line width=0.5pt] coordinates{
	      (0.000100,0.294341)(0.000178,0.294341)(0.000316,0.294341)(0.000562,0.294341)(0.001000,0.294341)(0.001778,0.294342)(0.003162,0.294344)(0.005623,0.294350)(0.010000,0.294370)(0.017783,0.294430)(0.031623,0.294608)(0.056234,0.295120)(0.100000,0.296495)(0.177828,0.299886)(0.316228,0.307401)(0.562341,0.322223)(1.000000,0.348439)(1.778279,0.390640)(3.162278,0.452929)(5.623413,0.536305)(10.000000,0.634795)(17.782794,0.734852)(31.622777,0.821727)(56.234133,0.887362)(100.000000,0.931925)
      };
      \addplot[blue,densely dashed,smooth,line width=0.5pt] coordinates{
	      (0.000100,1.189155)(0.000178,1.189028)(0.000316,1.188801)(0.000562,1.188398)(0.001000,1.187683)(0.001778,1.186416)(0.003162,1.184179)(0.005623,1.180247)(0.010000,1.173400)(0.017783,1.161661)(0.031623,1.142060)(0.056234,1.110722)(0.100000,1.063879)(0.177828,1.000357)(0.316228,0.924614)(0.562341,0.847437)(1.000000,0.782535)(1.778279,0.741642)(3.162278,0.731375)(5.623413,0.751623)(10.000000,0.794661)(17.782794,0.847103)(31.622777,0.895901)(56.234133,0.933822)(100.000000,0.959881)
      };
      \addplot[blue,only marks,mark=o,line width=0.5pt] coordinates{
	      (0.000100,0.293507)(0.000178,0.298607)(0.000316,0.294135)(0.000562,0.297307)(0.001000,0.295185)(0.001778,0.295655)(0.003162,0.295997)(0.005623,0.300262)(0.010000,0.301801)(0.017783,0.295596)(0.031623,0.291974)(0.056234,0.295627)(0.100000,0.300309)(0.177828,0.300973)(0.316228,0.309001)(0.562341,0.327726)(1.000000,0.352140)(1.778279,0.396906)(3.162278,0.454890)(5.623413,0.543404)(10.000000,0.636879)(17.782794,0.736362)(31.622777,0.820428)(56.234133,0.887852)(100.000000,0.932523)
      };
      \addplot[blue,only marks,mark=x,line width=0.5pt] coordinates{
	      (0.000100,1.190005)(0.000178,1.193406)(0.000316,1.191980)(0.000562,1.190730)(0.001000,1.210303)(0.001778,1.193818)(0.003162,1.209886)(0.005623,1.195055)(0.010000,1.193258)(0.017783,1.170058)(0.031623,1.149518)(0.056234,1.122838)(0.100000,1.079273)(0.177828,0.991851)(0.316228,0.918450)(0.562341,0.844621)(1.000000,0.780261)(1.778279,0.742204)(3.162278,0.732798)(5.623413,0.753116)(10.000000,0.796694)(17.782794,0.847020)(31.622777,0.894876)(56.234133,0.934895)(100.000000,0.960317)
      };
      \addplot[green!60!black,smooth,line width=0.5pt] coordinates{
	      (0.000100,0.500274)(0.000178,0.500274)(0.000316,0.500274)(0.000562,0.500275)(0.001000,0.500275)(0.001778,0.500276)(0.003162,0.500279)(0.005623,0.500288)(0.010000,0.500318)(0.017783,0.500408)(0.031623,0.500675)(0.056234,0.501434)(0.100000,0.503458)(0.177828,0.508366)(0.316228,0.518932)(0.562341,0.538814)(1.000000,0.571474)(1.778279,0.618657)(3.162278,0.678988)(5.623413,0.747181)(10.000000,0.814886)(17.782794,0.873751)(31.622777,0.918972)(56.234133,0.950336)(100.000000,0.970516)
      };
      \addplot[green!60!black,densely dashed,smooth,line width=0.5pt] coordinates{
	      (0.000100,2.012558)(0.000178,2.012345)(0.000316,2.011967)(0.000562,2.011295)(0.001000,2.010103)(0.001778,2.007990)(0.003162,2.004258)(0.005623,1.997698)(0.010000,1.986266)(0.017783,1.966646)(0.031623,1.933825)(0.056234,1.881168)(0.100000,1.801973)(0.177828,1.693331)(0.316228,1.560855)(0.562341,1.419527)(1.000000,1.287864)(1.778279,1.179992)(3.162278,1.101796)(5.623413,1.051722)(10.000000,1.023598)(17.782794,1.009784)(31.622777,1.003776)(56.234133,1.001395)(100.000000,1.000503)
      };
      \addplot[green!60!black,only marks,mark=o,line width=0.5pt] coordinates{
	      (0.000100,0.499064)(0.000178,0.494018)(0.000316,0.499220)(0.000562,0.498713)(0.001000,0.499637)(0.001778,0.496288)(0.003162,0.498687)(0.005623,0.496910)(0.010000,0.499307)(0.017783,0.502079)(0.031623,0.500345)(0.056234,0.501988)(0.100000,0.502470)(0.177828,0.511545)(0.316228,0.518892)(0.562341,0.539897)(1.000000,0.570803)(1.778279,0.617909)(3.162278,0.681177)(5.623413,0.747050)(10.000000,0.815582)(17.782794,0.874467)(31.622777,0.919037)(56.234133,0.950240)(100.000000,0.970654)
      };
      \addplot[green!60!black,only marks,mark=x,line width=0.5pt] coordinates{
	      (0.000100,2.021738)(0.000178,2.041790)(0.000316,2.037710)(0.000562,2.018973)(0.001000,2.009969)(0.001778,2.036796)(0.003162,2.020555)(0.005623,2.016873)(0.010000,1.982284)(0.017783,1.980500)(0.031623,1.948979)(0.056234,1.899229)(0.100000,1.801710)(0.177828,1.690282)(0.316228,1.562571)(0.562341,1.427339)(1.000000,1.287909)(1.778279,1.181485)(3.162278,1.102856)(5.623413,1.050473)(10.000000,1.024870)(17.782794,1.009816)(31.622777,1.003577)(56.234133,1.001429)(100.000000,1.000477)
      };
      \addplot[red,smooth,line width=0.5pt] coordinates{
	      (0.000100,0.211604)(0.000178,0.211604)(0.000316,0.211604)(0.000562,0.211604)(0.001000,0.211604)(0.001778,0.211605)(0.003162,0.211606)(0.005623,0.211610)(0.010000,0.211624)(0.017783,0.211664)(0.031623,0.211786)(0.056234,0.212134)(0.100000,0.213076)(0.177828,0.215416)(0.316228,0.220655)(0.562341,0.231151)(1.000000,0.250188)(1.778279,0.282163)(3.162278,0.332746)(5.623413,0.407560)(10.000000,0.507422)(17.782794,0.622510)(31.622777,0.734209)(56.234133,0.826188)(100.000000,0.892492)
      };
      \addplot[red,smooth,densely dashed,line width=0.5pt] coordinates{
	      (0.000100,0.843056)(0.000178,0.842969)(0.000316,0.842814)(0.000562,0.842538)(0.001000,0.842049)(0.001778,0.841182)(0.003162,0.839650)(0.005623,0.836957)(0.010000,0.832266)(0.017783,0.824214)(0.031623,0.810748)(0.056234,0.789162)(0.100000,0.756777)(0.177828,0.712650)(0.316228,0.659794)(0.562341,0.605924)(1.000000,0.561450)(1.778279,0.536486)(3.162278,0.539069)(5.623413,0.573712)(10.000000,0.638142)(17.782794,0.720447)(31.622777,0.802835)(56.234133,0.871147)(100.000000,0.920385)
      };
      \addplot[red,only marks,mark=o,line width=0.5pt] coordinates{
	      (0.000100,0.224073)(0.000178,0.225284)(0.000316,0.220261)(0.000562,0.220192)(0.001000,0.223329)(0.001778,0.223665)(0.003162,0.225116)(0.005623,0.222437)(0.010000,0.223523)(0.017783,0.224997)(0.031623,0.222922)(0.056234,0.223748)(0.100000,0.226553)(0.177828,0.228833)(0.316228,0.232534)(0.562341,0.242176)(1.000000,0.264195)(1.778279,0.299556)(3.162278,0.349944)(5.623413,0.425668)(10.000000,0.521760)(17.782794,0.629774)(31.622777,0.738828)(56.234133,0.827399)(100.000000,0.893128)
      };
      \addplot[red,only marks,mark=x,line width=0.5pt] coordinates{
	      (0.000100,0.877086)(0.000178,0.880711)(0.000316,0.894424)(0.000562,0.893535)(0.001000,0.892084)(0.001778,0.888574)(0.003162,0.879327)(0.005623,0.881115)(0.010000,0.869189)(0.017783,0.874815)(0.031623,0.846885)(0.056234,0.821074)(0.100000,0.795166)(0.177828,0.739262)(0.316228,0.683697)(0.562341,0.624960)(1.000000,0.584023)(1.778279,0.559675)(3.162278,0.560275)(5.623413,0.591154)(10.000000,0.651348)(17.782794,0.729238)(31.622777,0.806514)(56.234133,0.871218)(100.000000,0.921555)
      };
      \draw[->,thick] (axis cs:3,.3) -- (axis cs:.3,.3) node [right,pos=0,font=\footnotesize] { \BLUE $W_{ij}\sim \mathcal N(0,1)$ };
      \draw[->,thick] (axis cs:.004,.4) -- (axis cs:.004,.5) node [below,pos=0,font=\footnotesize] { \GREEN $W_{ij}\sim {\rm Bern}$ };
      \draw[->,thick] (axis cs:3,.23) -- (axis cs:.6,.23) node [right,pos=0,font=\footnotesize] { \RED $W_{ij}\sim {\rm Stud}$ };
      \legend{ { $\bar E_{\rm train}$ }, { $\bar E_{\rm test}$ }, { $E_{\rm train}$ }, { $E_{\rm test}$ } }
    \end{loglogaxis}
  \end{tikzpicture}

\medskip

  \begin{tabular}{ccc}
  \begin{tikzpicture}[font=\footnotesize]
    \renewcommand{\axisdefaulttryminticks}{4} 
    \tikzstyle{every major grid}+=[style=densely dashed]       
    \tikzstyle{every axis y label}+=[yshift=-10pt] 
    \tikzstyle{every axis x label}+=[yshift=5pt]
    \tikzstyle{every axis legend}+=[cells={anchor=west},fill=white,
        at={(0.98,0.98)}, anchor=north east, font=\scriptsize ]
    \begin{axis}[
      width=.42\linewidth,
      height=.4\linewidth,
      xmin=0,
      ymin=0,
      xmax=14,
      xticklabels={},
      yticklabels={},
      axis lines=left,
      xtick=\empty,
      ytick=\empty,
      bar width=1pt,
      grid=none,
      ymajorgrids=false,
      scaled ticks=true,
      ]
      \addplot+[ybar,mark=none,color=black,fill=blue!80!black] coordinates{
	      (1.455462,0.029329)(1.588780,0.043993)(1.722098,0.124648)(1.855416,0.256628)(1.988734,0.337282)(2.122052,0.381275)(2.255370,0.403272)(2.388688,0.417937)(2.522006,0.454598)(2.655324,0.425269)(2.788642,0.425269)(2.921960,0.410604)(3.055278,0.403272)(3.188596,0.381275)(3.321914,0.359279)(3.455232,0.337282)(3.588550,0.315286)(3.721868,0.285957)(3.855186,0.271292)(3.988504,0.241963)(4.121822,0.212634)(4.255140,0.205302)(4.388458,0.161309)(4.521776,0.124648)(4.655094,0.124648)(4.788412,0.080654)(4.921730,0.080654)(5.055048,0.051326)(5.188366,0.051326)(5.321684,0.036661)(5.455002,0.021997)(5.588320,0.007332)(5.721638,0.014664)(5.854956,0.007332)(5.988274,0.007332)(6.121592,0.000000)(6.254910,0.000000)(6.388228,0.000000)(6.521546,0.000000)(6.654864,0.000000)(6.788182,0.000000)(6.921500,0.000000)(7.054818,0.000000)(7.188136,0.000000)(7.321455,0.000000)(7.454773,0.000000)(7.588091,0.000000)(7.721409,0.000000)(7.854727,0.000000)(7.988045,0.007332)
      };
      \node[font=\footnotesize] at (axis cs:7.99,.18) { close };
      \draw[->,thick] (axis cs:7.99,.1) -- (axis cs:7.99,.02) node [above,pos=0,font=\footnotesize] { spike };
    \end{axis}
  \end{tikzpicture}
  &
  \begin{tikzpicture}[font=\footnotesize]
    \renewcommand{\axisdefaulttryminticks}{4} 
    \tikzstyle{every major grid}+=[style=densely dashed]       
    \tikzstyle{every axis y label}+=[yshift=-10pt] 
    \tikzstyle{every axis x label}+=[yshift=5pt]
    \tikzstyle{every axis legend}+=[cells={anchor=west},fill=white,
        at={(0.98,0.98)}, anchor=north east, font=\scriptsize ]
    \begin{axis}[
      width=.42\linewidth,
      height=.4\linewidth,
      xmin=0,
      ymin=0,
      xmax=14,
      xticklabels={},
      yticklabels={},
      bar width=1.3pt,
      grid=none,
      axis lines=left,
      xtick=\empty,
      ytick=\empty,
      ymajorgrids=false,
      scaled ticks=true,
      ]
      \addplot+[ybar,mark=none,color=black,fill=green!60!black] coordinates{
	      (1.542570,0.041936)(1.729049,0.115323)(1.915528,0.272583)(2.102007,0.372180)(2.288486,0.424600)(2.474965,0.450810)(2.661444,0.424600)(2.847923,0.429842)(3.034401,0.408874)(3.220880,0.382664)(3.407359,0.361696)(3.593838,0.314518)(3.780317,0.277824)(3.966796,0.241131)(4.153275,0.214921)(4.339754,0.167743)(4.526233,0.136291)(4.712712,0.104839)(4.899191,0.099597)(5.085669,0.036694)(5.272148,0.036694)(5.458627,0.026210)(5.645106,0.010484)(5.831585,0.005242)(6.018064,0.005242)
      };
      \node[font=\footnotesize] at (axis cs:10,.16) { no spike };
    \end{axis}
  \end{tikzpicture}
  &
  \begin{tikzpicture}[font=\footnotesize]
    \renewcommand{\axisdefaulttryminticks}{4} 
    \tikzstyle{every major grid}+=[style=densely dashed]       
    \tikzstyle{every axis y label}+=[yshift=-10pt] 
    \tikzstyle{every axis x label}+=[yshift=5pt]
    \tikzstyle{every axis legend}+=[cells={anchor=west},fill=white,
        at={(0.98,0.98)}, anchor=north east, font=\scriptsize ]
    \begin{axis}[
      width=.42\linewidth,
      height=.4\linewidth,
      xmin=0,
      ymin=0,
      xmax=14,
      xticklabels={},
      yticklabels={},
      bar width=1.5pt,
      grid=none,
      axis lines=left,
      xtick=\empty,
      ytick=\empty,
      ymajorgrids=false,
      scaled ticks=true,
      ]
      \addplot+[ybar,mark=none,color=black,fill=red] coordinates{
	      (1.528825,0.030309)(1.754585,0.134227)(1.980345,0.281443)(2.206105,0.381030)(2.431866,0.424329)(2.657626,0.428659)(2.883386,0.411340)(3.109146,0.398350)(3.334906,0.359381)(3.560666,0.316082)(3.786427,0.285773)(4.012187,0.246804)(4.237947,0.212165)(4.463707,0.168866)(4.689467,0.134227)(4.915227,0.082268)(5.140987,0.060618)(5.366748,0.030309)(5.592508,0.021649)(5.818268,0.008660)(6.044028,0.004330)(6.269788,0.000000)(6.495548,0.004330)(6.721309,0.000000)(6.947069,0.000000)(7.172829,0.000000)(7.398589,0.000000)(7.624349,0.000000)(7.850109,0.000000)(8.075870,0.000000)(8.301630,0.000000)(8.527390,0.000000)(8.753150,0.000000)(8.978910,0.000000)(9.204670,0.000000)(9.430430,0.000000)(9.656191,0.000000)(9.881951,0.000000)(10.107711,0.000000)(10.333471,0.000000)(10.559231,0.000000)(10.784991,0.000000)(11.010752,0.000000)(11.236512,0.000000)(11.462272,0.000000)(11.688032,0.000000)(11.913792,0.000000)(12.139552,0.000000)(12.365313,0.000000)(12.591073,0.004330)
      };
      \node[font=\footnotesize] at (axis cs:12.59,.18) { far };
      \draw[->,thick] (axis cs:12.59,.1) -- (axis cs:12.59,.02) node [above,pos=0,font=\footnotesize] { spike };
    \end{axis}
  \end{tikzpicture}
  \\
  \begin{tikzpicture}[font=\footnotesize]
    \renewcommand{\axisdefaulttryminticks}{4} 
    \tikzstyle{every major grid}+=[style=densely dashed] 
    \tikzstyle{every axis y label}+=[yshift=-10pt] 
    \tikzstyle{every axis x label}+=[yshift=5pt]
    \tikzstyle{every axis legend}+=[cells={anchor=west},fill=white,
        at={(0.02,0.98)}, anchor=north west, font=\scriptsize ]
    \begin{axis}[
      width=.42\linewidth,
      height=.25\linewidth,
      xmin=1,
      xmax=1024,
      grid=none,
      axis y line=left,
      axis x line=center,
      xtick=\empty,
      ytick=\empty,
      grid=major,
      ymajorgrids=false,
      scaled ticks=true,
      ]
      \addplot[blue,smooth,line width=0.5pt] coordinates{
	      (1,0.037818)(6,0.029495)(11,0.024266)(16,0.036880)(21,0.050837)(26,0.039024)(31,0.030538)(36,0.030677)(41,0.037307)(46,0.024004)(51,0.031549)(56,0.021945)(61,0.033042)(66,0.052454)(71,0.024725)(76,0.030202)(81,0.023079)(86,0.038887)(91,0.030747)(96,0.027759)(101,0.019611)(106,0.034334)(111,0.049252)(116,0.029411)(121,0.033057)(126,0.030803)(131,0.032121)(136,0.018558)(141,0.033234)(146,0.066166)(151,0.015076)(156,0.040411)(161,0.022116)(166,0.024241)(171,0.030821)(176,0.015354)(181,0.021561)(186,0.018662)(191,0.030717)(196,0.040237)(201,0.025057)(206,0.037179)(211,0.037981)(216,0.032532)(221,0.023865)(226,0.027581)(231,0.025257)(236,0.048615)(241,0.014670)(246,0.032767)(251,0.023645)(256,0.032883)(261,0.032484)(266,0.035468)(271,0.019458)(276,0.019662)(281,0.024293)(286,0.025439)(291,0.050837)(296,0.023538)(301,0.027808)(306,0.045330)(311,0.036850)(316,0.037780)(321,0.017372)(326,0.030170)(331,0.033653)(336,0.028505)(341,0.028565)(346,0.037981)(351,0.029885)(356,0.032274)(361,0.033847)(366,0.040845)(371,0.033351)(376,0.039274)(381,0.025664)(386,0.025269)(391,0.020643)(396,0.028030)(401,0.020515)(406,0.020776)(411,0.040266)(416,0.034470)(421,0.022618)(426,0.037207)(431,0.019446)(436,0.034692)(441,0.025039)(446,0.037864)(451,0.028708)(456,0.024751)(461,0.030001)(466,0.038186)(471,0.026238)(476,0.035579)(481,0.041447)(486,0.020271)(491,0.067070)(496,0.031143)(501,0.018277)(506,0.017066)(511,0.026323)(516,-0.022501)(521,-0.029083)(526,-0.028991)(531,-0.035335)(536,-0.020832)(541,-0.021170)(546,-0.029379)(551,-0.023367)(556,-0.025302)(561,-0.028235)(566,-0.069148)(571,-0.020906)(576,-0.029888)(581,-0.015041)(586,-0.041608)(591,-0.056512)(596,-0.032243)(601,-0.026623)(606,-0.031555)(611,-0.020229)(616,-0.019793)(621,-0.025792)(626,-0.025293)(631,-0.031517)(636,-0.030252)(641,-0.026219)(646,-0.032342)(651,-0.020373)(656,-0.015442)(661,-0.032583)(666,-0.029604)(671,-0.029873)(676,-0.034837)(681,-0.020020)(686,-0.017577)(691,-0.020616)(696,-0.031620)(701,-0.035262)(706,-0.055430)(711,-0.015695)(716,-0.023333)(721,-0.032576)(726,-0.015941)(731,-0.018878)(736,-0.018232)(741,-0.012552)(746,-0.023444)(751,-0.018351)(756,-0.023812)(761,-0.015155)(766,-0.025720)(771,-0.069702)(776,-0.037705)(781,-0.016926)(786,-0.027963)(791,-0.036155)(796,-0.017563)(801,-0.023209)(806,-0.034810)(811,-0.058750)(816,-0.017845)(821,-0.038203)(826,-0.018846)(831,-0.025815)(836,-0.015856)(841,-0.023029)(846,-0.059483)(851,-0.022696)(856,-0.021433)(861,-0.023523)(866,-0.015594)(871,-0.022994)(876,-0.039522)(881,-0.023515)(886,-0.031767)(891,-0.024802)(896,-0.045565)(901,-0.026474)(906,-0.027392)(911,-0.023428)(916,-0.039617)(921,-0.048146)(926,-0.029952)(931,-0.019395)(936,-0.020475)(941,-0.017799)(946,-0.026479)(951,-0.024467)(956,-0.023342)(961,-0.017308)(966,-0.048216)(971,-0.034403)(976,-0.012449)(981,-0.027277)(986,-0.023888)(991,-0.050781)(996,-0.035230)(1001,-0.025994)(1006,-0.020496)(1011,-0.022865)(1016,-0.035093)(1021,-0.018410)
      };
    \end{axis}
  \end{tikzpicture}
  &
  \begin{tikzpicture}[font=\footnotesize]
    \renewcommand{\axisdefaulttryminticks}{4} 
    \tikzstyle{every major grid}+=[style=densely dashed] 
    \tikzstyle{every axis y label}+=[yshift=-10pt] 
    \tikzstyle{every axis x label}+=[yshift=5pt]
    \tikzstyle{every axis legend}+=[cells={anchor=west},fill=white,
        at={(0.02,0.98)}, anchor=north west, font=\scriptsize ]
    \begin{axis}[
      width=.42\linewidth,
      height=.25\linewidth,
      xmin=1,
      ymin=-.1,
      xmax=1024,
      ymax=.1,
      grid=none,
      axis y line=left,
      axis x line=center,
      xtick=\empty,
      ytick=\empty,
      grid=major,
      ymajorgrids=false,
      scaled ticks=true,
      ]
      \addplot[green!60!black,smooth,line width=0.5pt] coordinates{
	      (1,-0.002222)(6,0.008039)(11,-0.002379)(16,0.002475)(21,0.003087)(26,0.001529)(31,-0.004506)(36,-0.007458)(41,-0.003066)(46,-0.002261)(51,-0.006692)(56,-0.004378)(61,-0.004327)(66,-0.001608)(71,-0.005438)(76,-0.005145)(81,-0.001296)(86,-0.001741)(91,-0.002866)(96,-0.003431)(101,-0.001799)(106,-0.000267)(111,-0.001402)(116,-0.005966)(121,-0.009411)(126,0.002957)(131,-0.008061)(136,-0.004520)(141,0.001500)(146,-0.001696)(151,-0.001562)(156,-0.003302)(161,0.000073)(166,-0.002800)(171,-0.004041)(176,0.007406)(181,-0.005698)(186,0.005572)(191,0.004594)(196,0.010897)(201,-0.003824)(206,-0.002347)(211,-0.003665)(216,-0.014802)(221,-0.016161)(226,-0.003660)(231,-0.002168)(236,-0.003763)(241,-0.008169)(246,-0.007746)(251,-0.006027)(256,-0.006614)(261,-0.001610)(266,0.002252)(271,-0.001830)(276,0.010138)(281,-0.002067)(286,-0.001756)(291,-0.001088)(296,-0.000549)(301,-0.005450)(306,-0.003246)(311,-0.002984)(316,-0.007095)(321,-0.005747)(326,-0.004256)(331,-0.002759)(336,0.005863)(341,-0.003676)(346,-0.002364)(351,0.005985)(356,0.000744)(361,-0.008944)(366,0.000149)(371,0.000812)(376,-0.002364)(381,-0.004136)(386,0.008597)(391,-0.000530)(396,-0.005605)(401,0.006028)(406,0.000230)(411,0.017295)(416,-0.004406)(421,0.006920)(426,0.018242)(431,0.003137)(436,-0.002111)(441,-0.001112)(446,0.001129)(451,-0.001817)(456,-0.004038)(461,-0.003030)(466,-0.002631)(471,0.000390)(476,-0.002275)(481,-0.003555)(486,0.022437)(491,-0.002449)(496,0.002046)(501,-0.001887)(506,0.003778)(511,0.003578)(516,0.005943)(521,-0.014724)(526,-0.013373)(531,-0.002963)(536,-0.009792)(541,-0.008550)(546,-0.007653)(551,0.016970)(556,0.016948)(561,-0.006876)(566,0.006906)(571,-0.013491)(576,0.004654)(581,-0.004177)(586,0.003965)(591,-0.011391)(596,-0.002027)(601,-0.002998)(606,-0.017602)(611,-0.000836)(616,-0.014160)(621,0.004208)(626,0.003635)(631,-0.005663)(636,0.920222)(641,0.025544)(646,-0.011057)(651,0.000311)(656,0.002344)(661,-0.002198)(666,-0.016623)(671,-0.006689)(676,-0.007252)(681,-0.013887)(686,-0.006341)(691,0.016698)(696,-0.009637)(701,-0.006028)(706,0.000892)(711,-0.012037)(716,-0.003580)(721,-0.004444)(726,-0.008397)(731,0.007123)(736,0.001264)(741,-0.004185)(746,-0.006505)(751,-0.002687)(756,-0.006487)(761,0.058782)(766,-0.006229)(771,0.026874)(776,-0.006482)(781,-0.024990)(786,0.012463)(791,0.037058)(796,-0.003422)(801,0.003708)(806,-0.044089)(811,-0.004826)(816,-0.019090)(821,0.008606)(826,0.034162)(831,-0.004693)(836,0.041857)(841,-0.006870)(846,-0.005756)(851,-0.015806)(856,-0.006591)(861,-0.006038)(866,0.012758)(871,-0.002360)(876,0.008530)(881,-0.005579)(886,-0.006225)(891,0.001958)(896,-0.005784)(901,0.000853)(906,-0.001740)(911,-0.001767)(916,-0.004184)(921,-0.015699)(926,-0.013084)(931,-0.005341)(936,-0.004668)(941,-0.003867)(946,-0.003596)(951,-0.008894)(956,-0.006610)(961,0.141784)(966,0.002471)(971,0.019587)(976,-0.013789)(981,-0.012552)(986,0.004599)(991,-0.009370)(996,-0.005911)(1001,-0.003098)(1006,-0.004689)(1011,-0.007554)(1016,-0.008100)(1021,-0.010510)
      };
    \end{axis}
  \end{tikzpicture}
  &
  \begin{tikzpicture}[font=\footnotesize]
    \renewcommand{\axisdefaulttryminticks}{4} 
    \tikzstyle{every major grid}+=[style=densely dashed] 
    \tikzstyle{every axis y label}+=[yshift=-10pt] 
    \tikzstyle{every axis x label}+=[yshift=5pt]
    \tikzstyle{every axis legend}+=[cells={anchor=west},fill=white,
        at={(0.02,0.98)}, anchor=north west, font=\scriptsize ]
    \begin{axis}[
      width=.42\linewidth,
      height=.25\linewidth,
      xmin=1,
      xmax=1024,
      grid=none,
      axis y line=left,
      axis x line=center,
      xtick=\empty,
      ytick=\empty,
      grid=major,
      ymajorgrids=false,
      scaled ticks=true,
      ]
      \addplot[red,smooth,line width=0.5pt] coordinates{
	      (1,-0.019164)(6,-0.016625)(11,-0.017413)(16,-0.033070)(21,-0.020911)(26,-0.038029)(31,-0.023589)(36,-0.024252)(41,-0.036864)(46,-0.038648)(51,-0.033177)(56,-0.020684)(61,-0.023300)(66,-0.023510)(71,-0.034564)(76,-0.042085)(81,-0.022745)(86,-0.033179)(91,-0.023489)(96,-0.027219)(101,-0.022440)(106,-0.032243)(111,-0.038437)(116,-0.042175)(121,-0.014122)(126,-0.028242)(131,-0.027692)(136,-0.026181)(141,-0.037487)(146,-0.032897)(151,-0.019717)(156,-0.043877)(161,-0.031780)(166,-0.020298)(171,-0.036722)(176,-0.032127)(181,-0.021670)(186,-0.025151)(191,-0.019444)(196,-0.031864)(201,-0.030182)(206,-0.041608)(211,-0.023201)(216,-0.054346)(221,-0.020261)(226,-0.016688)(231,-0.040375)(236,-0.030418)(241,-0.030563)(246,-0.044557)(251,-0.034791)(256,-0.031335)(261,-0.040709)(266,-0.033127)(271,-0.032524)(276,-0.024779)(281,-0.035941)(286,-0.019543)(291,-0.025308)(296,-0.017016)(301,-0.028784)(306,-0.032956)(311,-0.023236)(316,-0.028368)(321,-0.032170)(326,-0.041128)(331,-0.019694)(336,-0.035721)(341,-0.048011)(346,-0.027781)(351,-0.038892)(356,-0.022872)(361,-0.022350)(366,-0.050208)(371,-0.031029)(376,-0.020497)(381,-0.026153)(386,-0.034847)(391,-0.030319)(396,-0.023796)(401,-0.022284)(406,-0.033955)(411,-0.034820)(416,-0.040370)(421,-0.039889)(426,-0.024294)(431,-0.030363)(436,-0.029446)(441,-0.031771)(446,-0.029048)(451,-0.029131)(456,-0.036958)(461,-0.031307)(466,-0.035026)(471,-0.028021)(476,-0.023157)(481,-0.024548)(486,-0.029800)(491,-0.026241)(496,-0.023603)(501,-0.032588)(506,-0.033287)(511,-0.037645)(516,0.028362)(521,0.030497)(526,0.026791)(531,0.027770)(536,0.028571)(541,0.041319)(546,0.026173)(551,0.017658)(556,0.027793)(561,0.020302)(566,0.022524)(571,0.022277)(576,0.029495)(581,0.024903)(586,0.033915)(591,0.030554)(596,0.032602)(601,0.038410)(606,0.024458)(611,0.020677)(616,0.019023)(621,0.034444)(626,0.020406)(631,0.027553)(636,0.030098)(641,0.023139)(646,0.031413)(651,0.028313)(656,0.031140)(661,0.018041)(666,0.022037)(671,0.030555)(676,0.035096)(681,0.032984)(686,0.032374)(691,0.036091)(696,0.028177)(701,0.017210)(706,0.028554)(711,0.053341)(716,0.037205)(721,0.040134)(726,0.017363)(731,0.028226)(736,0.036485)(741,0.034755)(746,0.025141)(751,0.024446)(756,0.031017)(761,0.032387)(766,0.030148)(771,0.031970)(776,0.029548)(781,0.030411)(786,0.037248)(791,0.041727)(796,0.043232)(801,0.025898)(806,0.024026)(811,0.025608)(816,0.028083)(821,0.019113)(826,0.021610)(831,0.017291)(836,0.026999)(841,0.033023)(846,0.024307)(851,0.013531)(856,0.037831)(861,0.021997)(866,0.031194)(871,0.025378)(876,0.030241)(881,0.040006)(886,0.039109)(891,0.025078)(896,0.037648)(901,0.041248)(906,0.040175)(911,0.024992)(916,0.034177)(921,0.033626)(926,0.023517)(931,0.025738)(936,0.041721)(941,0.035021)(946,0.026035)(951,0.027241)(956,0.025679)(961,0.028485)(966,0.027152)(971,0.041668)(976,0.017911)(981,0.024498)(986,0.032934)(991,0.030955)(996,0.028231)(1001,0.027248)(1006,0.025507)(1011,0.031183)(1016,0.027985)(1021,0.035991)
      };
    \end{axis}
  \end{tikzpicture}
\\
$W_{ij}\sim \mathcal N(0,1)$ & $W_{ij}\sim {\rm Bern}$ & $W_{ij}\sim {\rm Stud}$
\end{tabular}
\caption{(Top) Neural network performance for $\sigma(t)=-\frac12t^2+1$, with different $W_{ij}$, for a $2$-class Gaussian mixture model (see details in text), $n=512$, $T=\hat T=1024$, $p=256$. (Bottom) Spectra and second eigenvector of $\Phi$ for different $W_{ij}$ (first eigenvalues are of order $n$ and not shown; associated eigenvectors are provably non informative).}
  \label{fig:poly2}
\end{figure}

\subsection{Limiting cases}
\label{sec:limiting_cases}

We have suggested that $\Phi$ contains, in its dominant eigenmodes, all the usable information describing $X$. In the Gaussian mixture example above, it was notably shown that $\Phi$ may completely fail to contain this information, resulting in the impossibility to perform a classification task, even if one were to take infinitely many neurons in the network. For $\Phi$ containing useful information about $X$, it is intuitive to expect that both $\inf_\gamma\bar E_{\rm train}$ and $\inf_\gamma\bar E_{\rm test}$ become smaller as $n/T$ and $n/p$ become large. It is in fact easy to see that, if $\Phi$ is invertible (which is likely to occur in most cases if $\liminf_n T/p>1$), then
\begin{align*}
	&\lim_{n\to\infty} \bar E_{\rm train} = 0 \\
	&\textcolor{black}{\lim_{n\to\infty} \bar E_{\rm test} - \frac1{\hat T} \left\| \hat Y^\trans - \Phi_{\hat XX} \Phi^{-1} Y^\trans \right\|^2_F  = 0}
\end{align*}
and we fall back on the performance of a classical kernel regression. It is interesting in particular to note that, as the number of neurons $n$ becomes large, the effect of $\gamma$ \textcolor{black}{on $E_{\rm test}$} flattens out. Therefore, a smart choice of $\gamma$ is only relevant for small (and thus computationally more efficient) neuron layers. This observation is depicted in Figure~\ref{fig:perf_limit} where it is made clear that a growth of $n$ reduces $E_{\rm train}$ to zero while $E_{\rm test}$ saturates to a non-zero limit which becomes increasingly irrespective of $\gamma$. Note additionally the interesting phenomenon occurring for $n\leq T$ where too small values of $\gamma$ induce important performance losses, thereby suggesting a strong importance of proper choices of $\gamma$ in this regime.

\begin{figure}[h!]
  \centering
  \begin{tikzpicture}[font=\footnotesize]
    \renewcommand{\axisdefaulttryminticks}{4} 
    \tikzstyle{every major grid}+=[style=densely dashed] 
    \tikzstyle{every axis y label}+=[yshift=-10pt] 
    \tikzstyle{every axis x label}+=[yshift=5pt]
    \tikzstyle{every axis legend}+=[cells={anchor=west},fill=white,
        at={(0.02,0.98)}, anchor=north west, font=\scriptsize ]
    \begin{loglogaxis}[
      width=.8\linewidth,
      xmin=1e-4,
      ymin=5e-2,
      xmax=1e2,
      ymax=1,
      bar width=1.5pt,
      grid=major,
      ymajorgrids=false,
      scaled ticks=true,
      xlabel={$\gamma$},
      ylabel={MSE}
      ]
      \addplot[black,smooth,line width=0.5pt] coordinates{
              (100,100)
      };
      \addplot[black,densely dashed,smooth,line width=0.5pt] coordinates{
	      (100,100)
      };
      \addplot[black,only marks,mark=o,line width=0.5pt] coordinates{
              (100,100)
      };
      \addplot[black,only marks,mark=x,line width=0.5pt] coordinates{
	      (100,100)
      };
      \addplot[blue!50!white,smooth,line width=0.5pt] coordinates{
	      (0.000100,0.129021)(0.000178,0.129022)(0.000316,0.129024)(0.000562,0.129031)(0.001000,0.129051)(0.001778,0.129113)(0.003162,0.129289)(0.005623,0.129763)(0.010000,0.130925)(0.017783,0.133471)(0.031623,0.138367)(0.056234,0.146647)(0.100000,0.159151)(0.177828,0.176374)(0.316228,0.198592)(0.562341,0.226512)(1.000000,0.262677)(1.778279,0.312930)(3.162278,0.385409)(5.623413,0.484216)(10.000000,0.601201)(17.782794,0.717102)(31.622777,0.813921)(56.234133,0.884416)(100.000000,0.930951)
      };
      \addplot[blue!50!white,densely dashed,smooth,line width=0.5pt] coordinates{
(0.000100,0.317063)(0.000178,0.316774)(0.000316,0.316266)(0.000562,0.315382)(0.001000,0.313866)(0.001778,0.311335)(0.003162,0.307291)(0.005623,0.301265)(0.010000,0.293167)(0.017783,0.283740)(0.031623,0.274696)(0.056234,0.268231)(0.100000,0.266280)(0.177828,0.270109)(0.316228,0.280412)(0.562341,0.297916)(1.000000,0.324689)(1.778279,0.365738)(3.162278,0.428526)(5.623413,0.517252)(10.000000,0.624694)(17.782794,0.732659)(31.622777,0.823648)(56.234133,0.890254)(100.000000,0.934363)
      };
      \addplot[blue!50!white,only marks,mark=o,line width=0.5pt] coordinates{
(0.000100,0.129704)(0.000178,0.129705)(0.000316,0.129707)(0.000562,0.129715)(0.001000,0.129737)(0.001778,0.129802)(0.003162,0.129989)(0.005623,0.130488)(0.010000,0.131705)(0.017783,0.134346)(0.031623,0.139376)(0.056234,0.147807)(0.100000,0.160452)(0.177828,0.177807)(0.316228,0.200180)(0.562341,0.228323)(1.000000,0.264802)(1.778279,0.315451)(3.162278,0.388331)(5.623413,0.487345)(10.000000,0.604136)(17.782794,0.719478)(31.622777,0.815617)(56.234133,0.885521)(100.000000,0.931629)
      };
      \addplot[blue!50!white,only marks,mark=x,line width=0.5pt] coordinates{
(0.000100,0.320014)(0.000178,0.319711)(0.000316,0.319177)(0.000562,0.318250)(0.001000,0.316660)(0.001778,0.314012)(0.003162,0.309793)(0.005623,0.303533)(0.010000,0.295177)(0.017783,0.285540)(0.031623,0.276415)(0.056234,0.270033)(0.100000,0.268314)(0.177828,0.272496)(0.316228,0.283246)(0.562341,0.301268)(1.000000,0.328600)(1.778279,0.370195)(3.162278,0.433390)(5.623413,0.522147)(10.000000,0.629068)(17.782794,0.736088)(31.622777,0.826049)(56.234133,0.891803)(100.000000,0.935310)
      };
      \addplot[blue!75!white,smooth,line width=0.5pt] coordinates{
              (0.000100,0.060032)(0.000178,0.060037)(0.000316,0.060051)(0.000562,0.060092)(0.001000,0.060210)(0.001778,0.060520)(0.003162,0.061264)(0.005623,0.062850)(0.010000,0.065838)(0.017783,0.070852)(0.031623,0.078496)(0.056234,0.089318)(0.100000,0.103786)(0.177828,0.122225)(0.316228,0.144733)(0.562341,0.171222)(1.000000,0.201947)(1.778279,0.238929)(3.162278,0.287718)(5.623413,0.357078)(10.000000,0.453269)(17.782794,0.570838)(31.622777,0.691251)(56.234133,0.794663)(100.000000,0.871479)	      
      };
      \addplot[blue!75!white,densely dashed,smooth,line width=0.5pt] coordinates{
	      (0.000100,0.369989)(0.000178,0.368107)(0.000316,0.364890)(0.000562,0.359540)(0.001000,0.351041)(0.001778,0.338439)(0.003162,0.321477)(0.005623,0.301251)(0.010000,0.280183)(0.017783,0.261069)(0.031623,0.246054)(0.056234,0.236308)(0.100000,0.232295)(0.177828,0.234146)(0.316228,0.241874)(0.562341,0.255459)(1.000000,0.275160)(1.778279,0.302571)(3.162278,0.342330)(5.623413,0.402283)(10.000000,0.488500)(17.782794,0.596315)(31.622777,0.708353)(56.234133,0.805456)(100.000000,0.877995)
      };
      \addplot[blue!75!white,only marks,mark=o,line width=0.5pt] coordinates{
              (0.000100,0.060169)(0.000178,0.059654)(0.000316,0.059800)(0.000562,0.060305)(0.001000,0.060212)(0.001778,0.060921)(0.003162,0.061980)(0.005623,0.063356)(0.010000,0.065644)(0.017783,0.070753)(0.031623,0.079214)(0.056234,0.090096)(0.100000,0.105253)(0.177828,0.122725)(0.316228,0.145592)(0.562341,0.171732)(1.000000,0.202450)(1.778279,0.240647)(3.162278,0.288083)(5.623413,0.357426)(10.000000,0.454285)(17.782794,0.570757)(31.622777,0.691082)(56.234133,0.793788)(100.000000,0.872311)
      };
      \addplot[blue!75!white,only marks,mark=x,line width=0.5pt] coordinates{
	      (0.000100,0.371541)(0.000178,0.371174)(0.000316,0.364297)(0.000562,0.359050)(0.001000,0.349780)(0.001778,0.338103)(0.003162,0.321833)(0.005623,0.299803)(0.010000,0.278571)(0.017783,0.259929)(0.031623,0.245208)(0.056234,0.238587)(0.100000,0.232204)(0.177828,0.235324)(0.316228,0.242129)(0.562341,0.255795)(1.000000,0.275866)(1.778279,0.304368)(3.162278,0.342947)(5.623413,0.401261)(10.000000,0.489728)(17.782794,0.596495)(31.622777,0.708384)(56.234133,0.805123)(100.000000,0.879209)
      };
      \addplot[blue,smooth,line width=0.5pt] coordinates{
	      (0.000100,0.001990)(0.000178,0.002654)(0.000316,0.003542)(0.000562,0.004727)(0.001000,0.006310)(0.001778,0.008427)(0.003162,0.011256)(0.005623,0.015031)(0.010000,0.020050)(0.017783,0.026678)(0.031623,0.035346)(0.056234,0.046527)(0.100000,0.060712)(0.177828,0.078345)(0.316228,0.099725)(0.562341,0.124880)(1.000000,0.153557)(1.778279,0.185677)(3.162278,0.222680)(5.623413,0.269489)(10.000000,0.334874)(17.782794,0.426743)(31.622777,0.542609)(56.234133,0.665616)(100.000000,0.774656)
      };
      \addplot[blue,densely dashed,smooth,line width=0.5pt] coordinates{
	      (0.000100,1.255024)(0.000178,0.980667)(0.000316,0.774793)(0.000562,0.620286)(0.001000,0.504342)(0.001778,0.417404)(0.003162,0.352365)(0.005623,0.303970)(0.010000,0.268369)(0.017783,0.242774)(0.031623,0.225200)(0.056234,0.214281)(0.100000,0.209138)(0.177828,0.209281)(0.316228,0.214517)(0.562341,0.224813)(1.000000,0.240188)(1.778279,0.260839)(3.162278,0.288030)(5.623413,0.325857)(10.000000,0.382056)(17.782794,0.464091)(31.622777,0.570071)(56.234133,0.684321)(100.000000,0.786589)
      };
      \addplot[blue,only marks,mark=o,line width=0.5pt] coordinates{
	      (0.000100,0.002038)(0.000178,0.002722)(0.000316,0.003637)(0.000562,0.004859)(0.001000,0.006490)(0.001778,0.008665)(0.003162,0.011563)(0.005623,0.015409)(0.010000,0.020495)(0.017783,0.027185)(0.031623,0.035902)(0.056234,0.047114)(0.100000,0.061299)(0.177828,0.078900)(0.316228,0.100217)(0.562341,0.125276)(1.000000,0.153827)(1.778279,0.185804)(3.162278,0.222694)(5.623413,0.269465)(10.000000,0.334900)(17.782794,0.426862)(31.622777,0.542794)(56.234133,0.665804)(100.000000,0.774806)
      };
      \addplot[blue,only marks,mark=x,line width=0.5pt] coordinates{
	      (0.000100,1.268876)(0.000178,0.991490)(0.000316,0.783034)(0.000562,0.626562)(0.001000,0.509153)(0.001778,0.421168)(0.003162,0.355485)(0.005623,0.306754)(0.010000,0.270955)(0.017783,0.245174)(0.031623,0.227396)(0.056234,0.216271)(0.100000,0.210927)(0.177828,0.210876)(0.316228,0.215926)(0.562341,0.226057)(1.000000,0.241288)(1.778279,0.261803)(3.162278,0.288859)(5.623413,0.326544)(10.000000,0.382596)(17.782794,0.464484)(31.622777,0.570327)(56.234133,0.684466)(100.000000,0.786659)
      };
      \addplot[blue!75!black,smooth,line width=0.5pt] coordinates{
	      (0.000100,0.000001)(0.000178,0.000004)(0.000316,0.000011)(0.000562,0.000033)(0.001000,0.000095)(0.001778,0.000265)(0.003162,0.000687)(0.005623,0.001630)(0.010000,0.003503)(0.017783,0.006820)(0.031623,0.012138)(0.056234,0.019998)(0.100000,0.030897)(0.177828,0.045262)(0.316228,0.063433)(0.562341,0.085585)(1.000000,0.111624)(1.778279,0.141146)(3.162278,0.173788)(5.623413,0.210429)(10.000000,0.255202)(17.782794,0.316461)(31.622777,0.403210)(56.234133,0.515863)(100.000000,0.639942)
      };
      \addplot[blue!75!black,densely dashed,smooth,line width=0.5pt] coordinates{
	      (0.000100,0.283191)(0.000178,0.282217)(0.000316,0.280540)(0.000562,0.277723)(0.001000,0.273168)(0.001778,0.266232)(0.003162,0.256533)(0.005623,0.244379)(0.010000,0.230936)(0.017783,0.217862)(0.031623,0.206679)(0.056234,0.198421)(0.100000,0.193645)(0.177828,0.192634)(0.316228,0.195580)(0.562341,0.202693)(1.000000,0.214191)(1.778279,0.230239)(3.162278,0.251042)(5.623413,0.277581)(10.000000,0.313327)(17.782794,0.365559)(31.622777,0.442628)(56.234133,0.545325)(100.000000,0.660318)
      };
      \addplot[blue!75!black,only marks,mark=o,line width=0.5pt] coordinates{
	      (0.000100,0.000001)(0.000178,0.000004)(0.000316,0.000011)(0.000562,0.000033)(0.001000,0.000098)(0.001778,0.000271)(0.003162,0.000700)(0.005623,0.001656)(0.010000,0.003547)(0.017783,0.006885)(0.031623,0.012221)(0.056234,0.020096)(0.100000,0.031012)(0.177828,0.045406)(0.316228,0.063616)(0.562341,0.085799)(1.000000,0.111827)(1.778279,0.141291)(3.162278,0.173862)(5.623413,0.210479)(10.000000,0.255318)(17.782794,0.316741)(31.622777,0.403715)(56.234133,0.516543)(100.000000,0.640650)
      };
      \addplot[blue!75!black,only marks,mark=x,line width=0.5pt] coordinates{
	      (0.000100,0.283244)(0.000178,0.282267)(0.000316,0.280587)(0.000562,0.277767)(0.001000,0.273215)(0.001778,0.266300)(0.003162,0.256660)(0.005623,0.244621)(0.010000,0.231339)(0.017783,0.218425)(0.031623,0.207337)(0.056234,0.199071)(0.100000,0.194197)(0.177828,0.193049)(0.316228,0.195877)(0.562341,0.202904)(1.000000,0.214328)(1.778279,0.230298)(3.162278,0.251041)(5.623413,0.277586)(10.000000,0.313437)(17.782794,0.365869)(31.622777,0.443175)(56.234133,0.546039)(100.000000,0.661044)
      };
      \addplot[blue!50!black,smooth,line width=0.5pt] coordinates{
	      (0.000100,0.000000)(0.000178,0.000000)(0.000316,0.000001)(0.000562,0.000003)(0.001000,0.000008)(0.001778,0.000025)(0.003162,0.000074)(0.005623,0.000215)(0.010000,0.000586)(0.017783,0.001484)(0.031623,0.003423)(0.056234,0.007127)(0.100000,0.013407)(0.177828,0.022996)(0.316228,0.036435)(0.562341,0.054050)(1.000000,0.075949)(1.778279,0.101954)(3.162278,0.131569)(5.623413,0.164234)(10.000000,0.200349)(17.782794,0.243226)(31.622777,0.300533)(56.234133,0.381871)(100.000000,0.490286)
      };
      \addplot[blue!50!black,densely dashed,smooth,line width=0.5pt] coordinates{
(0.000100,0.207330)(0.000178,0.207249)(0.000316,0.207105)(0.000562,0.206852)(0.001000,0.206413)(0.001778,0.205664)(0.003162,0.204423)(0.005623,0.202458)(0.010000,0.199567)(0.017783,0.195726)(0.031623,0.191267)(0.056234,0.186908)(0.100000,0.183564)(0.177828,0.182108)(0.316228,0.183241)(0.562341,0.187505)(1.000000,0.195352)(1.778279,0.207161)(3.162278,0.223204)(5.623413,0.243700)(10.000000,0.269389)(17.782794,0.303120)(31.622777,0.351506)(56.234133,0.423326)(100.000000,0.521769)

      };
      \addplot[blue!50!black,only marks,mark=o,line width=0.5pt] coordinates{
(0.000100,0.000000)(0.000178,0.000000)(0.000316,0.000001)(0.000562,0.000003)(0.001000,0.000008)(0.001778,0.000025)(0.003162,0.000075)(0.005623,0.000217)(0.010000,0.000593)(0.017783,0.001499)(0.031623,0.003453)(0.056234,0.007179)(0.100000,0.013486)(0.177828,0.023104)(0.316228,0.036576)(0.562341,0.054233)(1.000000,0.076192)(1.778279,0.102268)(3.162278,0.131944)(5.623413,0.164641)(10.000000,0.200777)(17.782794,0.243704)(31.622777,0.301126)(56.234133,0.382637)(100.000000,0.491190)
      };
      \addplot[blue!50!black,only marks,mark=x,line width=0.5pt] coordinates{
(0.000100,0.208988)(0.000178,0.208905)(0.000316,0.208757)(0.000562,0.208498)(0.001000,0.208049)(0.001778,0.207282)(0.003162,0.206009)(0.005623,0.203996)(0.010000,0.201034)(0.017783,0.197098)(0.031623,0.192529)(0.056234,0.188061)(0.100000,0.184630)(0.177828,0.183122)(0.316228,0.184239)(0.562341,0.188513)(1.000000,0.196373)(1.778279,0.208177)(3.162278,0.224184)(5.623413,0.244618)(10.000000,0.270244)(17.782794,0.303941)(31.622777,0.352352)(56.234133,0.424251)(100.000000,0.522755)
      };

      \addplot[black,line width=1pt] coordinates {(1e-4,0.1689)(1e2,0.1689)};
      \draw[->,>=stealth,semithick] (axis cs:.01,.15) parabola bend (axis cs:1,0.08) (axis cs: 10,0.1);
      \node[font=\footnotesize] at (axis cs:10,.07) {$n=256\to 4\,096$};
      \legend{ { $\bar E_{\rm train}$ }, { $\bar E_{\rm test}$ }, { $E_{\rm train}$ }, { $E_{\rm test}$ } }
    \end{loglogaxis}
  \end{tikzpicture}
  \caption{Neural network performance for growing $n$ ($256$, $512$, $1\,024$, $2\,048$, $4\,096$) as a function of $\gamma$, $\sigma(t)=\max(t,0)$; 2-class MNIST data (sevens, nines), $T=\hat T=1024$, $p=784$. Limiting ($n=\infty$) $\bar E_{\rm test}$ shown in thick black line.}
  \label{fig:perf_limit}
\end{figure}
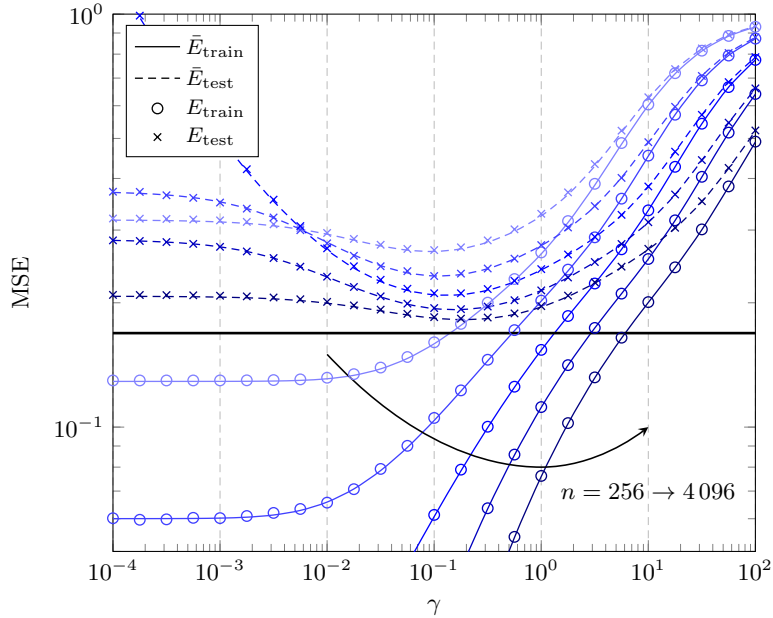

\bigskip

Of course, practical interest lies precisely in situations where $n$ is not too large. We may thus subsequently assume that $\limsup_n n/T<1$. In this case, as suggested by Figures~\ref{fig:perf}--\ref{fig:perf2}, the mean-square error performances achieved as $\gamma\to 0$ may predict the superiority of specific choices of $\sigma(\cdot)$ for optimally chosen $\gamma$. It is important for this study to differentiate between cases where $r\equiv {\rm rank}(\Phi)$ is smaller or greater than $n$. Indeed, observe that, with the spectral decomposition $\Phi=U_r\Lambda_r U_r^\trans$ for $\Lambda_r\in\RR^{r\times r}$ diagonal and $U_r\in\RR^{T\times r}$,
\begin{align*}
	\delta &= \frac1T\tr \Phi\left( \frac{n}T \frac{\Phi}{1+\delta} + \gamma I_T \right)^{-1} = \frac1T \tr \Lambda_r \left( \frac{n}T \frac{\Lambda_r}{1+\delta} + \gamma I_r \right)^{-1}
\end{align*}
which satisfies, as $\gamma\to 0$,
\begin{align*}
	\left\{
	\begin{array}{rll}
		\delta &\to \frac{r}{n-r} &,~r<n \\
		\gamma \delta &\to \Delta = \frac1T\tr \Phi \left( \frac{n}T \frac{\Phi}{\Delta} + I_T \right)^{-1} &,~r\geq n.
	\end{array}
	\right.
\end{align*}
A phase transition therefore exists whereby $\delta$ assumes a finite positive value in the small $\gamma$ limit if $r/n<1$, or scales like $1/\gamma$ otherwise. 

As a consequence, if $r<n$, as $\gamma\to 0$, $\Psi\to \frac{n}T(1-\frac{r}n)\Phi$ and $\bar Q\sim \frac{T}{n-r} U_r \Lambda_r^{-1} U_r^\trans+\frac1\gamma V_rV_r^\trans$, where $V_r\in\RR^{T\times (n-r)}$ is any matrix such that $[U_r~V_r]$ is orthogonal, so that $\Psi\bar Q\to U_rU_r^\trans$ and $\Psi\bar Q^2\to U_r\Lambda_r^{-1}U_r^\trans$; and thus, $\bar{E}_{\rm train}\to \frac1T\tr YV_rV_r^\trans Y^\trans=\frac1T\|YV_r\|^2_F$, which states that the residual training error corresponds to the energy of $Y$ not captured by the space spanned by $\Phi$. Since $E_{\rm train}$ is an increasing function of $\gamma$, so is $\bar E_{\rm train}$ (at least for all large $n$) and thus $\frac1T\|YV_r\|^2_F$ corresponds to the lowest achievable asymptotic training error. 

If instead $r>n$ (which is the most likely outcome in practice), as $\gamma\to 0$, $\bar Q\sim \frac1\gamma(\frac{n}T\frac{\Phi}{\Delta}+I_T)^{-1}$ and thus 
\begin{align*}
	\bar E_{\rm train} \overset{\gamma\to 0}{\longrightarrow} \frac1T \tr YQ_\Delta \left[ \frac{\frac1n\tr \Psi_\Delta Q_\Delta^2}{1-\frac1n\tr (\Psi_\Delta Q_\Delta)^2}\Psi_\Delta + I_T\right] Q_\Delta Y^\trans
\end{align*}
where $\Psi_\Delta=\frac{n}T\frac{\Phi}{\Delta}$ and $Q_\Delta=(\frac{n}T\frac{\Phi}{\Delta}+I_T)^{-1}$.

These results suggest that neural networks should be designed both in a way that reduces the rank of $\Phi$ while maintaining a strong alignment between the dominant eigenvectors of $\Phi$ and the output matrix $Y$.

Interestingly, if $X$ is assumed as above to be extracted from a Gaussian mixture and that $Y\in\RR^{1\times T}$ is a classification vector with $Y_{1j}\in\{-1,1\}$, then the tools proposed in \citep{COU16} (related to spike random matrix analysis) allow for an explicit evaluation of the aforementioned limits as $n,p,T$ grow large. This analysis is however cumbersome and outside the scope of the present work.

\section{Proof of the Main Results}
\label{sec:proofs}

In the remainder, we shall use extensively the following notations:
\begin{align*}
	\Sigma &= \sigma(WX) = \begin{bmatrix} \sigma_1^\trans \\ \vdots \\ \sigma_n^\trans \end{bmatrix},\quad W = \begin{bmatrix} w_1^\trans \\ \vdots \\ w_n^\trans \end{bmatrix} 
\end{align*}
i.e., $\sigma_i=\sigma(w_i^\trans X)^\trans$. 
Also, we shall define $\Sigma_{-i}\in\RR^{(n-1)\times T}$ the matrix $\Sigma$ with $i$-th row removed, and correspondingly
\begin{align*}
	Q_{-i} &= \left( \frac1T\Sigma^\trans \Sigma - \frac1T \sigma_i\sigma_i^\trans + \gamma I_T \right)^{-1}.
\end{align*}
Finally, because of exchangeability, it shall often be convenient to work with the generic random vector $w\sim\mathcal N_\varphi(0,I_T)$, the random vector $\sigma$ distributed as any of the $\sigma_i$'s, the random matrix $\Sigma_-$ distributed as any of the $\Sigma_{-i}$'s, and with the random matrix $Q_-$ distributed as any of the $Q_{-i}$'s.

%
%

\subsection{Concentration Results on $\Sigma$}

Our first results provide concentration of measure properties on functionals of $\Sigma$.
\textcolor{black}{These results unfold from the following concentration inequality for Lipschitz applications of a Gaussian vector; see e.g., \cite[Corollary~2.6, Propositions~1.3,~1.8]{LED05} or \cite[Theorem~2.1.12]{TAO12}. For $d \in \mathbb{N}$, consider $\mu$ the canonical Gaussian probability on $\mathbb{R}^d$ defined through its density $d\mu(w)=(2 \pi)^{-\frac{d}2}e^{- \frac12\|w\|^2}$ and $f : \mathbb{R}^d \rightarrow \mathbb{R}$ a $\lambda_f$-Lipschitz function. Then, we have the said {\it normal concentration}
	\begin{align}
		\label{eq:conc_lip}
		\mu\left( \left\{ \left| f - \int f d\mu \right| \geq t \right\} \right) &\leq  Ce^{-c\frac{t^2}{\lambda_f^2}}
	\end{align}
	where $C,c>0$ are independent of $d$ and $\lambda_f$. As a corollary (see e.g., \cite[Proposition~1.10]{LED05}), for every $k\geq 1$,
	\begin{align*}
		\EE\left[ \left| f - \int f d\mu \right|^k \right] &\leq \left(\frac{C\lambda_f}{\sqrt{c}}\right)^k.
	\end{align*}
}


\bigskip

The main approach \textcolor{black}{to the proof of our results, starting with that of the key Lemma~\ref{lem:concentration_quadform},} is as follows: since $W_{ij}=\varphi(\tilde W_{ij})$ with $\tilde W_{ij}\sim\mathcal N(0,1)$ and $\varphi$ Lipschitz, the normal concentration of $\tilde W$ transfers to $W$ which further induces a normal concentration of the random vector $\sigma$ and the matrix $\Sigma$, thereby implying that Lipschitz functionals of $\sigma$ or $\Sigma$ also concentrate. \textcolor{black}{As pointed out earlier,} these concentration results are used in place for the independence assumptions (and their multiple consequences on convergence of random variables) classically exploited in random matrix theory.
	
\bigskip

{\it Notations:} In all subsequent lemmas and proofs, the letters $c,c_i,C,C_i>0$ will be used interchangeably as positive constants independent of the key equation parameters (notably $n$ and $t$ below) and may be reused from line to line. Additionally, the variable $\varepsilon>0$ will denote any small positive number; the variables $c,c_i,C,C_i$ may depend on $\varepsilon$.

\bigskip

We start \textcolor{black}{by recalling the first part of the statement of Lemma~\ref{lem:concentration_quadform} and subsequently providing its proof.} 

\begin{lemma}[Concentration of quadratic forms]
	\label{lem:concentration_quadform_1stpart}
	Let \textcolor{black}{Assumptions~\ref{ass:W}--\ref{ass:sigma}} hold. \textcolor{black}{Let also $A \in \RR^{T\times T}$ such that $\| A \| \leq 1$} and, for $X\in\RR^{p\times T}$ and $w\sim \mathcal N_\varphi(0,I_p)$, define the random vector $\sigma\equiv \sigma(w^\trans X)^\trans\in\RR^T$. Then, 
	\begin{align*}
		\textcolor{black}{P\left( \left| \frac1T\sigma^\trans A \sigma - \frac1T\tr \Phi A  \right| > t \right) } & \textcolor{black}{\leq C e^{-\frac{cT}{\|X \|^2 \lambda_{\varphi}^2\lambda_{\sigma}^2} \min\left(\frac{t^2}{t_0^2}, t\right)}}
	\end{align*}
	\textcolor{black}{for $t_0\equiv |\sigma(0)| + \lambda_\varphi\lambda_{\sigma} \|X \| \sqrt{\frac{p}T}$ and $C,c>0$ independent of all other parameters.}
\end{lemma}
\begin{proof}
	The layout of the proof is as follows: since the application $w\mapsto \frac1T\sigma^\trans A\sigma$ is ``quadratic'' in $w$ and thus not Lipschitz (therefore not allowing for a natural transfer of the concentration of $w$ to $\frac1T\sigma^\trans A\sigma$), we first prove that $\frac1{\sqrt{T}}\|\sigma\|$ satisfies a concentration inequality, which provides a high probability $O(1)$ bound on $\frac1{\sqrt{T}}\|\sigma\|$. Conditioning on this event, the map $w\mapsto \frac1{\sqrt{T}}\sigma^\trans A\sigma$ can then be shown to be Lipschitz (by isolating one of the $\sigma$ terms for bounding and the other one for retrieving the Lipschitz character) and\textcolor{black}{, up to an appropriate control of concentration results under conditioning,} the result is obtained.

	Following this plan, we first provide a concentration inequality for $\|\sigma\|$. To this end, note that the application \textcolor{black}{$\psi:\RR^p\to\RR^T$, $\tilde w\mapsto \sigma(\varphi(\tilde w)^\trans X)^\trans$} is Lipschitz with parameter \textcolor{black}{$\lambda_{\varphi}\lambda_\sigma \|X\|$} as the combination of the \textcolor{black}{$\lambda_{\varphi}$-Lipschitz function $\varphi : \tilde w \mapsto w$}, the \textcolor{black}{$\|X\|$}-Lipschitz map $\RR^n\to\RR^T$, $w\mapsto X^\trans w$ and the $\lambda_\sigma$-Lipschitz map $\RR^T\to\RR^T$, $Y\mapsto \sigma(Y)$. \textcolor{black}{As a Gaussian vector, $\tilde w$} has a normal concentration and so does $\psi(\tilde{w})$. Since the Euclidean norm $\RR^T\to\RR$, $Y\mapsto \|Y\|$ is $1$-Lipschitz, we thus have immediately by \eqref{eq:conc_lip}
	\begin{align*}
		\textcolor{black}{P\left( \left| \left\| \frac1{\sqrt{T}}\sigma(w^\trans X) \right\| - \EE\left[  \left\| \frac1{\sqrt{T}}\sigma(w^\trans X) \right\| \right] \right| \geq t \right) \leq C e^{-\frac{c T t^2}{\|X\|^2 \lambda_{\sigma}^2\lambda_{\varphi}^2}}}
	\end{align*}
	for some $c,C>0$ independent of all parameters. 
	
	Finally, using again the Lipschitz character of $\sigma(w^\trans X)$,
	\begin{align*}
		\left| \left\|\sigma(w^\trans X)\right\| - \left\| \sigma(0) 1_T^\trans \right\| \right| &\leq \left\|\sigma(w^\trans X) - \sigma(0) 1_T^\trans \right\| \leq \lambda_\sigma \left\| w\right\| \cdot \left\| X\right\|
	\end{align*}
	so that, by Jensen's inequality,
	\begin{align*}
		\EE\left[  \left\| \frac1{\sqrt{T}}\sigma(w^\trans X) \right\| \right] &\leq |\sigma(0)| + \lambda_\sigma \EE \left[ \frac1{\sqrt{T}} \left\| w\right\| \right] \left\| X\right\| \\
		&\leq |\sigma(0)| + \lambda_\sigma \sqrt{\EE \left[ \frac1T\|w\|^2 \right]} \|X\|
	\end{align*}
	\textcolor{black}{with $\EE[\|\varphi(\tilde{w})\|^2]\leq \lambda_\varphi^2\EE[\|\tilde{w}\|^2]=p\lambda_\varphi^2$ (since $\tilde w\sim \mathcal N(0,I_p)$). Letting $t_0\equiv |\sigma(0)| + \lambda_\sigma \lambda_\varphi \|X\|\sqrt{\frac{p}T}$, we then find
	\begin{align*}
		P\left( \left\| \frac1{\sqrt{T}}\sigma(w^\trans X) \right\| \geq t + t_0 \right) &\leq C e^{-\frac{cTt^2}{\lambda_{\varphi}^2\lambda_{\sigma}^2 \|X\|^2}}
	\end{align*}}%
which, with the remark $t\geq 4 t_0 \Rightarrow (t-t_0)^2 \geq t^2/2$, may be equivalently stated as
	\textcolor{black}{\begin{align}
		\label{eq:sigma_bound}
		\forall t\geq 4t_0,~P\left( \left\| \frac1{\sqrt{T}}\sigma(w^\trans X) \right\| \geq t \right) &\leq C e^{-\frac{cTt^2}{2\lambda_{\varphi}^2\lambda_{\sigma}^2 \|X\|^2}}.
	\end{align}}
	As a side (but important) remark, note that, since 
	\begin{align*}
		P\left( \left\| \frac{\Sigma}{\sqrt{T}} \right\|_F \geq t \sqrt{T} \right) 
		&= P\left( \sqrt{\sum_{i=1}^n \left\| \frac{\sigma_i}{\sqrt{T}} \right\|^2} \geq t \sqrt{T} \right) \\
		&\leq P\left( \max_{1\leq i\leq n} \left\| \frac{\sigma_i}{\sqrt{T}} \right\| \geq \sqrt{\frac{T}n} t \right) \\
		&\leq n P\left( \left\| \frac{\sigma}{\sqrt{T}} \right\| \geq \sqrt{\frac{T}n} t \right)
	\end{align*}
	the result above implies that 
	\begin{align*}
		\textcolor{black}{\forall t\geq 4t_0,~P\left( \left\| \frac{\Sigma}{\sqrt{T}}\right\|_F \geq t \sqrt{T} \right) \leq Cn e^{-\frac{cT^2t^2}{2n\lambda_{\varphi}^2\lambda_{\sigma}^2 \|X\|^2}}}
	\end{align*}
	and thus, since $\|\cdot\|_F\geq \|\cdot\|$, we have
	\begin{align*}
		\forall t\geq 4t_0,~P\left( \left\| \frac{\Sigma}{\sqrt{T}}\right\| \geq t \sqrt{T} \right) &\leq Cn e^{-\frac{cT^2t^2}{2n\lambda_{\varphi}^2\lambda_{\sigma}^2 \|X\|^2}}
	\end{align*}
	Thus, in particular, under the additional Assumption~\ref{ass:growth}, with high probability, the operator norm of $\frac{\Sigma}{\sqrt{T}}$ cannot exceed a rate $\sqrt{T}$. 

	\begin{remark}[Loss of control of the structure of $\Sigma$]
		\label{rem:Sigma}
		The aforementioned control of $\|\Sigma\|$ arises from the bound $\|\Sigma\|\leq \|\Sigma\|_F$ which may be quite loose (by as much as a factor $\sqrt{T}$). Intuitively, \textcolor{black}{under the supplementary Assumption~\ref{ass:growth},} if $\EE[\sigma]\neq 0$, then $\frac{\Sigma}{\sqrt{T}}$ is ``dominated'' by the matrix $\frac1{\sqrt{T}}\EE[\sigma]1_T^\trans$, the operator norm of which is indeed of order $\sqrt{n}$ and the bound is tight. If $\sigma(t)=t$ and $\EE[W_{ij}]=0$, we however know that $\|\frac{\Sigma}{\sqrt T}\|=O(1)$ \citep{SIL98}. One is tempted to believe that, more generally, if $\EE[\sigma]=0$, then $\|\frac{\Sigma}{\sqrt T}\|$ should remain of this order. And, if instead $\EE[\sigma]\neq 0$, the contribution of $\frac1{\sqrt{T}}\EE[\sigma]1_T^\trans$ should merely engender a single large amplitude isolate singular value in the spectrum of $\frac{\Sigma}{\sqrt T}$ and the other singular values remain of order $O(1)$. These intuitions are not captured by our concentration of measure approach. 
		
		Since $\Sigma=\sigma(WX)$ is an entry-wise operation, concentration results with respect to the Frobenius norm are natural, where with respect to the operator norm are hardly accessible. 
	\end{remark}
	
	Back to our present considerations, let us define the probability space $\mathcal A_K=\{w,~\|\sigma(w^\trans X)\|\leq K\sqrt{T}\}$. 
	\textcolor{black}{
		Conditioning the random variable of interest in Lemma~\ref{lem:concentration_quadform_1stpart} with respect to $\mathcal{A}_K$ and its complementary $\mathcal{A}_K^c$, for some $K\geq 4t_0$, gives
	\begin{align*}
		&P\left( \left| \frac1T\sigma(w^\trans X) A \sigma(w^\trans X)^\trans - \frac1T\tr \Phi A  \right| > t \right) \\
		&\leq P\left( \left\{\left| \frac1T\sigma(w^\trans X) A \sigma(w^\trans X)^\trans - \frac1T\tr \Phi A  \right| > t \right\}, \mathcal{A}_K\right) + P(\mathcal{A}_K^c).
	\end{align*}
	We can already bound $P(\mathcal A_K^c)$ thanks to \eqref{eq:sigma_bound}. As for the first right-hand side term, note that on the set $\{ \sigma(w^\trans X), w \in \mathcal{A}_K\}$, the function $f : \mathbb{R}^T \rightarrow \mathbb{R} : \ \sigma\mapsto \sigma^\trans A \sigma$ is $K \sqrt{T}$-Lipschitz. This is because, for all $\sigma,\sigma+h\in \{ \sigma(w^\trans X), w \in \mathcal{A}_K\}$,
	\begin{align*}
		\left\| f(\sigma+h) - f(\sigma) \right\| &= \left\| h^\trans A \sigma + (\sigma+h)^\trans A h \right\| \leq K\sqrt{T} \left\| h \right\|.
	\end{align*}
	Since conditioning does not allow for a straightforward application of \eqref{eq:conc_lip}, we consider instead $\tilde{f}$, a $K\sqrt{T}$-Lipschitz continuation to $\RR^T$ of $f_{\mathcal{A}_K}$, the restriction of $f$ to $\mathcal A_K$, such that all the radial derivative of $\tilde{f}$ are constant in the set $\{ \sigma, \| \sigma \| \geq K \sqrt{T}\}$. We may thus now apply \eqref{eq:conc_lip} and our previous results to obtain
	\begin{align*}
		P\left(\left|\tilde f (\sigma(w^\trans X))- \EE [\tilde{f}(\sigma(w^\trans X))]\right| \geq K T t\right) \leq e^{-\frac{cT t^2}{\|X\|^2 \lambda_{\sigma}^2\lambda_{\varphi}^2}}.
	\end{align*}
	Therefore,
	\begin{align*}
		&P\left(\left\{\left|f(\sigma(w^\trans X))- \EE [\tilde{f}(\sigma(w^\trans X))]\right| \geq K T t\right\}, \mathcal{A}_K\right) \\
		&= P\left(\left\{\left|\tilde f(\sigma(w^\trans X))- \EE [\tilde{f}(\sigma(w^\trans X))]\right| \geq K T t\right\}, \mathcal{A}_K\right) \\
		&\leq P\left(\left|\tilde f(\sigma(w^\trans X))- \EE [\tilde{f}(\sigma(w^\trans X))]\right| \geq K T t\right) \leq e^{-\frac{cT t^2}{\|X\|^2 \lambda_{\sigma}^2\lambda_{\varphi}^2}}.
	\end{align*}
	Our next step is then to bound the difference $\Delta=| \EE[\tilde{f}(\sigma(w^\trans X))]- \EE [f(\sigma(w^\trans X))]|$. Since $f$ and $\tilde f$ are equal on $\{\sigma, \|\sigma\|\leq K\sqrt{T}\}$,
	\begin{align*}
		\Delta \leq \int_{\| \sigma \| \geq K \sqrt{T}} \left(|f(\sigma)| + |\tilde{f}(\sigma)|\right) d\mu_{\sigma}(\sigma)
	\end{align*}
	where $\mu_{\sigma}$ is the law of $\sigma(w^\trans X)$. Since $\|A\| \leq 1$, for $\|\sigma \| \geq K \sqrt{T}$, $\max(|f(\sigma) |,|\tilde{f}(\sigma) |) \leq \|\sigma \|^2$ and thus
	\begin{align*}
		\Delta
		&\leq 2 \int_{\|\sigma\| \geq K \sqrt{T}} \|\sigma \|^2 d\mu_{\sigma} = 2 \int_{\|\sigma\| \geq K \sqrt{T}} \int_{t=0}^{\infty} \mathds{1}_{\|\sigma\|^2\geq t} dt d\mu_{\sigma}\\
		&= 2 \int_{t=0}^{\infty} P\left(\left\{\|\sigma\|^2 \geq t \right\}, \mathcal{A}_K^c \right) dt\\
		&\leq 2\int_{t=0}^{K^2 T}P(\mathcal{A}_K^c)dt +2\int_{t=K^2 T}^{\infty} P(\|\sigma(w^\trans X) \|^2 \geq t) dt \\
		& \leq 2 P(\mathcal{A}_K^c) K^2 T + 2 \int_{t=K^2 T}^{\infty} C e^{- \frac{ct}{2\lambda_{\varphi}^2\lambda_{\sigma}^2 \|X\|^2}} dt\\
		&\leq 2 C TK^2 e^{-\frac{cTK^2}{2\lambda_{\varphi}^2\lambda_{\sigma}^2 \|X\|^2}}+\frac{2 C \lambda_{\varphi}^2\lambda_{\sigma}^2 \|X\|^2}{c} e^{- \frac{cTK^2}{2\lambda_{\varphi}^2\lambda_{\sigma}^2 \|X\|^2}} \\
		&\leq \frac{6C}{c}\lambda_{\varphi}^2\lambda_{\sigma}^2 \|X\|^2
	\end{align*}
	where in last inequality we used the fact that for $x\in\RR$, $x e^{-x} \leq e^{-1} \leq 1$, and $K\geq 4t_0\geq 4\lambda_\sigma\lambda_\varphi\|X\|\sqrt{\frac{p}T}$. As a consequence,
	\begin{align*}
		P\left(\left\{\left|f(\sigma(w^\trans X))- \EE [f(\sigma(w^\trans X))]\right| \geq K T t +\Delta\right\}, \mathcal{A}_K\right) \leq C e^{-\frac{cT t^2}{\|X\|^2 \lambda_{\varphi}^2\lambda_{\sigma}^2}}
	\end{align*}
	so that, with the same remark as before, for $t \geq \frac{4\Delta}{K T}$,
	\begin{align*}
	 	P\left(\left\{\left|f(\sigma(w^\trans X))- \EE [f(\sigma(w^\trans X))]\right| \geq K T t \right\}, \mathcal{A}_K\right) \leq C e^{-\frac{cT t^2}{2\|X\|^2 \lambda_{\varphi}^2\lambda_{\sigma}^2}}.
	\end{align*}
	To avoid the condition $t \geq \frac{4\Delta}{KT}$, we use the fact that, probabilities being lower than one, it suffices to replace $C$ by $\lambda C$ with $\lambda \geq 1$ such that
	\begin{align*}
		\lambda C e^{-c\frac{T t^2}{2\|X\|^2 \lambda_{\varphi}^2\lambda_{\sigma}^2}} \geq 1 \ \ \text{for} \ t \leq \frac{4\Delta}{K T}.
	\end{align*}
	The above inequality holds if we take for instance $\lambda=\frac1Ce^{\frac{ 18 C^2}{c}}$ since then $t\leq \frac{4\Delta}{KT}\leq\frac{24C\lambda_\varphi^2\lambda_\sigma^2\|X\|^2}{cKT}\leq \frac{6C\lambda_\varphi\lambda_\sigma\|X\|}{c\sqrt{pT}}$ (using successively $\Delta\geq \frac{6C}c \lambda_{\varphi}^2\lambda_{\sigma}^2 \|X\|^2$ and $K\geq 4\lambda_\sigma\lambda_\varphi\|X\|\sqrt{\frac{p}T}$) and thus
	\begin{align*}
		\lambda C e^{-\frac{cTt^2}{2\|X\|^2 \lambda_{\varphi}^2\lambda_{\sigma}^2}} \geq \lambda C e^{-\frac{ 18C^2}{cp}} \geq \lambda C e^{-\frac{ 18C^2}{c}} \geq 1.
	\end{align*}
	Therefore, setting $\lambda=\max(1,\frac1C e^{\frac{{C'}^2c}2})$, we get for every $t>0$
	\begin{align*}
		P\left(\left\{\left|f(\sigma(w^\trans X)- \EE [f(\sigma(w^\trans X)]\right| \geq K T t \right\}, \mathcal{A}_K\right) \leq \lambda C e^{-\frac{cT t^2}{2\|X\|^2 \lambda_{\varphi}^2\lambda_{\sigma}^2}}
	\end{align*}
	which, together with the inequality $P(\mathcal{A}_K^c) \leq C e^{-\frac{cTK^2}{2\lambda_{\varphi}^2\lambda_{\sigma}^2 \|X\|^2}}$, gives
	\begin{align*}
		P\left(\left|f(\sigma(w^\trans X)- \EE [f(\sigma(w^\trans X)]\right| \geq K T t \right) 
		&\leq \lambda C e^{-\frac{T ct^2}{2\|X\|^2 \lambda_{\varphi}^2\lambda_{\sigma}^2}}  + C e^{-\frac{cTK^2}{2\lambda_{\varphi}^2\lambda_{\sigma}^2 \|X\|^2}}.
	\end{align*}
	We then conclude
	\begin{align*}
		&P\left(\left|\frac1T\sigma(w^\trans X)A\sigma(w^\trans X)^\trans- \frac1T\tr (\Phi A)\right|\geq   t \right) \\
		&\leq (\lambda+1)C e^{-\frac{cT}{2 \|X \|^2 \lambda_{\varphi}^2\lambda_{\sigma}^2} \min(t^2/K^2, K^2)}
	\end{align*}
	and, with $K=\max(4t_0,\sqrt{t})$,
	\begin{align*}
		P\left(\left|\frac1T\sigma(w^\trans X)A\sigma(w^\trans X)^\trans- \frac1T\tr (\Phi A)\right|\geq   t \right) \leq (\lambda+1)C e^{-\frac{cT \min\left( \frac{t^2}{16t_0^2}, t\right)}{2 \|X \|^2 \lambda_{\varphi}^2\lambda_{\sigma}^2}}.
	\end{align*}
	Indeed, if $4t_0 \leq \sqrt{t}$ then $\min(t^2/K^2,K^2)=t$, while if $4t_0 \geq \sqrt{t}$ then $\min(t^2/K^2,K^2)=\min(t^2/16t_0^2,16t_0^2)=t^2/16t_0^2$.
	}
\end{proof}

As a corollary of Lemma~\ref{lem:concentration_quadform_1stpart}, we have the following control of the moments of $\frac1T\sigma^\trans A\sigma$.
\begin{corollary}[Moments of quadratic forms]
	\label{cor:moments_quadform}
	Let Assumptions~\ref{ass:W}--\ref{ass:sigma} hold. For $w\sim \mathcal N_\varphi(0,I_p)$, $\sigma\equiv \sigma(w^\trans X)^\trans\in\RR^T$, \textcolor{black}{$A \in \RR^{T\times T}$ such that $\| A \| \leq 1$}, and $k\in\NN$,
	\textcolor{black}{
	\begin{align*}
		\EE \left[ \left| \frac1T\sigma^\trans A \sigma - \frac1T\tr \Phi A  \right|^k \right] &\leq C_1 \left(\frac{t_0 \eta}{\sqrt{T}}\right)^k + C_2\left(\frac{\eta^2}{T}\right)^k
	\end{align*}
	with $t_0=|\sigma(0)| + \lambda_\sigma \lambda_\varphi \|X\|\sqrt{\frac{p}T}$, $\eta = \|X\| \lambda_\sigma \lambda_\varphi$, and $C_1,C_2>0$ independent of the other parameters. In particular, under the additional Assumption~\ref{ass:growth},
	\begin{align*}
		\EE \left[ \left| \frac1T\sigma^\trans A \sigma - \frac1T\tr \Phi A  \right|^k \right] &\leq  \frac{C}{n^{k/2}} 
	\end{align*}}
\end{corollary}
\begin{proof}
	We use the fact that, for a nonnegative random variable $Y$, $\EE[Y]=\int_0^\infty P(Y>t)dt$, so that
	\begin{align*}
		&\EE \left[ \left| \frac1T\sigma^\trans A \sigma - \frac1T\tr \Phi A  \right|^k \right] \\ 
		& \ \ \ \ = \int_0^\infty P\left( \left| \frac1T\sigma^\trans A \sigma - \frac1T\tr \Phi A  \right|^k > u \right) du \\
		& \ \ \ \ = \int_0^\infty kv^{k-1} P\left( \left| \frac1T\sigma^\trans A \sigma - \frac1T\tr \Phi A  \right| > v \right) dv \\
		& \ \ \ \ \textcolor{black}{\leq \int_0^\infty kv^{k-1}C e^{-\frac{cT}{\eta^2} \min\left(\frac{v^2}{t_0^2},v\right)} dv	} \\
		& \ \ \ \ \textcolor{black}{\leq \int_0^{t_0} kv^{k-1}C e^{-\frac{cTv^2}{t_0^2\eta^2}} dv +\int_{t_0}^{\infty}kv^{k-1} Ce^{-\frac{cTv}{\eta^2}}} dv \\
		& \ \ \ \ \textcolor{black}{\leq \int_0^{\infty} kv^{k-1}C e^{-\frac{cTv^2}{t_0^2\eta^2}}dv +\int_{0}^{\infty}kv^{k-1} Ce^{-\frac{cTv}{\eta^2} }dv} \\
		& \ \ \ \ \textcolor{black}{= \left(\frac{t_0 \eta}{\sqrt{cT}}\right)^k\int_0^{\infty} kt^{k-1}C e^{-t^2}dt + \left(\frac{\eta^2}{cT}\right)^k\int_{0}^{\infty}kt^{k-1} C e^{-t}dt}
	\end{align*}
	which, along with the boundedness of the integrals, concludes the proof.
\end{proof}

\medskip

Beyond concentration results on functions of the vector $\sigma$, we also have the following convenient property for functions of the matrix $\Sigma$.
\begin{lemma}[Lipschitz functions of $\Sigma$]
	\label{lem:Lip_Sigma}
	Let $f:\RR^{n\times T}\to \RR$ be a $\lambda_f$-Lipschitz function with respect to the Froebnius norm. Then, under Assumptions~\ref{ass:W}--\ref{ass:sigma},
	\begin{align*}
		\textcolor{black}{P\left( \left| f\left( \frac{\Sigma}{\sqrt{T}} \right) - \EE f\left( \frac{\Sigma}{\sqrt{T}} \right) \right| > t \right) \leq Ce^{-\frac{cTt^2}{\lambda_\sigma^2\lambda_\varphi^2 \lambda_f^2 \|X\|^2}}}
	\end{align*}
	for some $C,c>0$. \textcolor{black}{In particular, under the additional Assumption~\ref{ass:growth},}
	\begin{align*}
		\textcolor{black}{P\left( \left| f\left( \frac{\Sigma}{\sqrt{T}} \right) - \EE f\left( \frac{\Sigma}{\sqrt{T}} \right) \right| > t \right) \leq Ce^{-cTt^2}.}
	\end{align*}
\end{lemma}
\begin{proof}
	Denoting $W=\varphi(\tilde W)$, since ${\rm vec}(\tilde W)\equiv\textcolor{black}{[\tilde W_{11},\cdots,\tilde W_{np}]}$ is a Gaussian vector, by the normal concentration of Gaussian vectors, for $g$ a $\lambda_g$-Lipschitz function of $W$ with respect to the Frobenius norm (i.e., the Euclidean norm of ${\rm vec}(W)$), by \eqref{eq:conc_lip},
	\begin{align*}
		P\left( \left| g(W) - \EE[g(W)] \right| > t \right) = P\left( \left| g(\varphi(\tilde W)) - \EE[g(\varphi(\tilde W))] \right| > t \right) &\leq Ce^{-\frac{ct^2}{\lambda_g^2\lambda_\varphi^2}}
	\end{align*}
	for some $C,c>0$. Let's \textcolor{black}{consider} in particular $g:W\mapsto f(\Sigma/\sqrt{T})$ and remark that
	\begin{align*}
		\left| g(W+H)-g(W) \right| &= \left| f\left( \frac{\sigma( (W+H)X )}{\sqrt{T}} \right) - f\left( \frac{\sigma( WX )}{\sqrt{T}} \right) \right| \\
		&\leq \frac{\lambda_f}{\sqrt{T}} \left\| \sigma( (W+H)X ) - \sigma( WX ) \right\|_F \\
		&\leq \frac{\lambda_f\lambda_\sigma}{\sqrt{T}} \left\| HX \right\|_F \\
		&= \frac{\lambda_f\lambda_\sigma}{\sqrt{T}} \sqrt{ \tr HXX^\trans H^\trans } \\
		&\leq \frac{\lambda_f\lambda_\sigma}{\sqrt{T}} \sqrt{\left\| XX^\trans \right\|} \|H\|_F
	\end{align*}
	\textcolor{black}{concluding the proof.}
\end{proof}

A first corollary of Lemma~\ref{lem:Lip_Sigma} is the concentration of the Stieltjes transform $\frac1T\tr \left( \frac1T\Sigma^\trans \Sigma -z I_T \right)^{-1}$ of $\mu_n$, the empirical spectral measure of $\frac1T\Sigma^\trans \Sigma$, for all $z\in\CC\setminus\RR^+$ (so in particular, for $z=-\gamma$, $\gamma>0$).
\begin{corollary}[Concentration of the Stieltjes transform of $\mu_n$]
	\label{cor:ST}
	Under \textcolor{black}{Assumptions~\ref{ass:W}--\ref{ass:sigma}, for $z\in\CC\setminus \RR^+$,
	\begin{align*}
		&P\left( \left| \frac1T\tr \left( \frac1T\Sigma^\trans \Sigma -z I_T \right)^{-1} - \EE\left[ \frac1T\tr\left( \frac1T\Sigma^\trans \Sigma -z I_T \right)^{-1} \right] \right| > t \right) \nonumber \\
		&\leq Ce^{-\frac{c{\rm dist}(z,\RR^+)^2Tt^2}{\lambda_\sigma^2\lambda_\varphi^2 \|X\|^2}}
	\end{align*}
	for some $C,c>0$, where ${\rm dist}(z,\RR^+)$ is the Hausdorff set distance. In particular, for $z=-\gamma$, $\gamma>0$, and under the additional Assumption~\ref{ass:growth}}
	\begin{align*}
		P\left( \left| \frac1T\tr Q - \frac1T\tr \EE[Q] \right| > t \right) &\leq Ce^{-cnt^2}.
	\end{align*}
\end{corollary}
\begin{proof}
	We can apply Lemma~\ref{lem:Lip_Sigma} for $f:R\mapsto \frac1T\tr (R^\trans R-zI_T)^{-1}$, since we have
	\begin{align*}
		&\left| f(R+H) - f(R) \right| \nonumber \\
		&= \left| \frac1T\tr ( (R+H)^\trans (R+H)-zI_T)^{-1} ( (R+H)^\trans H + H^\trans R ) (R^\trans R-zI_T)^{-1} \right| \\
		&\leq \left| \frac1T\tr ( (R+H)^\trans (R+H)-zI_T)^{-1} (R+H)^\trans H (R^\trans R-zI_T)^{-1} \right| \nonumber \\ 
		&+ \left| \frac1T\tr ( (R+H)^\trans (R+H)-zI_T)^{-1}H^\trans R(R^\trans R-zI_T)^{-1} \right| \\
		&\leq \frac{2\|H\|}{ {\rm dist}(z,\RR^+)^{\frac32}} \leq \frac{2\|H\|_F}{ {\rm dist}(z,\RR^+)^{\frac32}}
	\end{align*}
	where, for the second to last inequality, we successively used the relations $|\tr AB|\leq \sqrt{\tr AA^\trans}\sqrt{\tr BB^\trans}$, $|\tr CD|\leq \|D\| \tr C$ for nonnegative definite $C$, and $\|(R^\trans R-zI_T)^{-1}\|\leq {\rm dist}(z,\RR^+)^{-1}$, $\|(R^\trans R-zI_T)^{-1}R^\trans R\|\leq 1$, $\|(R^\trans R-zI_T)^{-1}R^\trans \|=\|(R^\trans R-zI_T)^{-1}R^\trans R (R^\trans R-zI_T)^{-1}\|^{\frac12}\leq \|(R^\trans R-zI_T)^{-1}R^\trans R\|^\frac12 \| (R^\trans R-zI_T)^{-1}\|^{\frac12}\leq {\rm dist}(z,\RR^+)^{-\frac12}$, for $z\in\CC\setminus\RR^+$, and finally $\|\cdot\|\leq \|\cdot\|_F$. 
\end{proof}

Lemma~\ref{lem:Lip_Sigma} also allows for an important application of Lemma~\ref{lem:concentration_quadform_1stpart} as follows.
\begin{lemma}[Concentration of $\frac1T\sigma^\trans Q_-\sigma$]
	\label{lem:concentration_sQs}
	Let Assumptions~\ref{ass:W}--\ref{ass:growth} hold and write $W^\trans=[w_1,\ldots,w_n]$. Define $\sigma\equiv \sigma(w_1^\trans X)^\trans\in\RR^T$ and, for $W_-^\trans=[w_2,\ldots,w_n]$ and $\Sigma_-=\sigma(W_-X)$, let $Q_-=(\frac1T\Sigma_-^\trans\Sigma_-+\gamma I_T)^{-1}$. Then, for $A,B\in\RR^{T\times T}$ such that \textcolor{black}{$\|A\|, \|B\| \leq 1$} 
	\begin{align*}
		P\left( \left| \frac1T\sigma^\trans AQ_-B \sigma - \frac1T\tr \Phi A\EE[Q_-]B  \right| > t \right) &\leq C e^{-cn \min (t^2,t)}
	\end{align*}
	for some $C,c>0$ independent of the other parameters.
\end{lemma}
\begin{proof}
	Let $f:R\mapsto \frac1T\sigma^\trans A(R^\trans R+\gamma I_T)^{-1}B\sigma$. Reproducing the proof of Corollary~\ref{cor:ST}, conditionally to $\frac1T\|\sigma\|^2\leq K$ for any arbitrary large enough $K>0$, it appears that $f$ is Lipschitz with parameter of order $O(1)$. Along with \eqref{eq:sigma_bound} \textcolor{black}{and Assumption~\ref{ass:growth}}, this thus ensures that
	\begin{align*}
		&P\left( \left| \frac1T\sigma^\trans AQ_-B \sigma - \frac1T\sigma^\trans A\EE[Q_-]B \sigma  \right| > t \right) \nonumber \\
		&\leq P\left( \left| \frac1T\sigma^\trans AQ_-B \sigma - \frac1T\sigma^\trans A\EE[Q_-]B \sigma  \right| > t,~\frac{\|\sigma\|^2}T\leq K \right) + P\left( \frac{\|\sigma\|^2}T> K \right) \\
		&\leq C e^{-cn t^2}
	\end{align*}
	for some $C,c>0$. We may then apply Lemma~\ref{lem:concentration_quadform} on the bounded norm matrix $A\EE[Q_-]B$ to further find that
	\begin{align*}
		&P\left( \left| \frac1T\sigma^\trans AQ_-B \sigma - \frac1T\tr \Phi A\EE[Q_-]B \right| > t \right) \nonumber \\
		&\leq P\left( \left| \frac1T\sigma^\trans AQ_-B \sigma - \frac1T\sigma^\trans A\EE[Q_-]B \sigma \right| > \frac{t}2 \right) \nonumber \\
		&+ P\left( \left| \frac1T\sigma^\trans A\EE[Q_-]B \sigma - \frac1T\tr \Phi A\EE[Q_-]B \right| > \frac{t}2 \right) \\
		&\leq C' e^{-c'n \min (t^2,t)}
	\end{align*}
	which concludes the proof.
\end{proof}

As a further corollary of Lemma~\ref{lem:Lip_Sigma}, we have the following concentration result on the training mean-square error of the neural network under study.
\begin{corollary}[Concentration of the mean-square error]
	\label{cor:Etrain}
	Under Assumptions~\ref{ass:W}--\ref{ass:growth},
	\begin{align*}
		P\left( \left| \frac1T\tr Y^\trans YQ^2 - \frac1T \tr Y^\trans Y\EE\left[ Q^2 \right]\right| > t \right) &\leq Ce^{-cnt^2} 
	\end{align*}
	for some $C,c>0$ independent of the other parameters.
\end{corollary}
\begin{proof}
	We apply Lemma~\ref{lem:Lip_Sigma} to the mapping $f:R\mapsto \frac1T\tr Y^\trans Y(R^\trans R+\gamma I_T)^{-2}$. Denoting $Q=(R^\trans R+\gamma I_T)^{-1}$ and $Q^H=( (R+H)^\trans(R+H)+\gamma I_T)^{-1}$, remark indeed that
	\begin{align*}
		&\left| f(R+H) - f(R) \right| \nonumber \\
		&= \left| \frac1T\tr Y^\trans Y ( (Q^H)^2-Q^2) \right| \\
		&\leq \left| \frac1T\tr Y^\trans Y (Q^H-Q)Q^H \right| + \left| \frac1T\tr Y^\trans Y Q (Q^H-Q) \right| \\
		&= \left| \frac1T\tr Y^\trans Y Q^H ( (R+H)^\trans (R+H) - R^\trans R  ) QQ^H \right| \nonumber \\
		&+ \left| \frac1T\tr Y^\trans Y Q Q^H ( (R+H)^\trans (R+H) - R^\trans R) Q \right| \\
		&\leq \left| \frac1T\tr Y^\trans Y Q^H (R+H)^\trans H QQ^H \right| + \left| \frac1T\tr Y^\trans Y Q^H H^\trans R QQ^H \right| \nonumber \\
		&+ \left| \frac1T\tr Y^\trans Y Q Q^H (R+H)^\trans R Q \right| + \left| \frac1T\tr Y^\trans Y Q Q^H H^\trans R Q \right|.
	\end{align*}
	As $\|Q^H(R+H)^\trans\|=\sqrt{\|Q^H(R+H)^\trans(R+H)Q^H\|}$ and $\|RQ\|=\sqrt{\|QR^\trans RQ\|}$ are bounded and $\frac1T\tr Y^\trans Y$ is also bounded by Assumption~\ref{ass:growth}, this implies
	\begin{align*}
		\left| f(R+H) - f(R) \right| &\leq C \|H\| \leq C\|H\|_F
	\end{align*}
	for some $C>0$. The function $f$ is thus Lipschitz with parameter independent of $n$, which allows us to conclude using Lemma~\ref{lem:Lip_Sigma}.
\end{proof}

\bigskip

The aforementioned concentration results are the building blocks of the proofs of Theorem~\ref{th:EQ}--\ref{th:Etrain} which, \textcolor{black}{under all Assumptions~\ref{ass:W}--\ref{ass:growth}}, are established using standard random matrix approaches.

\subsection{Asymptotic Equivalents}

\subsubsection{First Equivalent for $\EE[Q]$}

This section is dedicated to a first characterization of $\EE[Q]$, in the \textcolor{black}{``simultaneously large''} $n,p,T$ regime. This preliminary step is classical in studying resolvents in random matrix theory as the direct comparison of $\EE[Q]$ to $\bar Q$ with the implicit $\delta$ may be cumbersome. To this end, let us thus define the intermediary deterministic matrix
\begin{align*}
	\tilde Q &= \left( \frac{n}T \frac{\Phi}{1+\alpha} + \gamma I_T \right)^{-1}
\end{align*}
with $\alpha \equiv \frac1T\tr \Phi \EE[Q_-]$, where we recall that $Q_-$ is a random matrix distributed as, say, $(\frac1T\Sigma^\trans \Sigma - \frac1T\sigma_1\sigma_1^\trans + \gamma I_T)^{-1}$. 

First note that, since $\frac1T\tr \Phi=\EE[\frac1T \|\sigma\|^2]$ and, from \eqref{eq:sigma_bound} \textcolor{black}{and Assumption~\ref{ass:growth}}, $P(\frac1T \|\sigma\|^2>t)\leq Ce^{-cnt^2}$ for all large $t$, we find that $\frac1T\tr \Phi=\int_0^\infty t^2P(\frac1T \|\sigma\|^2>t)dt\leq C'$ for some constant $C'$. Thus, $\alpha\leq \|\EE[Q_-]\|\frac1T\tr \Phi\leq \frac{C'}\gamma$ is uniformly bounded.

\bigskip

We will show here that $\|\EE[Q]-\tilde{Q}\|\to 0$ as $n\to\infty$ in the regime of Assumption~\ref{ass:growth}. \textcolor{black}{As the proof steps are somewhat classical, we defer to the appendix some classical intermediary lemmas (Lemmas~\ref{lem:resolvent_identity}--\ref{lem:norm_control}).} Using the resolvent identity, Lemma~\ref{lem:resolvent_identity}, we start by writing
\begin{align*}
	\EE[Q] - \tilde{Q} &= \EE\left[ Q \left( \frac{n}T \frac{\Phi}{1+\alpha} - \frac1T\Sigma^\trans \Sigma \right) \right] \tilde{Q} \\
	&= \EE[Q] \frac{n}T\frac{\Phi}{1+\alpha}\tilde{Q} - \EE\left[ Q \frac1T\Sigma^\trans \Sigma \right]\tilde{Q} \\
	&= \EE[Q] \frac{n}T\frac{\Phi}{1+\alpha}\tilde{Q} - \frac1T \sum_{i=1}^n \EE\left[ Q \sigma_i \sigma_i^\trans \right]\tilde{Q}
\end{align*}
which, from Lemma~\ref{lem:QQ-}, gives, for $Q_{-i}=(\frac1T\Sigma^\trans\Sigma-\frac1T\sigma_i\sigma_i^\trans+\gamma I_T)^{-1}$,
\begin{align*}
	\EE[Q] - \tilde{Q} &= \EE[Q] \frac{n}T\frac{\Phi}{1+\alpha}\tilde Q - \frac1T \sum_{i=1}^n \EE\left[ Q_{-i} \frac{\sigma_i \sigma_i^\trans}{1+\frac1T\sigma_i^\trans Q_{-i}\sigma_i} \right]\tilde{Q} \\
	&= \EE[Q] \frac{n}T\frac{\Phi}{1+\alpha}\tilde Q - \frac1{1+\alpha} \frac1T \sum_{i=1}^n \EE\left[ Q_{-i} \sigma_i \sigma_i^\trans \right]\tilde{Q} \nonumber \\
	&+ \frac1T \sum_{i=1}^n \EE\left[ \frac{Q_{-i} \sigma_i \sigma_i^\trans \left(\frac1T\sigma_i^\trans Q_{-i}\sigma_i-\alpha\right) }{(1+\alpha)(1+\frac1T\sigma_i^\trans Q_{-i}\sigma_i)} \right]\tilde{Q}.
\end{align*}
Note now, from the independence of $Q_{-i}$ and $\sigma_i\sigma_i^\trans$, that the second right-hand side expectation is simply $\EE[Q_{-i}]\Phi$. Also, exploiting Lemma~\ref{lem:QQ-} in reverse on the rightmost term, this gives
\begin{align}
	\label{eq:first_bound_on_QtQ}
	\EE[Q] - \tilde{Q} &= \frac1T\sum_{i=1}^n\frac{\EE[Q-Q_{-i}]\Phi}{1+\alpha}\tilde{Q} \nonumber \\
	&+ \frac1{1+\alpha} \frac1T \sum_{i=1}^n \EE\left[Q\sigma_i \sigma_i^\trans\tilde{Q} \left(\frac1T\sigma_i^\trans Q_{-i}\sigma_i-\alpha\right) \right].
\end{align}
It is convenient at this point to note that, since $\EE[Q]-\tilde{Q}$ is symmetric, we may write
\begin{align}
	\label{eq:bound_on_QtQ}
	\EE[Q] - \tilde{Q} &= \frac12\frac1{1+\alpha}\left( \frac1T\sum_{i=1}^n\left( \EE[Q-Q_{-i}]\Phi \tilde{Q} + \tilde{Q}\Phi \EE[Q-Q_{-i}] \right) \nonumber\right. \\
	&\left. + \frac1T \sum_{i=1}^n \EE\left[\left(Q\sigma_i \sigma_i^\trans\tilde{Q}+\tilde{Q}\sigma_i\sigma_i^\trans Q\right) \left(\frac1T\sigma_i^\trans Q_{-i}\sigma_i-\alpha\right) \right] \right).
\end{align}

We study the two right-hand side terms of \eqref{eq:bound_on_QtQ} independently. 

For the first term, since $Q-Q_{-i}=-Q\frac1T\sigma_i\sigma_i^\trans Q_{-i}$,
\begin{align*}
	\frac1T\sum_{i=1}^n\frac{\EE[Q-Q_{-i}]\Phi}{1+\alpha}\tilde{Q} &= \frac1{1+\alpha} \frac1T \EE\left[ Q \frac1T \sum_{i=1}^n\sigma_i\sigma_i^\trans Q_{-i} \right]\Phi \tilde{Q} \\
	&= \frac1{1+\alpha} \frac1T \EE\left[ Q \frac1T \sum_{i=1}^n\sigma_i\sigma_i^\trans Q \left(1+\frac1T\sigma_i^\trans Q_{-i}\sigma_i\right) \right]\Phi \tilde{Q}
\end{align*}
where we used again Lemma~\ref{lem:QQ-} in reverse. Denoting $D=\diag( \{ 1+\frac1T\sigma_i^\trans Q_{-i}\sigma_i \}_{i=1}^n )$, this can be compactly written
\begin{align*}
	\frac1T\sum_{i=1}^n\frac{\EE[Q-Q_{-i}]\Phi}{1+\alpha}\tilde{Q} &= \frac1{1+\alpha} \frac1T \EE\left[ Q \frac1T\Sigma^\trans D\Sigma Q\right]\Phi\tilde{Q}.
\end{align*}
Note at this point that, from Lemma~\ref{lem:norm_control}, $\|\Phi\tilde{Q}\|\leq (1+\alpha)\frac{T}n$ and 
\begin{align*}
	\left\| Q\frac1{\sqrt{T}}\Sigma^\trans \right\| &= \sqrt{\left\| Q\frac1T\Sigma^\trans\Sigma Q \right\|} \leq \gamma^{-\frac12}.
\end{align*}
Besides, by Lemma~\ref{lem:concentration_sQs} and the union bound,
\begin{align*}
	P \left( \max_{1\leq i\leq n} D_{ii} > 1+\alpha + t \right) &\leq  Cn \textcolor{black}{e^{-cn\min(t^2,t)}}
\end{align*}
for some $C,c>0$, so in particular, recalling that $\alpha\leq C'$ for some constant $C'>0$,
\textcolor{black}{
\begin{align*}
	\EE\left[\max_{1\leq i\leq n} D_{ii}\right] &= \int_0^{2(1+C')}  P\left( \max_{1\leq i\leq n} D_{ii} > t \right) dt + \int_{2(1+C')}^\infty P\left( \max_{1\leq i\leq n} D_{ii} > t \right) dt \\
	&\leq 2(1+C') + \int_{2(1+C')}^\infty Cn e^{-cn\min( (t-(1+C'))^2,t-(1+C'))}dt \\
	&=2(1+C') + \int_{1+C'}^\infty Cn e^{-cnt}dt \\
	&=2(1+C') + e^{-Cn(1+C')} = O(1).
\end{align*}} 
As a consequence of all the above (and of the boundedness of $\alpha$), we have that, for some $c>0$,
\begin{align}
	\label{eq:bound1}
	\frac1T \left\| \EE\left[ Q \frac1T\Sigma^\trans D\Sigma Q \right] \Phi\tilde{Q} \right\| &\leq \textcolor{black}{\frac{c}{n}}.
\end{align} 

Let us now consider the second right-hand side term of \eqref{eq:bound_on_QtQ}. Using the relation $ab^\trans + ba^\trans\preceq aa^\trans + bb^\trans$ in the order of Hermitian matrices (which unfolds from $(a-b)(a-b)^\trans\succeq 0$), we have, with $a=T^{\frac14}Q\sigma_i ( \frac1T\sigma_i^\trans Q_{-i}\sigma_i - \alpha)$ and $b=T^{-\frac14}\tilde{Q}\sigma_i$,
\begin{align*}
	&\frac1T\sum_{i=1}^n \EE\left[ \left( Q\sigma_i \sigma_i^\trans \tilde{Q} + \tilde{Q}\sigma_i\sigma^\trans Q \right) \left( \frac1T\sigma_i^\trans Q_{-i}\sigma_i - \alpha\right) \right] \nonumber \\
	&\preceq \frac1{\sqrt{T}}\sum_{i=1}^n \EE\left[ Q\sigma_i \sigma_i^\trans Q \left( \frac1T\sigma_i^\trans Q_{-i}\sigma_i - \alpha\right)^2 \right] + \frac1{T\sqrt{T}}\sum_{i=1}^n \EE\left[ \tilde{Q}\sigma_i\sigma_i^\trans \tilde{Q} \right] \\
	&= \sqrt{T} \EE\left[ Q \frac1T\Sigma^\trans D_2^2 \Sigma Q \right] + \frac{n}{T\sqrt{T}} \tilde{Q}\Phi\tilde{Q}
\end{align*}
where $D_2={\rm diag}( \{\frac1T\sigma_i^\trans Q_{-i}\sigma_i - \alpha\}_{i=1}^n )$. Of course, since we also have $-aa^\trans-bb^\trans\preceq ab^\trans + ba^\trans$ (from $(a+b)(a+b)^\trans \succeq 0$), we have symmetrically
\begin{align*}
	&\frac1T\sum_{i=1}^n \EE\left[ \left( Q\sigma_i \sigma_i^\trans \tilde{Q} + \tilde{Q}\sigma_i\sigma^\trans Q \right) \left( \frac1T\sigma_i^\trans Q_{-i}\sigma_i - \alpha\right) \right] \nonumber \\
	&\succeq -\sqrt{T} \EE\left[ Q \frac1T\Sigma^\trans D_2^2 \Sigma Q \right] - \frac{n}{T\sqrt{T}} \tilde{Q}\Phi\tilde{Q}.
\end{align*}

But from Lemma~\ref{lem:concentration_sQs}, 
\begin{align*}
	P\left( \left\| D_2\right\| > t n^{\varepsilon-\frac12} \right) &= P\left( \max_{1\leq i\leq n} \left| \frac1T\sigma_i^\trans Q_{-i}\sigma_i - \alpha \right| > tn^{\varepsilon-\frac12} \right) \\
	&\leq C n e^{-c \textcolor{black}{\min(n^{2\varepsilon}t^2,n^{\frac12 + \varepsilon}t)}}
\end{align*}
so that, with a similar reasoning as in the proof of Corollary~\ref{cor:moments_quadform},
\begin{align*}
	\left\|  \sqrt{T} \EE\left[ Q \frac1T\Sigma^\trans D_2^2 \Sigma Q \right] \right\| &\leq \sqrt{T} \EE \left[ \|D_2\|^2 \right] \leq C n^{\varepsilon'-\frac12}
\end{align*}
where we additionally used $\|Q\Sigma\|\leq \sqrt{T}$ in the first inequality.

Since in addition $\left\| \frac{n}{T\sqrt{T}} \tilde{Q}\Phi\tilde{Q} \right\|\leq Cn^{-\frac12}$, this gives
\begin{align*}
	\left\|\frac1T\sum_{i=1}^n \EE\left[ \left( Q\sigma_i \sigma_i^\trans \tilde{Q} + \tilde{Q}\sigma_i\sigma_i^\trans Q \right) \left( \frac1T\sigma_i^\trans Q_{-i}\sigma_i - \alpha\right) \right] \right\| &\leq Cn^{\varepsilon-\frac12}.
\end{align*}

Together with \eqref{eq:bound_on_QtQ}, we thus conclude that
\begin{align*}
	 \left\| \EE[Q] - \tilde{Q} \right\| &\leq Cn^{\varepsilon-\frac12}.
\end{align*}

\bigskip

Note in passing that we proved that
\begin{align*}
	\left\| \EE\left[ Q - Q_- \right] \right\| &= \frac{T}n \left\|\frac1T\sum_{i=1}^n \EE\left[Q-Q_{-i}\right] \right\| = \left\| \frac1n \EE\left[ Q \frac1T\Sigma^\trans D\Sigma Q \right] \right\| \leq \textcolor{black}{\frac{c}{n}}
\end{align*}
where the first equality holds by exchangeability arguments.

In particular,
\begin{align*}
	\alpha &= \frac1T\tr \Phi \EE[Q_-] =  \frac1T\tr \Phi \EE[Q] + \frac1T\tr \Phi (\EE[Q_-]-\EE[Q])
\end{align*}
where $|\frac1T\tr \Phi (\EE[Q_-]-\EE[Q])|\leq \textcolor{black}{\frac{c}{n}}$. And thus, by the previous result,
\begin{align*}
	\left| \alpha - \frac1T\tr \Phi \tilde{Q} \right| &\leq C n^{-\frac12+\varepsilon}\frac1T\tr \Phi.
\end{align*}
We have proved in the beginning of the section that $\frac1T\tr \Phi$ is bounded and thus we finally conclude that
\begin{align*}
	 \left\| \alpha - \frac1T\tr \Phi \tilde{Q} \right\| &\leq Cn^{\varepsilon-\frac12}.
\end{align*}

\subsubsection{Second Equivalent for $\EE[Q]$}

In this section, we show that $\EE[Q]$ can be approximated by the matrix $\bar{Q}$, which we recall is defined as
\begin{align*}
	\bar{Q} &= \left( \frac{n}T \frac{\Phi}{1+\delta} + \gamma I_T \right)^{-1}
\end{align*}
where $\delta>0$ is the unique positive solution to $\delta=\frac1T\tr \Phi\bar{Q}$. The fact that $\delta>0$ is well defined is quite standard and has already been proved several times for more elaborate models. Following the ideas of \citep{COU11d}, we may for instance use the framework of so-called standard interference functions \citep{YAT95} which claims that, if a map $f:[0,\infty)\to (0,\infty)$, $x\mapsto f(x)$, satisfies $x\geq x'\Rightarrow f(x)\geq f(x')$, $\forall a>1, af(x)>f(ax)$ and there exists $x_0$ such that $x_0\geq f(x_0)$, then $f$ has a unique fixed point \cite[Th~2]{YAT95}. It is easily shown that $\delta\mapsto \frac1T\tr \Phi\bar{Q}$ is such a map, so that $\delta$ exists and is unique.

\bigskip

To compare $\tilde{Q}$ and $\bar{Q}$, using the resolvent identity, Lemma~\ref{lem:resolvent_identity}, we start by writing
\begin{align*}
	\tilde{Q} - \bar{Q} &= (\alpha-\delta)\tilde{Q}\frac{n}T \frac{\Phi}{(1+\alpha)(1+\delta)}\bar{Q}
\end{align*}
from which
\begin{align*}
	\left| \alpha - \delta \right| &= \left|\frac1T\tr \Phi \left(\EE[Q_-] - \bar{Q}\right)\right| \\
	&\leq \left|\frac1T\tr \Phi \left(\tilde{Q} - \bar{Q}\right) \right| + cn^{-\frac12+\varepsilon} \\
	&= \left| \alpha-\delta \right| \frac1T\tr \frac{\Phi \tilde{Q}\frac{n}T \Phi \bar{Q}}{(1+\alpha)(1+\delta)} + cn^{-\frac12+\varepsilon}
\end{align*}
which implies that
\begin{align*}
	\left| \alpha - \delta \right|\left(1 - \frac1T\tr \frac{\Phi \tilde{Q}\frac{n}T \Phi \bar{Q}}{(1+\alpha)(1+\delta)} \right) &\leq cn^{-\frac12+\varepsilon}.
\end{align*}
It thus remains to show that 
\begin{align*}
	\limsup_n  \frac1T\tr \frac{\Phi \tilde{Q}\frac{n}T \Phi \bar{Q}}{(1+\alpha)(1+\delta)} < 1
\end{align*}
to prove that $|\alpha-\delta|\leq cn^{\varepsilon-\frac12}$. To this end, note that, by Cauchy--Schwarz's inequality,
\begin{align*}
	\frac1T\tr \frac{\Phi \tilde{Q}\frac{n}T \Phi \bar{Q}}{(1+\alpha)(1+\delta)} \leq \sqrt{\frac{n}{T(1+\delta)^2}\frac1T\tr \Phi^2 \bar{Q}^2 \cdot \frac{n}{T(1+\alpha)^2}\frac1T\tr \Phi^2 \tilde{Q}^2}
\end{align*}
so that it is sufficient to bound the limsup of both terms under the square root strictly by one. Next, remark that
\begin{align*}
	\delta &= \frac1T\tr \Phi\bar{Q} = \frac1T\tr \Phi\bar{Q}^2\bar{Q}^{-1} = \frac{n(1+\delta)}{T(1+\delta)^2} \frac1T\tr \Phi^2\bar{Q}^2 + \gamma \frac1T\tr \Phi \bar{Q}^2.
\end{align*}
In particular,
\begin{align*}
	\frac{n}{T(1+\delta)^2}\frac1T\tr \Phi^2 \bar{Q}^2 &= \frac{\delta \frac{n}{T(1+\delta)^2}\frac1T\tr \Phi^2 \bar{Q}^2}{(1+\delta)\frac{n}{T(1+\delta)^2}\frac1T\tr \Phi^2 \bar{Q}^2 + \gamma \frac1T\tr \Phi \bar{Q}^2} \leq \frac{\delta}{1+\delta}.
\end{align*}
But at the same time, since $\|(\frac{n}T \Phi + \gamma I_T)^{-1}\|\leq \gamma^{-1}$,
\begin{align*}
	\delta &\leq \frac1{\gamma T}\tr \Phi
\end{align*}
the limsup of which is bounded. We thus conclude that
\begin{align}
	\label{eq:bound_2ndorder_delta}
	\limsup_n \frac{n}{T(1+\delta)^2}\frac1T\tr \Phi^2 \bar{Q}^2 &<1.
\end{align}
Similarly, $\alpha$, which is known to be bounded, satisfies
\begin{align*}
	\alpha &= (1+\alpha) \frac{n}{T(1+\alpha)^2}\frac1T\tr \Phi^2 \tilde{Q}^2 + \gamma \frac1T\tr \Phi \tilde{Q}^2 + O(n^{\varepsilon-\frac12})
\end{align*}
and we thus have also
\begin{align*}
	\limsup_n \frac{n}{T(1+\alpha)^2}\frac1T\tr \Phi^2 \tilde{Q}^2 &<1
\end{align*}
which completes to prove that $|\alpha-\delta|\leq cn^{\varepsilon-\frac12}$.

As a consequence of all this,
\begin{align*}
	\| \tilde{Q} - \bar{Q} \| &= |\alpha-\delta| \cdot \left\| \frac{\tilde{Q}\frac{n}T \Phi \bar{Q}}{(1+\alpha)(1+\delta)} \right\| \leq c n^{-\frac12+\varepsilon}
\end{align*}
and we have thus proved that $\|\EE[Q]-\bar{Q}\|\leq c n^{-\frac12+\varepsilon}$ for some $c>0$.

\bigskip

From this result, along with Corollary~\ref{cor:ST}, we now have that
\begin{align*}
	&P\left( \left| \frac1T\tr Q - \frac1T\tr \bar Q \right| > t \right) \nonumber \\
	&\leq P\left( \left| \frac1T\tr Q - \frac1T\tr \EE[Q] \right| > t - \left|\frac1T\tr \EE[Q] - \frac1T\tr \bar Q \right| \right) \\
	&\leq C'e^{-c'n(t-cn^{-\frac12+\varepsilon})} \leq C'e^{-\frac12c'nt}
\end{align*}
for all large $n$. As a consequence, for all $\gamma>0$, $\frac1T\tr Q-\frac1T\tr \bar Q\to 0$ almost surely. As such, the difference $m_{\mu_n}-m_{\bar\mu_n}$ of Stieltjes transforms $m_{\mu_n}:\CC\setminus\RR^+\to \CC$, $z\mapsto \frac1T\tr (\frac1T\Sigma^\trans \Sigma-zI_T)^{-1}$ and $m_{\bar\mu_n}:\CC\setminus\RR^+\to \CC$, $z\mapsto \frac1T\tr (\frac{n}T\frac{\Phi}{1+\delta_z}-zI_T)^{-1}$ (with $\delta_z$ the unique Stieltjes transform solution to $\delta_z=\frac1T\tr \Phi (\frac{n}T\frac{\Phi}{1+\delta_z}-zI_T)^{-1}$) converges to zero for each $z$ in a subset of $\CC\setminus\RR^+$ having at least one accumulation point \textcolor{black}{(namely $\RR^-$)}, almost surely so (that is, on a probability set $\mathcal A_z$ with $P(\mathcal A_z)=1$). Thus, letting $\{z_k\}_{k=1}^\infty$ be a converging sequence strictly included in $\RR^-$, on the probability one space $\mathcal A=\cap_{k=1}^\infty\mathcal A_k$, $m_{\mu_n}(z_k)-m_{\bar\mu_n}(z_k)\to 0$ for all $k$. Now, $m_{\mu_n}$ is complex analytic on $\CC\setminus\RR^+$ and bounded on all compact subsets of $\CC\setminus\RR^+$. Besides, it was shown in \citep{SIL95,CHO95} that the function $m_{\bar\mu_n}$ is well-defined, complex analytic and bounded on all compact subsets of $\CC\setminus\RR^+$. As a result, on $\mathcal A$, $m_{\mu_n}-m_{\bar \mu_n}$ is complex analytic, bounded on all compact subsets of $\CC\setminus\RR^+$ and converges to zero on a subset admitting at least one accumulation point. Thus, by Vitali's convergence theorem \citep{TIT39}, with probability one, $m_{\mu_n}-m_{\bar \mu_n}$ converges to zero everywhere on $\CC\setminus\RR^+$. This implies, by \cite[Theorem~B.9]{SIL06}, that $\mu_n-\bar\mu_n\to 0$, \textcolor{black}{vaguely as a signed finite measure}, with probability one, and, \textcolor{black}{since $\bar\mu_n$ is a probability measure (again from the results of \citep{SIL95,CHO95})}, we have thus proved Theorem~\ref{th:lsd}.

\subsubsection{Asymptotic Equivalent for $\EE[QAQ]$, where $A$ is either $\Phi$ or symmetric of bounded norm}
\label{subsubsec:QAQ}
The evaluation of the second order statistics of the neural network under study requires, beside $\EE[Q]$, to evaluate the more involved form $\EE[QAQ]$, where $A$ is a symmetric matrix either equal to $\Phi$ or of bounded norm (so in particular $\|\bar QA\|$ is bounded). To evaluate this quantity, first write
\begin{align*}
	\EE[QAQ] &= \EE\left[\bar{Q}AQ\right] + \EE\left[ (Q-\bar{Q})AQ \right] \\
	&= \EE\left[\bar{Q}AQ\right] + \EE\left[ Q\left(\frac{n}T\frac{\Phi}{1+\delta} - \frac1T\Sigma^\trans \Sigma \right)\bar{Q}AQ \right] \\
	&= \EE\left[\bar{Q}AQ\right] +\frac{n}T\frac1{1+\delta} \EE\left[ Q\Phi \bar{Q}AQ \right] - \frac1T \sum_{i=1}^n \EE\left[ Q\sigma_i\sigma_i^\trans \bar{Q}AQ \right].
\end{align*}
Of course, since $QAQ$ is symmetric, we may write
\begin{align*}
	\EE[QAQ] &= \frac12 \left( \EE\left[\bar{Q}AQ + QA\bar{Q}\right] + \frac{n}T\frac1{1+\delta} \EE\left[ Q\Phi \bar{Q}AQ + QA\bar{Q}\Phi Q\right] \right. \nonumber \\ 
	&\left. - \frac1T \sum_{i=1}^n \EE\left[ Q\sigma_i\sigma_i^\trans \bar{Q}AQ + QA\bar{Q}\sigma_i\sigma_i^\trans Q \right] \right)
\end{align*}
which will reveal more practical to handle.

First note that, since $\left\|\EE[Q]-\bar{Q}\right\|\leq Cn^{\varepsilon-\frac12}$ and $A$ is such that $\|\bar{Q}A\|$ is bounded, $\|\EE[\bar{Q}AQ] -\bar{Q}A\bar{Q}]\|\leq \|\bar{Q}A\| \|\EE[Q]-\bar{Q}\|\leq C'n^{\varepsilon-\frac12}$, which provides an estimate for the first expectation. We next evaluate the last right-hand side expectation above. With the same notations as previously, from exchangeability arguments and using $Q=Q_--Q\frac1T\sigma\sigma^\trans Q_-$, observe that
\begin{align*}
	\frac1T\sum_{i=1}^n \EE\left[ Q\sigma_i\sigma_i^\trans \bar{Q}AQ \right] 
	&= \frac{n}T \EE\left[ Q\sigma\sigma^\trans \bar{Q}AQ \right] \\
	&= \frac{n}T \EE\left[ \frac{Q_-\sigma\sigma^\trans\bar{Q}AQ }{1 + \frac1T\sigma^\trans Q_-\sigma} \right] \\
	&= \frac{n}T \frac1{1+\delta} \EE\left[ Q_-\sigma\sigma^\trans \bar{Q}AQ\right] \nonumber \\ &+ \frac{n}T \frac1{1+\delta} \EE\left[ Q_-\sigma\sigma^\trans \bar{Q}AQ \frac{\delta - \frac1T\sigma^\trans Q_-\sigma}{1 + \frac1T\sigma^\trans Q_-\sigma} \right]
\end{align*}
which, reusing $Q=Q_--Q\frac1T\sigma\sigma^\trans Q_-$, is further decomposed as
\begin{align*}
	&\frac1T\sum_{i=1}^n \EE\left[ Q\sigma_i\sigma_i^\trans \bar{Q}AQ \right] \nonumber \\
	&= \frac{n}T \frac1{1+\delta} \EE\left[ Q_-\sigma\sigma^\trans \bar{Q}AQ_-\right] - \frac{n}{T^2}\frac1{1+\delta} \EE\left[ \frac{Q_-\sigma\sigma^\trans \bar{Q}AQ_- \sigma\sigma^\trans Q_-}{1+\frac1T\sigma^\trans Q_-\sigma}  \right]  \nonumber \\ 
	&+ \frac{n}T \EE\left[ Q_-\sigma\sigma^\trans \bar{Q}AQ_- \frac{\delta - \frac1T\sigma^\trans Q_-\sigma}{(1+\delta)\left(1 + \frac1T\sigma^\trans Q_-\sigma\right)} \right] \nonumber \\
	&- \frac{n}{T^2} \EE\left[ \frac{Q_-\sigma\sigma^\trans \bar{Q}AQ_- \sigma\sigma^\trans Q_-\left( \delta - \frac1T\sigma^\trans Q_-\sigma \right)}{(1+\delta)\left(1+\frac1T\sigma^\trans Q_-\sigma\right)^2}  \right] \\
	&= \frac{n}T \frac1{1+\delta} \EE\left[ Q_-\Phi\bar{Q}AQ_-\right] - \frac{n}{T}\frac1{1+\delta} \EE\left[ Q_-\sigma \sigma^\trans Q_- \frac{\frac1T  \sigma^\trans \bar{Q}AQ_- \sigma }{1+\frac1T\sigma^\trans Q_-\sigma}  \right]  \nonumber \\
	&+ \frac{n}T \EE\left[ Q_- \frac{\sigma\sigma^\trans\left( \delta - \frac1T\sigma^\trans Q_-\sigma \right)}{(1+\delta)\left(1 + \frac1T\sigma^\trans Q_-\sigma\right)} \bar{Q}AQ_- \right] \nonumber \\
	&- \frac{n}{T}\EE\left[ Q_-\sigma\sigma^\trans Q_- \frac{\frac1T\sigma^\trans \bar{Q}AQ_- \sigma \left( \delta - \frac1T\sigma^\trans Q_-\sigma \right)}{(1+\delta)\left(1+\frac1T\sigma^\trans Q_-\sigma\right)^2}  \right] \\
	&\equiv Z_1+Z_2+Z_3+Z_4
\end{align*}
(where in the previous to last line, we have merely reorganized the terms conveniently) and our interest is in handling $Z_1+Z_1^\trans+Z_2+Z_2^\trans+Z_3+Z_3^\trans+Z_4+Z_4^\trans$. Let us first treat term $Z_2$. Since $\bar{Q}AQ_-$ is bounded, by Lemma~\ref{lem:concentration_sQs}, $\frac1T\sigma^\trans \bar QAQ_-\sigma$ concentrates around $\frac1T\tr \Phi \bar QAE[Q_-]$; but, as $\|\Phi\bar Q\|$ is bounded, we also have $|\frac1T\tr \Phi \bar QAE[Q_-]-\frac1T\tr \Phi\bar QA\bar Q|\leq cn^{\varepsilon-\frac12}$. 
We thus deduce, with similar arguments as previously, that
\begin{align*}
	-Q_-\sigma\sigma^\trans Q_- Cn^{\varepsilon-\frac12} 
	&\preceq Q_-\sigma\sigma^\trans Q_- \left[\frac{\frac1T\sigma^\trans \bar{Q}AQ_- \sigma}{1+\frac1T\sigma^\trans Q_-\sigma} - \frac{\frac1T\tr \Phi \bar{Q}A\bar{Q}}{1+\delta} \right] \\
	&\preceq Q_-\sigma\sigma^\trans Q_- Cn^{\varepsilon-\frac12}
\end{align*}
with probability exponentially close to one, in the order of symmetric matrices. Taking expectation and norms on both sides, and conditioning on the aforementioned event and its complementary, we thus have that 
\begin{align*}
	&\left\| \EE\left[ Q_-\sigma\sigma^\trans Q_- \frac{\frac1T\sigma^\trans \bar{Q}AQ_- \sigma}{1+\frac1T\sigma^\trans Q_-\sigma} \right] - \EE\left[ Q_- \Phi Q_-\right] \frac{\frac1T\tr \Phi \bar{Q}A\bar{Q}}{1+\delta} \right\| \nonumber \\
	&\leq \left\| \EE\left[ Q_- \Phi Q_-\right] \right\| Cn^{\varepsilon-\frac12} + C'ne^{-c \textcolor{black}{n^{\varepsilon'}}} \\
	&\leq \left\| \EE\left[ Q_- \Phi Q_-\right] \right\| C''n^{\varepsilon-\frac12}
\end{align*}
But, again by exchangeability arguments,
\begin{align*}
	\EE[Q_-\Phi Q_-] &= \EE[Q_-\sigma\sigma^\trans Q_-] = \EE\left[Q\sigma\sigma^\trans Q \left(1+\frac1T\sigma^\trans Q_-\sigma\right)^2\right] \\
	&= \frac{T}n \EE\left[ Q\frac1T\Sigma^\trans D^2 \Sigma Q \right]
\end{align*}
with $D=\diag( \{1+\frac1T\sigma_i^\trans Q_-\sigma_i\} )$, the operator norm of which is bounded \textcolor{black}{as O(1)}. So finally,
\begin{align*}
	\left\| \EE\left[ Q_-\sigma\sigma^\trans Q_- \frac{\frac1T\sigma^\trans \bar{Q}AQ_- \sigma}{1+\frac1T\sigma^\trans Q_-\sigma} \right] - \EE\left[ Q_- \Phi Q_-\right] \frac{\frac1T\tr \Phi \bar{Q}A\bar{Q}}{1+\delta} \right\| &\leq Cn^{\varepsilon-\frac12}.
\end{align*}

\medskip

We now move to term $Z_3+Z_3^\trans$. Using the relation $ab^\trans+ba^\trans\preceq aa^\trans+bb^\trans$,
\begin{align*}
	&\EE\left[(\delta-\frac1T\sigma^\trans Q_-\sigma) \frac{Q_-\sigma\sigma^\trans \bar{Q}AQ_- + Q_-A\bar{Q}\sigma\sigma^\trans Q_-}{(1+\frac1T\sigma^\trans Q_-\sigma)^2} \right] \nonumber \\
	&\preceq \sqrt{n} \EE\left[ \frac{(\delta-\frac1T\sigma^\trans Q_-\sigma)^2}{(1+\frac1T\sigma^\trans Q_-\sigma)^4} Q_-\sigma\sigma^\trans Q_- \right] + \frac1{\sqrt{n}}\EE\left[ Q_-A\bar{Q}\sigma\sigma^\trans \bar{Q}AQ_- \right] \\
	&=\sqrt{n} \frac{T}n \EE\left[ Q \frac1T\Sigma^\trans D_3^2 \Sigma Q \right] + \frac1{\sqrt{n}} \EE\left[ Q_-A\bar{Q}\Phi\bar{Q}AQ_-\right]
\end{align*}
and the symmetrical lower bound (equal to the opposite of the upper bound), where $D_3=\diag( (\delta-\frac1T\sigma_i^\trans Q_{-i}\sigma_i)/(1+\frac1T\sigma_i^\trans Q_{-i}\sigma_i) )$. For the same reasons as above, the first right-hand side term is bounded by $Cn^{\varepsilon-\frac12}$. As for the second term, for $A=I_T$, it is clearly bounded; for $A=\Phi$, using $\frac{n}T\frac{\bar{Q}\Phi}{1+\delta}=I_T-\gamma\bar{Q}$, $\EE[Q_-A\bar{Q}\Phi\bar{Q}AQ_-]$ can be expressed in terms of $\EE[Q_-\Phi Q_-]$ and $\EE[Q_-\bar{Q}^k\Phi Q_-]$ for $k=1,2$, all of which have been shown to be bounded (at most by $Cn^\varepsilon$). We thus conclude that
\begin{align*}
	\left\| \EE\left[\left(\delta-\frac1T\sigma^\trans Q_-\sigma\right) \frac{Q_-\sigma\sigma^\trans \bar{Q}AQ_- + Q_-A\bar{Q}\sigma\sigma^\trans Q_-}{(1+\frac1T\sigma^\trans Q_-\sigma)^2} \right] \right\| &\leq Cn^{\varepsilon-\frac12}.
\end{align*}

Finally, term $Z_4$ can be handled similarly as term $Z_2$ and is shown to be of norm bounded by $Cn^{\varepsilon-\frac12}$.

\bigskip

As a consequence of all the above, we thus find that
\begin{align*}
	\EE\left[ QAQ \right] &= \bar Q A\bar Q + \frac{n}T\frac{\EE\left[ Q\Phi\bar QAQ \right]}{1+\delta} - \frac{n}T\frac{\EE\left[ Q_-\Phi\bar QAQ_- \right]}{1+\delta} \nonumber \\
	&+ \frac{n}T\frac{\frac1T\tr \Phi\bar QA\bar Q}{(1+\delta)^2} \EE[Q_-\Phi Q_-] + O(n^{\varepsilon-\frac12}).
\end{align*}

It is attractive to feel that the sum of the second and third terms above vanishes. This is indeed verified by observing that, for any matrix $B$, 
\begin{align*}
	\EE\left[ Q B Q \right] - \EE\left[ Q_- B Q\right] &= \frac1T \EE\left[ Q\sigma\sigma^\trans Q_- B Q \right] \\
	&= \frac1T \EE\left[ Q \sigma\sigma^\trans Q B Q \left(1+\frac1T\sigma^\trans Q_-\sigma\right)\right] \\
	&= \frac1n \EE\left[ Q \frac1T\Sigma^\trans D\Sigma Q BQ \right]
\end{align*}
and symmetrically
\begin{align*}
	\EE\left[ Q B Q \right] - \EE\left[ Q B Q_-\right] &= \frac1n \EE\left[ Q B Q  \frac1T\Sigma^\trans D\Sigma Q \right]
\end{align*}
with $D=\diag(1+\frac1T\sigma_i^\trans Q_{-i}\sigma_i)$, and a similar reasoning is performed to control $\EE[Q_-BQ]-\EE[Q_-BQ_-]$ and $\EE[QBQ_-]-\EE[Q_-BQ_-]$. For $B$ bounded, $\|\EE[ Q \frac1T\Sigma^\trans D\Sigma Q BQ]\|$ is bounded as \textcolor{black}{$O(1)$}, and thus $\|\EE[QBQ]-\EE[Q_-BQ_-]\|$ is of order \textcolor{black}{$O(n^{-1})$}. So in particular, taking $A$ of bounded norm, we find that 
\begin{align*}
	\EE\left[ Q A Q \right] &= \bar Q A \bar Q + \frac{n}T\frac{\frac1T\tr \Phi\bar QA \bar Q}{(1+\delta)^2} \EE[Q_-\Phi Q_-] + O(n^{\varepsilon-\frac12}).
\end{align*}

Take now $B=\Phi$. Then, from the relation $AB^\trans+BA^\trans\preceq AA^\trans + BB^\trans$ in the order of symmetric matrices,
\begin{align*}
	&\left\|\EE\left[ Q \Phi Q \right] - \frac12 \EE\left[ Q_- \Phi Q + Q\Phi Q_-\right] \right\| \nonumber \\
	&= \frac1{2n} \left\| \EE\left[ Q \frac1T\Sigma^\trans D\Sigma Q \Phi Q + Q\Phi Q \frac1T\Sigma^\trans D\Sigma Q \right] \right\| \\
	&\leq \frac1{2n} \left( \left\| \EE\left[ Q \frac1T\Sigma^\trans D\Sigma Q \frac1T\Sigma^\trans D\Sigma Q \right] \right\| + \left\| \EE\left[  Q \Phi Q \Phi Q \right] \right\| \right).
\end{align*}

The first norm in the parenthesis is bounded by $Cn^{\varepsilon}$ and it thus remains to control the second norm. To this end, similar to the control of $\EE[Q\Phi Q]$, by writing $\EE[Q\Phi Q\Phi Q]=\EE[Q\sigma_1\sigma_1^\trans Q \sigma_2\sigma_2^\trans Q]$ for $\sigma_1,\sigma_2$ independent vectors with the same law as $\sigma$, and exploiting the exchangeability, we obtain after some calculus that $\EE[Q\Phi Q]$ can be expressed as the sum of terms of the form $\EE[Q_{++}\frac1T\Sigma_{++}^\trans D\Sigma_{++} Q_{++}]$ or $\EE[Q_{++}\frac1T\Sigma_{++}^\trans D\Sigma_{++} Q_{++} \frac1T\Sigma_{++}^\trans D_2\Sigma_{++} Q_{++} ]$ for $D,D_2$ diagonal matrices of norm bounded as \textcolor{black}{$O(1)$}, while $\Sigma_{++}$ and $Q_{++}$ are similar as $\Sigma$ and $Q$, only for $n$ replaced by $n+2$. All these terms are bounded as \textcolor{black}{$O(1)$} and we finally obtain that $\EE[Q\Phi Q\Phi Q]$ is bounded and thus
\begin{align*}
	\left\|\EE\left[ Q \Phi Q \right] - \frac12 \EE\left[ Q_- \Phi Q + Q\Phi Q_-\right] \right\| &\leq \textcolor{black}{\frac{C}{n}}.
\end{align*}

\bigskip
	
With the additional control on $Q\Phi Q_--Q_-\Phi Q_-$ and $Q_-\Phi Q-Q_-\Phi Q_-$, together, this implies that $\EE[Q\Phi Q]=\EE[Q_-\Phi Q_-]+O_{\|\cdot\|}(\textcolor{black}{n^{-1}})$. Hence, for $A=\Phi$, exploiting the fact that $\frac{n}T\frac1{1+\delta}\Phi\bar{Q}\Phi=\Phi-\gamma \bar{Q}\Phi$, we have the simplification
\begin{align*}
	\EE\left[ Q\Phi Q \right] &= \bar Q \Phi \bar Q + \frac{n}T\frac{\EE\left[ Q\Phi\bar Q \Phi Q \right]}{1+\delta} - \frac{n}T\frac{\EE\left[ Q_-\Phi\bar Q\Phi Q_- \right]}{1+\delta} \nonumber \\
	&+ \frac{n}T\frac{\frac1T\tr \Phi^2\bar Q^2}{(1+\delta)^2} \EE[Q_-\Phi Q_-] + O_{\|\cdot\|}(n^{\varepsilon-\frac12}) \\
	&= \bar Q \Phi \bar Q + \frac{n}T\frac{\frac1T\tr \Phi^2\bar Q^2}{(1+\delta)^2} \EE[Q\Phi Q] + O_{\|\cdot\|}(n^{\varepsilon-\frac12}).
\end{align*}
or equivalently
\begin{align*}
	\EE\left[ Q\Phi Q \right]\left( 1 - \frac{n}T\frac{\frac1T\tr \Phi^2\bar Q^2}{(1+\delta)^2} \right) &= \bar Q \Phi \bar Q + O_{\|\cdot\|}(n^{\varepsilon-\frac12}).
\end{align*}
We have already shown in \eqref{eq:bound_2ndorder_delta} that $\limsup_n \frac{n}T\frac{\frac1T\tr \Phi^2\bar Q^2}{(1+\delta)^2}<1$ and thus 
\begin{align*}
	\EE\left[ Q\Phi Q \right] &= \frac{\bar Q \Phi \bar Q}{1 - \frac{n}T\frac{\frac1T\tr \Phi^2\bar Q^2}{(1+\delta)^2}} + O_{\|\cdot\|}(n^{\varepsilon-\frac12}).
\end{align*}

So finally, for all $A$ of bounded norm,
\begin{align*}
	\EE\left[ Q A Q \right] &= \bar Q A \bar Q + \frac{n}T\frac{\frac1T\tr \Phi\bar QA \bar Q}{(1+\delta)^2} \frac{\bar Q \Phi \bar Q}{1 - \frac{n}T\frac{\frac1T\tr \Phi^2\bar Q^2}{(1+\delta)^2}} + O(n^{\varepsilon-\frac12})
\end{align*}
which proves immediately Proposition~\ref{prop:QAQ} and Theorem~\ref{th:Etrain}.

\subsection{Derivation of $\Phi_{ab}$}

\subsubsection{Gaussian $w$}

In this section, we evaluate the terms $\Phi_{ab}$ provided in Table~\ref{tab:Phi}. The proof for the term corresponding to $\sigma(t)={\rm erf}(t)$ can be already be found in \cite[Section~3.1]{WIL98} and is not recalled here. For the other functions $\sigma(\cdot)$, we follow a similar approach as in \citep{WIL98}, as detailed next.

\bigskip

The evaluation of $\Phi_{ab}$ for $w\sim \mathcal N(0,I_p)$ requires to estimate
\begin{align*}
	\mathcal I &\equiv (2\pi)^{-\frac{p}2}\int_{\RR^p} \sigma(w^\trans a)\sigma(w^\trans b)e^{-\frac12\|w\|^2}dw.
\end{align*}
Assume that $a$ and $b$ and not linearly dependent. It is convenient to observe that this integral can be reduced to a two-dimensional integration by considering the basis $e_1,\ldots,e_p$ defined (for instance) by
\begin{align*}
	e_1 = \frac{a}{\|a\|},&\quad e_2 = \frac{\frac{b}{\|b\|}-\frac{a^\trans b}{\|a\|\|b\|} \frac{a}{\|a\|}}{\sqrt{1-\frac{(a^\trans b)^2}{\|a\|^2\|b\|^2}}}
\end{align*}
and $e_3,\ldots,e_p$ any completion of the basis. By letting $w=\tilde{w}_1e_1+\ldots+\tilde{w}_pe_p$ and $a=\tilde a_1e_1$ ($\tilde{a}_1=\|a\|$), $b=\tilde b_1e_1+\tilde b_2e_2$ (where $\tilde b_1=\frac{a^\trans b}{\|a\|}$ and $\tilde b_2=\|b\|\sqrt{1-\frac{(a^\trans b)^2}{\|a\|^2\|b\|^2}}$), this reduces $\mathcal I$ to
\begin{align*}
	\mathcal I &= \frac1{2\pi} \int_\RR\int_\RR \sigma(\tilde{w}_1\tilde{a}_1)\sigma(\tilde{w}_1\tilde{b}_1+\tilde{w}_2\tilde{b}_2) e^{-\frac12(\tilde w_1^2+\tilde w_2^2)}d\tilde w_1 d\tilde w_2.
\end{align*}
Letting $\tilde{w}=[\tilde w_1,\tilde w_2]^\trans$, $\tilde a=[\tilde a_1,0]^\trans$ and $\tilde b=[\tilde b_1,\tilde b_2]^\trans$, this is conveniently written as the two-dimensional integral
\begin{align*}
	\mathcal I &= \frac1{2\pi}\int_{\RR^2} \sigma(\tilde w^\trans \tilde a)\sigma(\tilde w^\trans \tilde b)e^{-\frac12\|\tilde w\|^2}d\tilde w.
\end{align*}

The case where $a$ and $b$ would be linearly dependent can then be obtained by continuity arguments.

\paragraph{The function $\sigma(t)=\max(t,0)$}

For this function, we have
\begin{align*}
	\mathcal I &= \frac1{2\pi}\int_{ \min(\tilde{w}^\trans\tilde a,\tilde{w}^\trans\tilde b)\geq 0} \tilde w^\trans \tilde a\cdot \tilde w^\trans \tilde b\cdot e^{-\frac12\|\tilde w\|^2}d\tilde w.
\end{align*}
Since $\tilde{a}=\tilde{a}_1e_1$, a simple geometric representation lets us observe that
\begin{align*}
	\left\{\tilde w~|~\min(\tilde{w}^\trans\tilde a,\tilde{w}^\trans\tilde b)\geq 0\right\} &= \left\{ r\cos(\theta)e_1+r\sin(\theta)e_2~|~r\geq 0,~\theta\in[\theta_0-\frac{\pi}2,\frac{\pi}2]\right\}
\end{align*}
where we defined $\theta_0\equiv \arccos\left( \frac{\tilde{b}_1}{\|\tilde b\|}\right)=-\arcsin\left( \frac{\tilde{b}_1}{\|\tilde b\|}\right)+\frac\pi2$. We may thus operate a polar coordinate change of variable (with inverse Jacobian determinant equal to $r$) to obtain
\begin{align*}
	\mathcal I &= \frac1{2\pi}\int_{\theta_0-\frac{\pi}2}^{\frac{\pi}2} \int_{\RR^+} \left(r\cos(\theta) \tilde{a}_1\right) \left( r\cos(\theta) \tilde{b}_1 + r\sin(\theta) \tilde{b}_2 \right) r e^{-\frac12r^2} d\theta dr \\
	&= \tilde{a}_1 \frac1{2\pi} \int_{\theta_0-\frac{\pi}2}^{\frac{\pi}2} \cos(\theta)\left( \cos(\theta) \tilde{b}_1 + \sin(\theta) \tilde{b}_2 \right) d\theta \int_{\RR^+} r^3e^{-\frac12r^2}dr.
\end{align*}
With two integration by parts, we have that $\int_{\RR^+} r^3e^{-\frac12r^2}dr=2$. Classical trigonometric formulas also provide
\begin{align*}
	\int_{\theta_0-\frac{\pi}2}^{\frac{\pi}2} \cos(\theta)^2 d\theta &= \frac12\left(\pi-\theta_0\right) + \frac12\sin(2\theta_0) \nonumber \\ 
	&= \frac12 \left(\pi-\arccos\left( \frac{\tilde{b}_1}{\|\tilde b\|}\right) + \frac{\tilde b_1}{\|\tilde b\|}\frac{\tilde b_2}{\|\tilde b\|}\right) \\
	\int_{\theta_0-\frac{\pi}2}^{\frac{\pi}2} \cos(\theta)\sin(\theta) d\theta &= \frac12\sin^2(\theta_0) = \frac12 \left(\frac{\tilde b_2}{\|\tilde b\|}\right)^2
\end{align*}
where we used in particular $\sin(2\arccos(x))=2x\sqrt{1-x^2}$. Altogether, this is after simplification and replacement of $\tilde a_1$, $\tilde b_1$ and $\tilde b_2$,
\begin{align*}
	\mathcal I &= \frac1{2\pi} \|a\| \|b\| \left( \sqrt{1-\angle(a,b)^2} + \angle(a,b) \arccos(-\angle(a,b)) \right).
\end{align*}
It is worth noticing that this may be more compactly written as
\begin{align*}
	\mathcal I &= \frac1{2\pi} \|a\| \|b\| \int_{-1}^{\angle(a,b)} \arccos(-x)dx.
\end{align*}
which is minimum for $\angle(a,b)\to -1$ (since $\arccos(-x)\geq 0$ on $[-1,1]$) and takes there the limiting value zero. Hence $\mathcal I>0$ for $a$ and $b$ not linearly dependent.

For $a$ and $b$ linearly dependent, we simply have $\mathcal I=0$ for $\angle(a,b)=-1$ and $\mathcal I=\frac12\|a\| \|b\|$ for $\angle(a,b)=1$.

\paragraph{The function $\sigma(t)=|t|$}

Since $|t|=\max(t,0)+\max(-t,0)$, we have
\begin{align*}
	| w^\trans a | \cdot | w^\trans b| &=\max(w^\trans a,0)\max(w^\trans b,0)+\max(w^\trans (-a),0)\max(w^\trans (-b),0) \nonumber \\
	&+\max(w^\trans (-a),0)\max(w^\trans b,0)+\max(w^\trans a,0)\max(w^\trans (-b),0). 
\end{align*}
Hence, reusing the results above, we have here
\begin{align*}
	\mathcal I &= \frac{\|a\| \|b\|}{2\pi} \left( 4\sqrt{1-\angle(a,b)^2} + 2\angle(a,b)\acos(-\angle(a,b)) - 2 \angle(a,b)\acos(\angle(a,b)) \right).
\end{align*}
Using the identity $\acos(-x)-\acos(x)=2\asin(x)$ provides the expected result.

\paragraph{The function $\sigma(t)=1_{t\geq 0}$}

With the same notations as in the case $\sigma(t)=\max(t,0)$, we have to evaluate
\begin{align*}
	\mathcal I &= \frac1{2\pi} \int_{ \min(\tilde w^\trans \tilde a,\tilde w^\trans \tilde b)\geq 0 } e^{-\frac12\|\tilde w\|^2}d\tilde w. 
\end{align*}
After a polar coordinate change of variable, this is
\begin{align*}
	\mathcal I &= \frac1{2\pi} \int_{\theta_0-\frac{\pi}2}^{\frac{\pi}2} d\theta \int_{\RR^+} r e^{-\frac12r^2} dr = \frac12 - \frac{\theta_0}{2\pi}
\end{align*}
from which the result unfolds.

\paragraph{The function $\sigma(t)={\rm sign}(t)$}

Here it suffices to note that ${\rm sign}(t)=1_{t\geq 0}-1_{-t\geq 0}$ so that
\begin{align*}
	\sigma(w^\trans a)\sigma(w^\trans b) &= 1_{w^\trans a\geq 0}1_{w^\trans b\geq 0} + 1_{w^\trans (-a)\geq 0}1_{w^\trans (-b)\geq 0} \nonumber \\
	&- 1_{w^\trans (-a)\geq 0}1_{w^\trans b\geq 0} - 1_{w^\trans a\geq 0}1_{w^\trans (-b)\geq 0}
\end{align*}
and to apply the result of the previous section, with either $(a,b)$, $(-a,b)$, $(a,-b)$ or $(-a,-b)$. Since $\arccos(-x)=-\arccos(x)+\pi$, we conclude that
\begin{align*}
	\mathcal I &= (2\pi)^{-\frac{p}2} \int_{\RR^p} {\rm sign}(w^\trans a) {\rm sign}(w^\trans b) e^{-\frac12\|w\|^2}dw = 1 - \frac{2\theta_0}{\pi}.
\end{align*}

\paragraph{The functions $\sigma(t)=\cos(t)$ and $\sigma(t)=\sin(t)$.}

Let us first consider $\sigma(t)=\cos(t)$. We have here to evaluate
\begin{align*}
	\mathcal I &= \frac1{2\pi} \int_{\RR^2} \cos\left( \tilde w^\trans \tilde a \right)\cos\left( \tilde w^\trans \tilde b \right) e^{-\frac12\|\tilde w\|^2}d\tilde w \\
	&= \frac1{8\pi} \int_{\RR^2} \left( e^{\imath \tilde w^\trans \tilde a} + e^{-\imath \tilde w^\trans \tilde a} \right)\left( e^{\imath \tilde w^\trans \tilde b} + e^{-\imath \tilde w^\trans \tilde b} \right) e^{-\frac12\|\tilde w\|^2}d\tilde w
\end{align*}
which boils down to evaluating, for $d\in\{\tilde a+\tilde b,\tilde a-\tilde b,-\tilde a+\tilde b,-\tilde a-\tilde b\}$, the integral
\begin{align*}
	e^{-\frac12\|d\|^2} \int_{\RR^2} e^{-\frac12\|\tilde w-\imath d\|^2} d\tilde w &= (2\pi) e^{-\frac12\|d\|^2}.
\end{align*}
Altogether, we find
\begin{align*}
	\mathcal I &= \frac12\left( e^{-\frac12\|a+b\|^2} + e^{-\frac12\|a-b\|^2} \right) = e^{-\frac12(\|a\|+\|b\|^2)}{\rm cosh}(a^\trans b).
\end{align*}

For $\sigma(t)=\sin(t)$, it suffices to appropriately adapt the signs in the expression of $\mathcal I$ (using the relation $\sin(t)=\frac1{2\imath}(e^{t}+e^{-t})$) to obtain in the end
\begin{align*}
	\mathcal I &= \frac12\left( e^{-\frac12\|a+b\|^2} + e^{-\frac12\|a-b\|^2} \right) = e^{-\frac12(\|a\|+\|b\|^2)}{\rm sinh}(a^\trans b)
\end{align*}
as desired.

\subsection{Polynomial $\sigma(\cdot)$ and generic $w$}

In this section, we prove Equation~\ref{eq:Phi_poly2} for $\sigma(t)=\zeta_2t^2+\zeta_1t+\zeta_0$ and $w\in\RR^p$ a random vector with independent and identically distributed entries of zero mean and moment of order $k$ equal to $m_k$. The result is based on standard combinatorics. We are to evaluate
\begin{align*}
	\Phi_{ab} &= \EE\left[ \left(\zeta_2 (w^\trans a)^2+\zeta_1 w^\trans a + \zeta_0 \right)\left( \zeta_2 (w^\trans b)^2+\zeta_1 w^\trans b + \zeta_0 \right) \right].
\end{align*}
After development, it appears that one needs only assess, for say vectors $c,d\in\RR^p$ that take values in $\{a,b\}$, the moments 
\begin{align*}
	\EE[(w^\trans c)^2(w^\trans d)^2] &= \sum_{i_1i_2j_1j_2} c_{i_1}c_{i_2}d_{j_1}d_{j_2} \EE[w_{i_1}w_{i_2}w_{j_1}w_{j_2}] \\
	&= \sum_{i_1} m_4 c_{i_1}^2d_{i_1}^2 + \sum_{i_1\neq j_1} m_2^2 c_{i_1}^2d_{j_1}^2 + 2 \sum_{i_1\neq i_2} m_2^2 c_{i_1}d_{i_1}c_{i_2}d_{i_2} \\
	&= \sum_{i_1} m_4 c_{i_1}^2d_{i_1}^2 + \left(\sum_{i_1j_1} - \sum_{i_1=j_1} \right) m_2^2 c_{i_1}^2d_{j_1}^2 \nonumber \\ 
	& + 2 \left(\sum_{i_1i_2} - \sum_{i_1\neq i_2} \right) m_2^2 c_{i_1}d_{i_1}c_{i_2}d_{i_2} \\
	&= m_4 (c^2)^\trans (d^2) + m_2^2(\|c\|^2\|d\|^2 - (c^2)^\trans (d^2)) \nonumber \\
	&+ 2 m_2^2 \left( (c^\trans d)^2 - (c^2)^\trans (d^2) \right) \\
	&= (m_4-3m_2^2) (c^2)^\trans (d^2) + m_2^2 \left( \|c\|^2\|d\|^2 + 2 (c^\trans d)^2 \right) \\
	\EE\left[ (w^\trans c)^2(w^\trans d) \right] &= \sum_{i_1i_2j} c_{i_1}c_{i_2}d_j \EE[w_{i_1}w_{i_2}w_j] = \sum_{i_1} m_3 c_{i_1}^2d_{i_1} = m_3 (c^2)d \\
	\EE\left[ (w^\trans c)^2 \right] &= \sum_{i_1i_2} c_{i_1}c_{i_2} \EE[w_{i_1}w_{i_2}] = m_2 \|c\|^2
\end{align*}
where we recall the definition $(a^2)=[a_1^2,\ldots,a_p^2]^\trans$. Gathering all the terms for appropriate selections of $c,d$ leads to \eqref{eq:Phi_poly2}.

\subsection{Heuristic derivation of Conjecture~\ref{claim:Etest}}
\label{sec:proof_Etest}

Conjecture~\ref{claim:Etest} essentially follows as an aftermath of Remark~\ref{rem:Sigma}. We believe that, similar to $\Sigma$, $\hat\Sigma$ is expected to be of the form $\hat\Sigma=\hat\Sigma^\circ+\hat{\bar\sigma}1_{\hat T}^\trans$, where $\hat{\bar\sigma}=\EE[\sigma(w^\trans\hat X)]^\trans$, with $\|\frac{\hat\Sigma^\circ}{\sqrt{T}}\|\leq n^{\varepsilon}$ with high probability. Besides, if $X,\hat X$ were chosen as constituted of Gaussian mixture vectors, with non-trivial growth rate conditions as introduced in \citep{COU16}, it is easily seen that $\bar\sigma=c1_p+v$ and $\hat{\bar\sigma}=c1_p+\hat v$, for some constant $c$ and $\|v\|,\|\hat v\|=O(1)$. 

This subsequently ensures that $\Phi_{X\hat X}$ and $\Phi_{\hat X\hat X}$ would be of a similar form $\Phi_{X\hat X}^\circ+\bar{\sigma}\hat{\bar{\sigma}}^\trans$ and $\Phi_{\hat X\hat X}^\circ+\hat{\bar{\sigma}}\hat{\bar{\sigma}}^\trans$ with $\Phi_{X\hat X}^\circ$ and $\Phi_{\hat X\hat X}^\circ$ of bounded norm. These facts, that would require more advanced proof techniques, let envision the following heuristic derivation for Conjecture~\ref{claim:Etest}.

\bigskip

Recall that our interest is on the test performance $E_{\rm test}$ defined as 
\begin{align*}
E_{\rm test}=\frac1{\hat T}\left\|\hat Y^\trans-\hat \Sigma^\trans\beta\right\|_F^2
\end{align*}
which may be rewritten as 
\begin{align}
E_{\rm test}&=\frac1{\hat T}\tr\left(\hat Y\hat Y^\trans\right)-\frac{2}{T\hat T}\tr\left(YQ\Sigma^\trans\hat \Sigma\hat Y^\trans\right)+\frac1{T^2\hat T}\tr\left(YQ\Sigma^\trans\hat \Sigma\hat \Sigma^\trans\Sigma QY^\trans\right)\nonumber \\
&\equiv Z_1-Z_2+Z_3.\label{eq:E-test}
\end{align}
If $\hat\Sigma=\hat\Sigma^\circ+\hat{\bar\sigma}1_{\hat T}^\trans$ follows the aforementioned claimed operator norm control, reproducing the steps of Corollary~\ref{cor:Etrain} leads to a similar concentration for $E_{\rm test}$, which we shall then admit. We are therefore left to evaluating $\EE[Z_2]$ and $\EE[Z_3]$.

\bigskip

We start with the term $\EE[Z_2]$, which we expand as
\begin{align*}
\EE[Z_2]&=\frac{2}{T\hat T}\EE\left[\tr(YQ\Sigma^\trans\hat \Sigma\hat Y^\trans)\right]=\frac{2}{T\hat T}\sum_{i=1}^n\left[\tr(YQ\sigma_i\hat \sigma_i^\trans\hat Y^\trans)\right]\\
&=\frac{2}{T\hat T}\sum_{i=1}^n\EE\left[\tr\left(\frac{YQ_{-i}\sigma_i\hat \sigma_i^\trans\hat Y^\trans}{1+\frac1T\sigma_i^\trans Q_{-i}\sigma_i}\right)\right]\\
&=\frac{2}{T\hat T}\frac1{1+\delta}\sum_{i=1}^n \EE\left[\tr\left(YQ_{-i}\sigma_i\hat \sigma_i^\trans\hat Y^\trans\right)\right]\\
&+\frac{2}{T\hat T}\frac1{1+\delta}\sum_{i=1}^n \EE\left[\tr\left(YQ_{-i}\sigma_i\hat \sigma_i^\trans\hat Y^\trans\right)\frac{\delta-\frac1T\sigma_i^\trans Q_{-i}\sigma_i}{1+\frac1T\sigma_i^\trans Q_{-i}\sigma_i}\right]\\
&=\frac{2n}{T\hat T}\frac1{1+\delta}\tr\left(Y\EE[Q_{-}]\Phi_{X\hat X}\hat Y^\trans\right)+\frac{2}{T\hat T}\frac1{1+\delta}\EE\left[\tr\left(YQ\Sigma^\trans{D}\hat \Sigma\hat Y^\trans\right)\right]\\
&\equiv Z_{21}+Z_{22}
\end{align*}
with $D=\diag(\{\delta-\frac1T\sigma_i^\trans Q_{-i}\sigma_i\})$, the operator norm of which is bounded by $n^{\varepsilon-\frac12}$ with high probability. Now, observe that, again with the assumption that $\hat\Sigma=\hat\Sigma^\circ+\bar\sigma 1_{\hat T}^\trans$ with controlled $\hat\Sigma^\circ$, $Z_{22}$ may be decomposed as
\begin{align*}
	\frac{2}{T\hat T}\frac1{1+\delta}\EE\left[\tr\left(YQ\Sigma^\trans{D}\hat \Sigma\hat Y^\trans\right)\right] &=\frac{2}{T\hat T}\frac1{1+\delta}\EE\left[\tr\left(YQ\Sigma^\trans{D}\hat\Sigma^\circ \hat Y^\trans\right)\right] \nonumber \\
	&+ \frac{2}{T\hat T}\frac1{1+\delta}1_{\hat T}^\trans \hat Y^\trans \EE\left[ YQ\Sigma^\trans D\bar\sigma \right].
\end{align*}
In the display above, the first right-hand side term is now of order $O(n^{\varepsilon-\frac12})$. As for the second right-hand side term, note that $D\bar\sigma$ is a vector of independent and identically distributed zero mean and variance $O(n^{-1})$ entries; while note formally independent of $YQ\Sigma^\trans$, it is nonetheless expected that this independence ``weakens'' asymptotically (a behavior several times observed in linear random matrix models), so that one expects by central limit arguments that the second right-hand side term be also of order $O(n^{\varepsilon-\frac12})$.

This would thus result in
\begin{align*}
\EE[Z_2]&=\frac{2n}{T\hat T}\frac1{1+\delta}\tr\left(Y\EE[Q_{-}]\Phi_{X\hat X}\hat Y^\trans\right)+O(n^{\varepsilon-\frac12})\\
&=\frac{2n}{T\hat T}\frac1{1+\delta}\tr\left(Y\bar Q\Phi_{X\hat X}\hat Y^\trans\right)+O(n^{\varepsilon-\frac12})\\
&=\frac{2}{\hat T}\tr\left(Y\bar Q\Psi_{X\hat X}\hat Y^\trans\right)+O(n^{\varepsilon-\frac12})
\end{align*}
where we used $\|\EE[Q_{-}]-\bar Q\|\leq Cn^{\varepsilon-\frac12}$ and the definition $\Psi_{X\hat X}=\frac{n}{T}\frac{\Phi_{X\hat X}}{1+\delta}$.

\bigskip

We then move on to $\EE[Z_3]$ of Equation \eqref{eq:E-test}, which can be developed as
\begin{align*}
\EE[Z_3]&=\frac1{T^2\hat T}\EE\left[\tr\left(YQ\Sigma^\trans\hat \Sigma\hat \Sigma^\trans\Sigma QY^\trans\right)\right] \\
&=\frac1{T^2\hat T}\sum_{i,j=1}^n\EE\left[\tr\left(YQ\sigma_i\hat \sigma_i^\trans\hat \sigma_j\sigma_j^\trans QY^\trans\right)\right]\\
&=\frac1{T^2\hat T}\sum_{i,j=1}^n\EE\left[\tr\left(Y\frac{Q_{-i}\sigma_i\hat \sigma_i^\trans}{1+\frac1T\sigma_i^\trans Q_{-i}\sigma_i}\frac{\hat \sigma_j\sigma_j^\trans Q_{-j}}{1+\frac1T\sigma_j^\trans Q_{-j}\sigma_j}Y^\trans\right)\right]\\
&=\frac1{T^2\hat T}\sum_{i=1}^n\sum_{j\neq i}\EE\left[\tr\left(Y\frac{Q_{-i}\sigma_i\hat \sigma_i^\trans}{1+\frac1T\sigma_i^\trans Q_{-i}\sigma_i}\frac{\hat \sigma_j\sigma_j^\trans Q_{-j}}{1+\frac1T\sigma_j^\trans Q_{-j}\sigma_j}Y^\trans\right)\right]\\
&+\frac1{T^2\hat T}\sum_{i=1}^n\EE\left[\tr\left(Y\frac{Q_{-i}\sigma_i\hat \sigma_i^\trans\hat \sigma_i\sigma_i^\trans Q_{-i}}{(1+\frac1T\sigma_i^\trans Q_{-i}\sigma_i)^2}Y^\trans\right)\right]\equiv Z_{31}+Z_{32}.
\end{align*}

In the term $Z_{32}$, reproducing the proof of Lemma~\ref{lem:concentration_quadform} with the condition $\|\hat X\|$ bounded, we obtain that $\frac{\hat \sigma_i^\trans\hat \sigma_i}{\hat T}$ concentrates around $\frac1{\hat T}\tr\Phi_{\hat X\hat X}$, which allows us to write 
\begin{align*}
Z_{32}&=\frac1{T^2\hat T}\sum_{i=1}^n\EE\left[\tr\left(Y\frac{Q_{-i}\sigma_i\tr(\Phi_{\hat X\hat X})\sigma_i^\trans Q_{-i}}{(1+\frac1T\sigma_i^\trans Q_{-i}\sigma_i)^2}Y^\trans\right)\right]\\
&+\frac1{T^2\hat T}\sum_{i=1}^n\EE\left[\tr\left(Y\frac{Q_{-i}\sigma_i\left(\hat \sigma_i^\trans\hat \sigma_i-\tr\Phi_{\hat T}\right)\sigma_i^\trans Q_{-i}}{(1+\frac1T\sigma_i^\trans Q_{-i}\sigma_i)^2}Y^\trans\right)\right]\\
&=\frac1{T^2}\frac{\tr(\Phi_{\hat X\hat X})}{\hat T}\sum_{i=1}^n\EE\left[\tr\left(Y\frac{Q_{-i}\sigma_i\sigma_i^\trans Q_{-i}}{(1+\frac1T\sigma_i^\trans Q_{-i}\sigma_i)^2}Y^\trans\right)\right]\\
&+\frac1{T^2}\sum_{i=1}^n\EE\left[\tr\left(YQ\sigma_i\left(\frac{\hat \sigma_i^\trans\hat \sigma_i-\tr\Phi_{\hat T}}{\hat T}\right)\sigma_i^\trans QY^\trans\right)\right]\\
&\equiv Z_{321}+Z_{322}
\end{align*}
with $D=\diag(\{\frac1{\hat T}\sigma_i^\trans\hat \sigma_i-\frac1{\hat T}\tr\Phi_{\hat T\hat T}\}_{i=1}^n)$ and thus $Z_{322}$ can be rewritten as
\begin{align*}
Z_{322}=\frac1T\EE\left[\tr\left(Y\frac{Q\Sigma^\trans}{\sqrt{T}}{D}\frac{\Sigma Q}{\sqrt{T}}Y^\trans\right)\right]=O(n^{\varepsilon-\frac12})
\end{align*}
while for $Z_{321}$, following the same arguments as previously, we have
\begin{align*}
Z_{321}&=\frac1{T^2}\frac{\tr\Phi_{\hat X\hat X}}{\hat T}\sum_{i=1}^n\EE\left[\tr\left(Y\frac{Q_{-i}\sigma_i\sigma_i^\trans Q_{-i}}{(1+\frac1T\sigma_i^\trans Q_{-i}\sigma_i)^2}Y^\trans\right)\right]\\
&=\frac1{T^2}\frac{\tr\Phi_{\hat X\hat X}}{\hat T}\sum_{i=1}^n\frac1{(1+\delta)^2}\EE\left[\tr\left(YQ_{-i}\sigma_i\sigma_i^\trans Q_{-i}Y^\trans\right)\right]\\
&+\frac1{T^2}\frac{\tr\Phi_{\hat X\hat X}}{\hat T}\sum_{i=1}^n\frac1{(1+\delta)^2}\EE\left[\tr\left(YQ\sigma_i\sigma_i^\trans QY^\trans\right)\left((1+\delta)^2-(1+\frac1T\sigma_i^\trans Q_{-i}\sigma_i)^2\right)\right]\\
&=\frac1{T^2}\frac{\tr\Phi_{\hat X\hat X}}{\hat T}\sum_{i=1}^n\frac1{(1+\delta)^2}\EE\left[\tr\left(YQ_{-i}\Phi_{X} Q_{-i}Y^\trans\right)\right]\\
&+\frac1{T^2}\frac{\tr\Phi_{\hat X\hat X}}{\hat T}\sum_{i=1}^n\frac1{(1+\delta)^2}\EE\left[\tr\left(YQ\Sigma^\trans{D}\Sigma QY^\trans\right)\right]\\
&=\frac{n}{T^2}\EE\left[\tr\left(YQ_{-}\Phi_{X}Q_{-}Y^\trans\right)\right]\frac{\tr(\Phi_{\hat X\hat X})}{\hat T(1+\delta)^2}+O(n^{\varepsilon-\frac12})
\end{align*}
where $D=\diag(\{(1+\delta)^2-(1+\frac1T\sigma_i^\trans Q_{-i}\sigma_i)^2\}_{i=1}^n)$.

Since $\EE[Q_{-}AQ_{-}]=\EE[QAQ]+O_{\|\cdot\|}(n^{\varepsilon-\frac12})$, we are free to plug in the asymptotic equivalent of $\EE[QAQ]$ derived in Section~\ref{subsubsec:QAQ}, and we deduce
\begin{align*}
Z_{32}&=\frac{n}{T^2}\EE\left[\tr Y\left(\bar Q\Phi_{X}\bar Q+\frac{\bar Q\Psi_{X}\bar Q \cdot \frac1{n}\tr\left(\Psi_{X}\bar Q\Phi_{X}\bar Q\right)}{1-\frac1{n}\tr\left(\Psi_{X}^2\bar Q^2\right)}\right)Y^\trans\right]\frac{\tr(\Phi_{\hat X\hat X})}{\hat T(1+\delta)^2}\\
&=\frac{\frac1{n}\tr\left(Y\bar Q\Psi_{X}\bar Q Y^\trans\right)}{1-\frac1{n}\tr\left(\Psi_{X}^2\bar Q^2\right)}\frac1{\hat T}\tr(\Psi_{\hat X\hat X})+O(n^{\varepsilon-\frac12}).
\end{align*}

\bigskip

The term $Z_{31}$ of the double sum over $i$ and $j$ ($j\neq i$) needs more efforts. To handle this term, we need to remove the dependence of both $\sigma_i$ and $\sigma_j$ in $Q$ in sequence. We start with $j$ as follows:
\begin{align*}
Z_{31}&=\frac1{T^2\hat T}\sum_{i=1}^n\sum_{j\neq i}\EE\left[\tr\left(YQ\sigma_i\hat \sigma_i^\trans\frac{\hat \sigma_j\sigma_j^\trans Q_{-j}}{1+\frac1T\sigma_j^\trans Q_{-j}\sigma_j}Y^\trans\right)\right]\\
&=\frac1{T^2\hat T}\sum_{i=1}^n\sum_{j\neq i}\EE\left[\tr\left(Y{Q_{-j}}\sigma_i\hat \sigma_i^\trans\frac{\hat \sigma_j\sigma_j^\trans Q_{-j}}{1+\frac1T\sigma_j^\trans Q_{-j}\sigma_j}Y^\trans\right)\right]\\
&-\frac1{T^3\hat T}\sum_{i=1}^n\sum_{j\neq i}\EE\left[\tr\left(Y\frac{Q_{-j}\sigma_j\sigma_j^\trans{Q_{-j}}\sigma_i\hat \sigma_i^\trans}{1+\frac1T\sigma_j^\trans Q_{-j}\sigma_j}\frac{\hat \sigma_j\sigma_j^\trans{Q_{-j}}}{1+\frac1T\sigma_j^\trans Q_{-j}\sigma_j}Y^\trans\right)\right]\\
&\equiv Z_{311}-Z_{312}
\end{align*}
where in the previous to last inequality we used the relation
\begin{align*}
Q= Q_{-j}- \frac{Q_{-j}\sigma_j \sigma_j^\trans Q_{-j}}{1+\frac1T\sigma_j^\trans Q_{-j} \sigma_j}.
\end{align*}
For $Z_{311}$, we replace $1+\frac1T\sigma_j^\trans{Q_{-j}}\sigma_j$ by $1+\delta$ and take expectation over $w_j$
\begin{align*}
Z_{311}&=\frac1{T^2\hat T}\sum_{i=1}^n\sum_{j\neq i}\EE\left[\tr\left(Y{Q_{-j}}\sigma_i\hat \sigma_i^\trans\frac{\hat \sigma_j\sigma_j^\trans Q_{-j}}{1+\frac1T\sigma_j^\trans Q_{-j}\sigma_j}Y^\trans\right)\right]\\
&=\frac1{T^2\hat T}\sum_{j=1}^n\EE\left[\tr\left(Y\frac{{Q_{-j}}\Sigma_{-j}^\trans\hat \Sigma_{-j}\hat \sigma_j\sigma_j^\trans Q_{-j}}{1+\frac1T\sigma_j^\trans Q_{-j}\sigma_j}Y^\trans\right)\right]\\
&=\frac1{T^2\hat T}\frac1{1+\delta}\sum_{j=1}^n\EE\left[\tr\left(Y{Q_{-j}}\Sigma_{-j}^\trans\hat \Sigma_{-j}\hat \sigma_j\sigma_j^\trans Q_{-j}Y^\trans\right)\right]\\
&+\frac1{T^2\hat T}\frac1{1+\delta}\sum_{j=1}^n\EE\left[\tr\left(Y\frac{{Q_{-j}}\Sigma_{-j}^\trans\hat \Sigma_{-j}\hat \sigma_j\sigma_j^\trans Q_{-j}(\delta-\frac1T\sigma_j^\trans Q_{-j}\sigma_j)}{1+\frac1T\sigma_j^\trans Q_{-j}\sigma_j}Y^\trans\right)\right]\\
&\equiv Z_{3111}+Z_{3112}.
\end{align*}
The idea to handle $Z_{3112}$ is to retrieve forms of the type $\sum_{j=1}^n d_j \hat \sigma_j\sigma_j^\trans=\hat \Sigma^\trans D \Sigma$ for some $D$ satisfying $\|D\|\leq n^{\varepsilon-\frac12}$ with high probability. To this end, we use
\begin{align*}
Q_{-j}\frac{\Sigma_{-j}^\trans \hat \Sigma_{-j}}{T}&=Q_{-j}\frac{\Sigma^\trans \hat \Sigma}{T}-Q_{-j}\frac{\sigma_j \hat \sigma_j^\trans}{T}\\
&=Q\frac{\Sigma^\trans \hat \Sigma}{T}+ \frac{Q \sigma_j \sigma_j^\trans Q}{1-\frac1T \sigma_j^\trans Q \sigma_j}\frac{\Sigma^\trans \hat \Sigma}{T}-Q_{-j}\frac{\sigma_j \hat \sigma_j^\trans}{T}
\end{align*}
and thus $Z_{3112}$ can be expanded as the sum of three terms that shall be studied in order:
\begin{align*}
Z_{3112}&=\frac1{T^2\hat T}\frac1{1+\delta}\sum_{j=1}^n\EE\left[\tr\left(Y\frac{{Q_{-j}}\Sigma_{-j}^\trans\hat \Sigma_{-j}\hat \sigma_j\sigma_j^\trans Q_{-j}(\delta-\frac1T\sigma_j^\trans Q_{-j}\sigma_j)}{1+\frac1T\sigma_j^\trans Q_{-j}\sigma_j}Y^\trans\right)\right]\\
&=\frac1{T\hat T}\frac1{1+\delta}\EE\left[\tr\left(YQ \frac{\Sigma^\trans \hat \Sigma}{T} \hat \Sigma^\trans D \Sigma QY^\trans\right)\right]\\
&+\frac1{T\hat T}\frac1{1+\delta}\sum_{j=1}^n\EE\left[\tr\left(Y\frac{Q\sigma_j \sigma_j^\trans Q \Sigma^\trans\hat \Sigma\hat \sigma_j (\delta-\frac1T\sigma_j^\trans Q_{-j}\sigma_j) \sigma_j^\trans Q}{T(1-\frac1T\sigma_j^\trans Q \sigma_j)}Y^\trans\right)\right]\\
&-\frac1{T^2\hat T}\frac1{1+\delta}\sum_{j=1}^n\EE\left[\tr\left(Y Q\sigma_j\hat \sigma_j^\trans \hat \sigma_j\sigma_j^\trans Q(\delta-\frac1T\sigma_j^\trans Q_{-j}\sigma_j)(1+\frac1T\sigma_j^\trans Q_{-j}\sigma_j)Y^\trans\right)\right]\\
&\equiv Z_{31121}+Z_{31122}-Z_{31123}.
\end{align*}
where $D=\diag(\{\delta-\frac1T\sigma_j^\trans Q_{-j}\sigma_j\}_{i=1}^n)$. First, $Z_{31121}$ is of order $O(n^{\varepsilon-\frac{1}{2}})$ since $Q\frac{\Sigma^\trans \hat \Sigma}{T}$ is of bounded operator norm. Subsequently, $Z_{31122}$ can be rewritten as 
\begin{align*}
Z_{31122}&=\frac1{\hat T}\frac1{1+\delta}\EE\left[\tr\left(Y Q\frac{\Sigma^\trans D \Sigma}{T} Q Y^\trans\right)\right]=O(n^{\varepsilon-\frac{1}{2}})
\end{align*}
with here 
\begin{align*}
	D &=\diag\left\{\frac{\left(\delta-\frac1T\sigma_j^\trans Q_{-j}\sigma_j\right)\left(\frac1T\tr\left(Q_{-j} \frac{\Sigma_{-j}^\trans \hat \Sigma_{-j}}{T}\Phi_{\hat X X}\right)+\frac1T\tr\left(Q_{-j}\Phi\right)\frac1T\tr\Phi_{\hat X\hat X}\right)}{(1-\frac1T\sigma_j^\trans Q \sigma_j)(1+\frac1T \sigma_j^\trans Q_{-j} \sigma_j)}\right\}_{i=1}^n.
\end{align*}
The same arguments apply for $Z_{31123}$ but for
\begin{align*}
	D &=\diag\left\{\frac{\tr\Phi_{\hat X\hat X}}{T}(\delta-\frac1T\sigma_j^\trans Q_{-j}\sigma_j)(1+\frac1T\sigma_j^\trans Q_{-j}\sigma_j)\right\}_{i=1}^n
\end{align*}
which completes to show that $|Z_{3112}|\leq C n^{\varepsilon-\frac{1}{2}}$ and thus
\begin{align*}
Z_{311}&=Z_{3111}+O(n^{\varepsilon-\frac{1}{2}})\\
&=\frac1{T^2\hat T}\frac1{1+\delta}\sum_{j=1}^n\EE\left[\tr\left(Y{Q_{-j}}\Sigma_{-j}^\trans\hat \Sigma_{-j}\hat \sigma_j\sigma_j^\trans Q_{-j}Y^\trans\right)\right]+O(n^{\varepsilon-\frac{1}{2}}).
\end{align*}

It remains to handle $Z_{3111}$. Under the same claims as above, we have
\begin{align*}
Z_{3111}&=\frac1{T\hat T}\frac1{1+\delta}\sum_{j=1}^n\EE\left[\tr\left(Y{Q_{-j}}\frac{\Sigma_{-j}^\trans\hat \Sigma_{-j}}{T}\Phi_{\hat X{X}}Q_{-j}Y^\trans\right)\right]\\
&=\frac1{T\hat T}\frac1{1+\delta}\sum_{j=1}^n\sum_{i\neq j}\EE\left[\tr\left(Y{Q_{-j}}\frac{\sigma_i\hat \sigma_i^\trans}{T}\Phi_{\hat X{X}}Q_{-j}Y^\trans\right)\right]\\
&=\frac1{T^2\hat T}\frac1{1+\delta}\sum_{j=1}^n\sum_{i\neq j}\EE\left[\tr\left(Y\frac{{Q_{-ij}}\sigma_i\hat \sigma_i^\trans}{1+\frac1T\sigma_i^\trans{Q_{-ij}}\sigma_i}\Phi_{\hat X{X}}Q_{-ij}Y^\trans\right)\right]\\
&-\frac1{T^3\hat T}\frac1{1+\delta}\sum_{j=1}^n\sum_{i\neq j}\EE\left[\tr\left(Y\frac{{Q_{-ij}}\sigma_i\hat \sigma_i^\trans}{1+\frac1T\sigma_i^\trans{Q_{-ij}}\sigma_i}\Phi_{\hat X{X}}\frac{Q_{-ij}\sigma_i\sigma_i^\trans{Q_{-ij}}}{1+\frac1T\sigma_i^\trans{Q_{-ij}}\sigma_i}Y^\trans\right)\right]\\
&\equiv Z_{31111}-Z_{31112}
\end{align*}
where we introduced the notation $Q_{-ij}=(\frac1T\Sigma^\trans \Sigma-\frac1T\sigma_i \sigma_i^\trans-\frac1T\sigma_j \sigma_j^\trans+\gamma I_T)^{-1}$.
For $Z_{31111}$, we replace $\frac1T\sigma_i^\trans{Q_{-ij}}\sigma_i$ by $\delta$, and take the expectation over $w_i$, as follows
\begin{align*}
&Z_{31111}=\frac1{T^2\hat T}\frac1{1+\delta}\sum_{j=1}^n\sum_{i\neq j}\EE\left[\tr\left(Y\frac{{Q_{-ij}}\sigma_i\hat \sigma_i^\trans}{1+\frac1T\sigma_i^\trans{Q_{-ij}}\sigma_i}\Phi_{\hat X{X}}Q_{-ij}Y^\trans\right)\right]\\
&=\frac1{T^2\hat T}\frac1{(1+\delta)^2}\sum_{j=1}^n\sum_{i\neq j}\EE\left[\tr\left(Y{Q_{-ij}}\sigma_i\hat \sigma_i^\trans\Phi_{\hat X{X}}Q_{-ij}Y^\trans\right)\right]\\
&+\frac1{T^2\hat T}\frac1{(1+\delta)^2}\sum_{j=1}^n\sum_{i\neq j}\EE\left[\tr\left(Y\frac{{Q_{-ij}}\sigma_i\hat \sigma_i^\trans(\delta-\frac1T\sigma_i^\trans{Q_{-ij}}\sigma_i)}{1+\frac1T\sigma_i^\trans{Q_{-ij}}\sigma_i}\Phi_{\hat X{X}}Q_{-ij}Y^\trans\right)\right]\\
&=\frac{n^2}{T^2\hat T}\frac1{(1+\delta)^2}\EE\left[\tr\left(Y{Q_{--}}\Phi_{X\hat X}\Phi_{\hat X{X}}Q_{--}Y^\trans\right)\right]\\
&+\frac1{T^2\hat T}\frac1{(1+\delta)^2}\sum_{j=1}^n\sum_{i\neq j}\EE\left[\tr\left(Y{Q_{-j}}\sigma_i\hat \sigma_i^\trans \left(\delta-\frac1T\sigma_i^\trans{Q_{-ij}}\sigma_i \right)\Phi_{\hat X{X}}Q_{-j}Y^\trans\right)\right]\\
&+\frac1{T^2\hat T}\frac1{(1+\delta)^2}\sum_{j=1}^n\sum_{i\neq j}\EE\left[\tr\left(Y{Q_{-j}}\sigma_i\hat \sigma_i^\trans\Phi_{\hat X{X}}\frac{Q_{-j}\frac1T\sigma_i\sigma_i^\trans{Q_{-j}}}{1-\frac1T\sigma_i^\trans{Q_{-j}\sigma_i}}Y^\trans \left(\delta-\frac1T\sigma_i^\trans{Q_{-ij}}\sigma_i\right) \right)\right]\\
&=\frac{n^2}{T^2\hat T}\frac1{(1+\delta)^2}\EE\left[\tr\left(Y{Q_{--}}\Phi_{X\hat X}\Phi_{\hat X{X}}Q_{--}Y^\trans\right)\right]\\
&+\frac1{T^2\hat T}\frac1{(1+\delta)^2}\sum_{j=1}^n\EE\left[\tr\left(Y{Q_{-j}}\Sigma_{-j}^\trans{D}\hat \Sigma_{-j}\Phi_{\hat X{X}}Q_{-j}Y^\trans\right)\right]\\
&+\frac{n}{T^2\hat T}\frac1{(1+\delta)^2}\sum_{j=1}^n\EE\left[Y{Q_{-j}}\Sigma_{-j}^\trans{D^\prime}\Sigma_{-j}Q_{-j}Y^\trans\right]+O(n^{\varepsilon-\frac12})\\
&=\frac{n^2}{T^2\hat T}\frac1{(1+\delta)^2}\EE\left[\tr\left(Y{Q_{--}}\Phi_{X\hat X}\Phi_{\hat X{X}}Q_{--}Y^\trans\right)\right]+O(n^{\varepsilon-\frac12})
\end{align*}
with $Q_{--}$ having the same law as $Q_{-ij}$, $D=\diag(\{\delta-\frac1T\sigma_i^\trans{Q_{-ij}}\sigma_i\}_{i=1}^n)$ and $D^\prime=\diag\left\{\frac{(\delta-\frac1T\sigma_i^\trans{Q_{-ij}}\sigma_i)\frac1T\tr\left(\Phi_{\hat X{X}}{Q_{-ij}}\Phi_{X\hat X}\right)}{(1-\frac1T\sigma_i^\trans{Q_{-j}}\sigma_i)(1+\frac1T\sigma_i^\trans{Q_{-ij}}\sigma_i)}\right\}_{i=1}^n$, both expected to be of order $O(n^{\varepsilon-\frac12})$. Using again the asymptotic equivalent of $\EE[QAQ]$ devised in Section~\ref{subsubsec:QAQ}, we then have 
\begin{align*}
Z_{31111}&=\frac{n^2}{T^2\hat T}\frac1{(1+\delta)^2}\EE\left[\tr\left(Y{Q_{--}}\Phi_{X\hat X}\Phi_{\hat X{X}}Q_{--}Y^\trans\right)\right]+O(n^{\varepsilon-\frac12})\\
&=\frac1{\hat T}\tr\left(Y\bar Q\Psi_{X\hat X}\Psi_{\hat X{X}}\bar Q{Y^\trans}\right)+\frac1{\hat T}\tr\left(\Psi_{X}\bar Q\Psi_{X\hat X}\Psi_{\hat X{X}}\bar Q\right)\frac{\frac1{n}\tr\left(Y\bar Q\Psi_{X}{\bar Q}Y^\trans\right)}{1-\frac1{n}\tr(\Psi_X^2\bar Q^2)}\\
&+O(n^{\varepsilon-\frac12}).
\end{align*}

Following the same principle, we deduce for $Z_{31112}$ that
\begin{align*}
Z_{31112}&=\frac1{T^3\hat T}\frac1{1+\delta}\sum_{j=1}^n\sum_{i\neq j}\EE\left[\tr\left(Y\frac{{Q_{-ij}}\sigma_i\hat \sigma_i^\trans}{1+\frac1T\sigma_i^\trans{Q_{-ij}}\sigma_i}\Phi_{\hat X{X}}\frac{Q_{-ij}\sigma_i\sigma_i^\trans{Q_{-ij}}}{1+\frac1T\sigma_i^\trans{Q_{-ij}}\sigma_i}Y^\trans\right)\right]\\
&=\frac1{T^3\hat T}\frac1{(1+\delta)^3}\sum_{j=1}^n\sum_{i\neq j}\EE\left[\tr\left(Y{Q_{-ij}}\sigma_i\sigma_i^\trans{Q_{-ij}}Y^\trans\right)\frac1T\tr\left(\Phi_{\hat X{X}}Q_{-ij}\Phi_{X\hat X}\right)\right]\\
&+\frac1{T^3\hat T}\frac1{(1+\delta)^3}\sum_{j=1}^n\sum_{i\neq j}\EE\left[\tr\left(Y{Q_{-j}}\sigma_i{D_i}\sigma_i^\trans{Q_{-j}}Y^\trans\right)\right]+O(n^{\varepsilon-\frac12})\\
&=\frac{n^2}{T^3\hat T}\frac1{1+\delta}\EE\left[\tr\left(Y{Q_{--}}\Phi_{X}{Q_{--}}Y^\trans\right)\frac1T\tr\left(\Phi_{\hat X{X}}Q_{--}\Phi_{X\hat X}\right)\right]+O(n^{\varepsilon-\frac12})\\
&=\frac1{\hat T}\tr\left(\Psi_{\hat X{X}}\bar Q\Psi_{X\hat X}\right)\frac{\frac1{n}\tr\left(Y\bar Q\Psi_{X}{\bar Q}Y^\trans\right)}{1-\frac1{n}\tr(\Psi_X^2\bar Q^2)}+O(n^{\varepsilon-\frac12}).
\end{align*}
with $D_i=\frac1T\tr\left(\Phi_{\hat X{X}}Q_{-ij}\Phi_{X\hat X}\right)\left[(1+\delta)^2-(1+\frac1T\sigma_i^\trans{Q_{-ij}}\sigma_i)^2\right]$, also believed to be of order $O(n^{\varepsilon-\frac12})$. Recalling the fact that $Z_{311}=Z_{3111}+O(n^{\varepsilon-\frac12})$, we can thus conclude for $Z_{311}$ that
\begin{align*}
Z_{311}&=\frac1{\hat T}\tr\left(Y\bar Q\Psi_{X\hat X}\Psi_{\hat X{X}}\bar Q{Y^\trans}\right)+\frac1{\hat T}\tr\left(\Psi_{X}\bar Q\Psi_{X\hat X}\Psi_{\hat X{X}}\bar Q\right)\frac{\frac1{n}\tr\left(Y\bar Q\Psi_{X}{\bar Q}Y^\trans\right)}{1-\frac1{n}\tr(\Psi_X^2\bar Q^2)}\\
&-\frac1{\hat T}\tr\left(\Psi_{\hat X{X}}\bar Q\Psi_{X\hat X}\right)\frac{\frac1{n}\tr\left(Y\bar Q\Psi_{X}{\bar Q}Y^\trans\right)}{1-\frac1{n}\tr(\Psi_X^2\bar Q^2)}+O(n^{\varepsilon-\frac12}).
\end{align*}

\bigskip

As for $Z_{312}$, we have
\begin{align*}
Z_{312}&=\frac1{T^3\hat T}\sum_{i=1}^n\sum_{j\neq i}\EE\left[\tr\left(Y\frac{Q_{-j}\sigma_j\sigma_j^\trans{Q_{-j}}\sigma_i\hat \sigma_i^\trans}{1+\frac1T\sigma_j^\trans Q_{-j}\sigma_j}\frac{\hat \sigma_j\sigma_j^\trans{Q_{-j}}}{1+\frac1T\sigma_j^\trans Q_{-j}\sigma_j}Y^\trans\right)\right]\\
&=\frac1{T^3\hat T}\sum_{j=1}^n\EE\left[\tr\left(Y\frac{Q_{-j}\sigma_j\sigma_j^\trans{Q_{-j}}\Sigma_{-j}^\trans\hat \Sigma_{-j}}{1+\frac1T\sigma_j^\trans Q_{-j}\sigma_j}\frac{\hat \sigma_j\sigma_j^\trans{Q_{-j}}}{1+\frac1T\sigma_j^\trans Q_{-j}\sigma_j}Y^\trans\right)\right].
\end{align*}
Since $Q_{-j}\frac1T\Sigma_{-j}^\trans\hat \Sigma_{-j}$ is expected to be of bounded norm, using the concentration inequality of the quadratic form $\frac1T\sigma_j^\trans{Q_{-j}}\frac{\Sigma_{-j}^\trans\hat \Sigma_{-j}}{T}\hat \sigma_j$, we infer
\begin{align*}
Z_{312}&=\frac1{T\hat T}\sum_{j=1}^n\EE\left[\tr\left(Y\frac{Q_{-j}\sigma_j\sigma_j^\trans{Q_{-j}}Y^\trans}{(1+\frac1T\sigma_j^\trans Q_{-j}\sigma_j)^2}\right)\left(\frac1{T^2}\tr\left(Q_{-j}\Sigma_{-j}^\trans\hat \Sigma_{-j}\Phi_{\hat X{X}}\right)+O(n^{\varepsilon-\frac12})\right)\right]\\
&=\frac1{T\hat T}\sum_{j=1}^n\EE\left[\tr\left(Y\frac{Q_{-j}\sigma_j\sigma_j^\trans{Q_{-j}}Y^\trans}{(1+\frac1T\sigma_j^\trans Q_{-j}\sigma_j)^2}\right)\left(\frac1{T^2}\tr\left(Q_{-j}\Sigma_{-j}^\trans\hat \Sigma_{-j}\Phi_{\hat X{X}}\right)\right)\right]+O(n^{\varepsilon-\frac12}).\\
\end{align*}
We again replace $\frac1T\sigma_j^\trans{Q_{-j}}\sigma_j$ by $\delta$ and take expectation over $w_j$ to obtain
\begin{align*}
Z_{312}&=\frac1{T\hat T}\frac1{(1+\delta)^2}\sum_{j=1}^n\EE\left[\tr\left(YQ_{-j}\sigma_j\sigma_j^\trans{Q_{-j}}Y^\trans\right)\frac1{T^2}\tr\left(Q_{-j}\Sigma_{-j}^\trans\hat \Sigma_{-j}\Phi_{\hat X{X}}\right)\right]\\
&+\frac1{T\hat T}\frac1{(1+\delta)^2}\sum_{j=1}^n\EE\left[\frac{\tr(YQ_{-j}\sigma_j{D_j}\sigma_j^\trans Q_{-j} Y^\trans)}{(1+\frac1T\sigma_j^\trans Q_{-j}\sigma_j)^2} \frac1{T^2}\tr\left(Q_{-j}\Sigma_{-j}^\trans\hat \Sigma_{-j}\Phi_{\hat X X}\right)\right]+O(n^{\varepsilon-\frac12})\\
&=\frac{n}{T\hat T}\frac1{(1+\delta)^2}\EE\left[\tr\left(YQ_{-}\Phi_{X}{Q_{-}}Y^\trans\right)\frac1{T^2}\tr\left(Q_{-}\Sigma_{-}^\trans\hat \Sigma_{-}\Phi_{\hat X{X}}\right)\right]\\
&+\frac1{T\hat T}\frac1{(1+\delta)^2}\EE\left[\tr\left(YQ\Sigma^\trans{D}\Sigma{Q}Y^\trans\right)\frac1{T^2}\tr\left(Q_{-}\Sigma_{-}^\trans\hat \Sigma_{-}\Phi_{\hat X{X}}\right)\right]+O(n^{\varepsilon-\frac12})
\end{align*}
with $D_j=(1+\delta)^2-(1+\frac1T\sigma_j^\trans Q_{-j}\sigma_j)^2=O(n^{\varepsilon-\frac12})$, which eventually brings the second term to vanish, and we thus get
\begin{align*}
Z_{312}&=\frac{n}{T\hat T}\frac1{(1+\delta)^2}\EE\left[\tr\left(YQ_{-}\Phi_{X}{Q_{-}}Y^\trans\right)\frac1{T^2}\tr\left(Q_{-}\Sigma_{-}^\trans\hat \Sigma_{-}\Phi_{\hat X{X}}\right)\right]+O(n^{\varepsilon-\frac12}).
\end{align*}

For the term $\frac1{T^2}\tr\left(Q_{-}\Sigma_{-}^\trans\hat \Sigma_{-}\Phi_{\hat X{X}}\right)$ we apply again the concentration inequality to get
\begin{align*}
&\frac1{T^2}\tr\left(Q_{-}\Sigma_{-}^\trans\hat \Sigma_{-}\Phi_{\hat X{X}}\right)=\frac1{T^2}\sum_{i\neq j}\tr\left(Q_{-j}\sigma_i\hat \sigma_i^\trans\Phi_{\hat X{X}}\right)\\
&=\frac1{T^2}\sum_{i\neq j}\tr\left(\frac{Q_{-ij}\sigma_i\hat \sigma_i^\trans}{1+\frac1T\sigma_i^\trans{Q_{-ij}}\sigma_i}\Phi_{\hat X{X}}\right)\\
&=\frac1{T^2}\frac1{1+\delta}\sum_{i\neq j}\tr\left(Q_{-ij}\sigma_i\hat \sigma_i^\trans\Phi_{\hat X{X}}\right)+\frac1{T^2}\frac1{1+\delta}\sum_{i\neq j}\tr\left(\frac{Q_{-ij}\sigma_i\hat \sigma_i^\trans(\delta-\frac1T\sigma_i^\trans{Q_{-ij}}\sigma_i)}{1+\frac1T\sigma_i^\trans{Q_{-ij}}\sigma_i}\Phi_{\hat X{X}}\right)\\
&=\frac{n-1}{T^2}\frac1{1+\delta}\tr\left(\Phi_{\hat X{X}}\EE[Q_{--}]\Phi_{X\hat X}\right)+\frac1{T^2}\frac1{1+\delta}\tr\left(Q_{-j}\Sigma_{-j}^\trans{D}\hat \Sigma_{-j}\Phi_{\hat X{X}}\right)+O(n^{\varepsilon-\frac12})
\end{align*}
with high probability, where $D=\diag(\{\delta-\frac1T\sigma_i^\trans{Q_{-ij}}\sigma_i\}_{i=1}^n)$, the norm of which is of order $O(n^{\varepsilon-\frac12})$. This entails 
\begin{equation*}
\frac1{T^2}\tr\left(Q_{-}\Sigma_{-}^\trans\hat \Sigma_{-}\Phi_{\hat X{X}}\right)=\frac{n}{T^2}\frac1{1+\delta}\tr\left(\Phi_{\hat X{X}} \EE[Q_{--}] \Phi_{X\hat X}\right)+O(n^{\varepsilon-\frac12})
\end{equation*}
with high probability.
Once more plugging the asymptotic equivalent of $\EE[QAQ]$ deduced in Section~\ref{subsubsec:QAQ}, we conclude for $Z_{312}$ that
\begin{align*}
Z_{312}&=\frac1{\hat T}\tr\left(\Psi_{\hat X{X}}\bar Q\Psi_{X\hat X}\right)\frac{\frac1{n}\tr\left(Y\bar Q\Psi_{X}{\bar Q}Y^\trans\right)}{1-\frac1{n}\tr(\Psi_X^2\bar Q^2)}+O(n^{\varepsilon-\frac12})
\end{align*}
and eventually for $Z_{31}$
\begin{align*}
Z_{31}&=\frac1{\hat T}\tr\left(Y\bar Q\Psi_{X\hat X}\Psi_{\hat X{X}}\bar Q{Y^\trans}\right)+\frac1{\hat T}\tr\left(\Psi_{X}\bar Q\Psi_{X\hat X}\Psi_{\hat X{X}}\bar Q\right)\frac{\frac1{n}\tr\left(Y\bar Q\Psi_{X}{\bar Q}Y^\trans\right)}{1-\frac1{n}\tr(\Psi_X^2\bar Q^2)}\\
&-\frac{2}{\hat T}\tr\left(\Psi_{\hat X{X}}\bar Q\Psi_{X\hat X}\right)\frac{\frac1{n}\tr\left(Y\bar Q\Psi_{X}{\bar Q}Y^\trans\right)}{1-\frac1{n}\tr(\Psi_X^2\bar Q^2)}+O(n^{\varepsilon-\frac12}).
\end{align*}

Combining the estimates of $\EE[Z_2]$ as well as $Z_{31}$ and $Z_{32}$, we finally have the estimates for the test error defined in \eqref{eq:E-test} as
\begin{align*}
E_{\rm test}&=\frac1{\hat T}\left\|\hat Y^\trans-\Psi_{X\hat X}^\trans\bar Q{Y^\trans}\right\|_F^2\\
&+\frac{\frac1{n}\tr\left(Y\bar Q\Psi_{X}{\bar Q}Y^\trans\right)}{1-\frac1{n}\tr(\Psi_X^2\bar Q^2)}\left[\frac1{\hat T}\tr\Psi_{\hat X\hat X}+\frac1{\hat T}\tr\left(\Psi_{X}\bar Q\Psi_{X\hat X}\Psi_{\hat X{X}}\bar Q\right)-\frac{2}{\hat T}\tr\left(\Psi_{\hat X\hat X}\bar Q\Psi_{X\hat X}\right)\right]\\
&+O(n^{\varepsilon-\frac12}).
\end{align*}
Since by definition, $\bar Q=\left(\Psi_{X}+\gamma{I_T}\right)^{-1}$, we may use
\begin{align*}
\Psi_{X}\bar Q=\left(\Psi_{X}+\gamma{I_T}-\gamma{I_T}\right)\left(\Psi_{X}+\gamma{I_T}\right)^{-1}=I_T-\gamma\bar Q
\end{align*}
in the second term in brackets to finally retrieve the form of Conjecture~\ref{claim:Etest}.

\section{Concluding Remarks}
\label{sec:conclusion}

This article provides a possible direction of exploration of random matrices involving entry-wise non-linear transformations (here through the function $\sigma(\cdot)$), as typically found in modelling neural networks, by means of a concentration of measure approach. The main advantage of the method is that it leverages the concentration of an initial random vector $w$ (here a Lipschitz function of a Gaussian vector) to transfer concentration to all vector $\sigma$ (or matrix $\Sigma$) being Lipschitz functions of $w$. This induces that Lipschitz functionals of $\sigma$ (or $\Sigma$) further satisfy concentration inequalities and thus, if the Lipschitz parameter scales with $n$, convergence results as $n\to\infty$. With this in mind, note that we could have generalized our input-output model $z = \beta^\trans \sigma(Wx)$ of Section~\ref{sec:model} to 
\begin{align*}
	z &= \beta^\trans \sigma(x;\mathcal W)
\end{align*}
for $\sigma:\RR^p\times \mathcal P\to \RR^n$ with $\mathcal P$ some probability space and $\mathcal W\in\mathcal P$ a random variable such that $\sigma(x;\mathcal W)$ and $\sigma(X;\mathcal W)$ (where $\sigma(\cdot)$ is here applied column-wise) satisfy a concentration of measure phenomenon; it is not even necessary that $\sigma(X;\mathcal W)$ has a {\it normal} concentration so long that the corresponding concentration function allows for appropriate convergence results. This generalized setting however has the drawback of being less explicit and less practical (as most neural networks involve linear maps $Wx$ rather than non-linear maps of $\mathcal W$ and $x$).

A much less demanding generalization though would consist in changing the vector $w\sim \mathcal N_\varphi(0,I_p)$ for a vector $w$ still satisfying an exponential (not necessarily normal) concentration. This is the case notably if $w=\varphi(\tilde w)$ with $\varphi(\cdot)$ a Lipschitz map with Lipschitz parameter bounded by, say, $\log(n)$ or any small enough power of $n$. This would then allow for $w$ with heavier than Gaussian tails.

\medskip

Despite its simplicity, the concentration method also has some strong limitations that presently do not allow for a sufficiently profound analysis of the testing mean square error. We believe that Conjecture~\ref{claim:Etest} can be proved by means of more elaborate methods. Notably, we believe that the powerful Gaussian method advertised in \citep{PAS11} which relies on Stein's lemma and the Poincar\'e--Nash inequality could provide a refined control of the residual terms involved in the derivation of Conjecture~\ref{claim:Etest}. However, since Stein's lemma (which states that $\EE[x\phi(x)]=\EE[\phi'(x)]$ for $x\sim\mathcal N(0,1)$ and differentiable polynomially bounded $\phi$) can only be used on products $x\phi(x)$ involving the linear component $x$, the latter is not directly accessible; we nonetheless believe that appropriate ansatzs of Stein's lemma, adapted to the non-linear setting and currently under investigation, could be exploited. 

As a striking example, one key advantage of such a tool would be the possibility to evaluate expectations of the type $Z=\EE[\sigma\sigma^\trans (\frac1T\sigma^\trans Q_-\sigma-\alpha)]$ which, in our present analysis, was shown to be bounded in the order of symmetric matrices by $\Phi Cn^{\varepsilon-\frac12}$ with high probability. Thus, if no matrix (such as $\bar Q$) pre-multiplies $Z$, since $\|\Phi\|$ can grow as large as $O(n)$, $Z$ cannot be shown to vanish. But such a bound does not account for the fact that $\Phi$ would in general be unbounded because of the term $\bar\sigma\bar\sigma^\trans$ in the display $\Phi=\bar{\sigma}\bar{\sigma}^\trans+\EE[(\sigma-\bar\sigma)(\sigma-\bar\sigma)^\trans]$, where $\bar\sigma=\EE[\sigma]$. Intuitively, the ``mean'' contribution $\bar\sigma\bar\sigma^\trans$ of $\sigma\sigma^\trans$, being post-multiplied in $Z$ by $\frac1T\sigma^\trans Q_-\sigma-\alpha$ (which averages to zero) disappears; and thus only smaller order terms remain. We believe that the aforementioned ansatzs for the Gaussian tools would be capable of subtly handling this self-averaging effect on $Z$ to prove that $\|Z\|$ vanishes (for $\sigma(t)=t$, it is simple to show that $\|Z\|\leq Cn^{-1}$). In addition, Stein's lemma-based methods only require the differentiability of $\sigma(\cdot)$, which need not be Lipschitz, thereby allowing for a larger class of activation functions.

\medskip

As suggested in the simulations of Figure~\ref{fig:perf2}, our results also seem to extend to non continuous functions $\sigma(\cdot)$. To date, we cannot envision a method allowing to tackle this setting.

\bigskip

In terms of neural network applications, the present article is merely a first step towards a better understanding of the ``hardening'' effect occurring in large dimensional networks with numerous samples and large data points (that is, simultaneously large $n,p,T$), which we exemplified here through the convergence of mean-square errors. The mere fact that some standard performance measure of these random networks would ``freeze'' as $n,p,T$ grow at the predicted regime and that the performance would heavily depend on the distribution of the random entries is already in itself an interesting result to neural network understanding and dimensioning. However, more interesting questions remain open. Since neural networks are today dedicated to classification rather than regression, a first question is the study of the asymptotic statistics of the output $z=\beta^\trans\sigma(Wx)$ itself; we believe that $z$ satisfies a central limit theorem with mean and covariance allowing for assessing the asymptotic misclassification rate. 

A further extension of the present work would be to go beyond the single-layer network and include multiple layers (finitely many or possibly a number scaling with $n$) in the network design. The interest here would be on the key question of the best distribution of the number of neurons across the successive layers.

It is also classical in neural networks to introduce different (possibly random) biases at the neuron level, thereby turning $\sigma(t)$ into $\sigma(t+b)$ for a random variable $b$ different for each neuron. This has the effect of mitigating the negative impact of the mean $\EE[\sigma(w_i^\trans x_j)]$, which is independent of the neuron index $i$.

Finally, neural networks, despite their having been recently shown to operate almost equally well when taken random in some very specific scenarios, are usually only {\it initiated} as random networks before being subsequently trained through backpropagation of the error on the training dataset (that is, essentially through convex gradient descent). We believe that our framework can allow for the understanding of at least finitely many steps of gradient descent, which may then provide further insights into the overall performance of deep learning networks.

\appendix

\section{Intermediary Lemmas}
\label{sec:lemmas}

This section recalls some elementary algebraic relations and identities used throughout the proof section.

\begin{lemma}[Resolvent Identity]
	\label{lem:resolvent_identity}
	For invertible matrices $A,B$, $A^{-1}-B^{-1}=A^{-1}(B-A)B^{-1}$.
\end{lemma}

\begin{lemma}[A rank-$1$ perturbation identity]
	\label{lem:QQ-}
	For $A$ Hermitian, $v$ a vector and $t\in\RR$, if $A$ and $A+tvv^\trans$ are invertible, then 
	\begin{align*}
		\left( A + tvv^\trans \right)^{-1} v &= \frac{A^{-1}v}{1+t v^\trans A^{-1}v}.
	\end{align*}
\end{lemma}

\begin{lemma}[Operator Norm Control]
	\label{lem:norm_control}
	For nonnegative definite $A$ and $z\in\CC\setminus\RR^+$,
	\begin{align*}
		\| \left( A -z I_T \right)^{-1} \| &\leq {\rm dist}(z,\RR^+)^{-1} \\
		\| A \left( A -z I_T \right)^{-1} \| &\leq 1
	\end{align*}
	where ${\rm dist}(x,\mathcal A)$ is the Hausdorff distance of a point to a set. In particular, for $\gamma>0$, $\| ( A +\gamma I_T )^{-1} \| \leq \gamma^{-1}$ and $\| A ( A +\gamma I_T )^{-1} \| \leq 1$.
\end{lemma}

\bibliographystyle{imsart-nameyear}
\bibliography{/home/romano/Documents/PhD/phd-group/papers/rcouillet/tutorial_RMT/book_final/IEEEabrv.bib,/home/romano/Documents/PhD/phd-group/papers/rcouillet/tutorial_RMT/book_final/IEEEconf.bib,/home/romano/Documents/PhD/phd-group/papers/rcouillet/tutorial_RMT/book_final/tutorial_RMT.bib}

\begin{thebibliography}{41}

\bibitem[\protect\citeauthoryear{Akhiezer and Glazman}{1993}]{AKH93}
\begin{bbook}[author]
\bauthor{\bsnm{Akhiezer},~\bfnm{N.~I.}\binits{N.~I.}} \AND
  \bauthor{\bsnm{Glazman},~\bfnm{I.~M.}\binits{I.~M.}}
(\byear{1993}).
\btitle{Theory of linear operators in Hilbert space}.
\bpublisher{Courier Dover Publications}.
\end{bbook}
\endbibitem

\bibitem[\protect\citeauthoryear{Bai and Silverstein}{1998}]{SIL98}
\begin{barticle}[author]
\bauthor{\bsnm{Bai},~\bfnm{Z.~D.}\binits{Z.~D.}} \AND
  \bauthor{\bsnm{Silverstein},~\bfnm{J.~W.}\binits{J.~W.}}
(\byear{1998}).
\btitle{{No eigenvalues outside the support of the limiting spectral
  distribution of large dimensional sample covariance matrices}}.
\bjournal{The Annals of Probability}
\bvolume{26}
\bpages{316-345}.
\end{barticle}
\endbibitem

\bibitem[\protect\citeauthoryear{Bai and Silverstein}{2007}]{BAI07}
\begin{barticle}[author]
\bauthor{\bsnm{Bai},~\bfnm{Z.~D.}\binits{Z.~D.}} \AND
  \bauthor{\bsnm{Silverstein},~\bfnm{J.~W.}\binits{J.~W.}}
(\byear{2007}).
\btitle{{On the signal-to-interference-ratio of CDMA systems in wireless
  communications}}.
\bjournal{Annals of Applied Probability}
\bvolume{17}
\bpages{81-101}.
\end{barticle}
\endbibitem

\bibitem[\protect\citeauthoryear{Bai and Silverstein}{2009}]{SIL06}
\begin{bbook}[author]
\bauthor{\bsnm{Bai},~\bfnm{Z.~D.}\binits{Z.~D.}} \AND
  \bauthor{\bsnm{Silverstein},~\bfnm{J.~W.}\binits{J.~W.}}
(\byear{2009}).
\btitle{{Spectral analysis of large dimensional random matrices}},
\bedition{second} ed.
\bpublisher{Springer Series in Statistics}, \baddress{New York, NY, USA}.
\end{bbook}
\endbibitem

\bibitem[\protect\citeauthoryear{Benaych-Georges and Nadakuditi}{2012}]{BEN12}
\begin{barticle}[author]
\bauthor{\bsnm{Benaych-Georges},~\bfnm{F.}\binits{F.}} \AND
  \bauthor{\bsnm{Nadakuditi},~\bfnm{R.~R.}\binits{R.~R.}}
(\byear{2012}).
\btitle{The singular values and vectors of low rank perturbations of large
  rectangular random matrices}.
\bjournal{Journal of Multivariate Analysis}
\bvolume{111}
\bpages{120--135}.
\end{barticle}
\endbibitem

\bibitem[\protect\citeauthoryear{Cambria et~al.}{2015}]{CAM15}
\begin{barticle}[author]
\bauthor{\bsnm{Cambria},~\bfnm{Erik}\binits{E.}},
  \bauthor{\bsnm{Gastaldo},~\bfnm{Paolo}\binits{P.}},
  \bauthor{\bsnm{Bisio},~\bfnm{Federica}\binits{F.}} \AND
  \bauthor{\bsnm{Zunino},~\bfnm{Rodolfo}\binits{R.}}
(\byear{2015}).
\btitle{An ELM-based model for affective analogical reasoning}.
\bjournal{Neurocomputing}
\bvolume{149}
\bpages{443--455}.
\end{barticle}
\endbibitem

\bibitem[\protect\citeauthoryear{Choromanska et~al.}{2015}]{CHO14}
\begin{binproceedings}[author]
\bauthor{\bsnm{Choromanska},~\bfnm{Anna}\binits{A.}},
  \bauthor{\bsnm{Henaff},~\bfnm{Mikael}\binits{M.}},
  \bauthor{\bsnm{Mathieu},~\bfnm{Michael}\binits{M.}},
  \bauthor{\bsnm{Arous},~\bfnm{G{\'e}rard~Ben}\binits{G.~B.}} \AND
  \bauthor{\bsnm{LeCun},~\bfnm{Yann}\binits{Y.}}
(\byear{2015}).
\btitle{The Loss Surfaces of Multilayer Networks}.
In \bbooktitle{AISTATS}.
\end{binproceedings}
\endbibitem

\bibitem[\protect\citeauthoryear{Couillet and {Benaych-Georges}}{2016}]{COU16}
\begin{barticle}[author]
\bauthor{\bsnm{Couillet},~\bfnm{R.}\binits{R.}} \AND
  \bauthor{\bsnm{{Benaych-Georges}},~\bfnm{F.}\binits{F.}}
(\byear{2016}).
\btitle{Kernel spectral clustering of large dimensional data}.
\bjournal{Electronic Journal of Statistics}
\bvolume{10}
\bpages{1393--1454}.
\end{barticle}
\endbibitem

\bibitem[\protect\citeauthoryear{Couillet and Kammoun}{2016}]{COU17}
\begin{binproceedings}[author]
\bauthor{\bsnm{Couillet},~\bfnm{Romain}\binits{R.}} \AND
  \bauthor{\bsnm{Kammoun},~\bfnm{Abla}\binits{A.}}
(\byear{2016}).
\btitle{Random Matrix Improved Subspace Clustering}.
In \bbooktitle{2016 Asilomar Conference on Signals, Systems, and Computers}.
\end{binproceedings}
\endbibitem

\bibitem[\protect\citeauthoryear{Couillet, Pascal and
  Silverstein}{2015}]{COU13b}
\begin{barticle}[author]
\bauthor{\bsnm{Couillet},~\bfnm{R.}\binits{R.}},
  \bauthor{\bsnm{Pascal},~\bfnm{F.}\binits{F.}} \AND
  \bauthor{\bsnm{Silverstein},~\bfnm{J.~W.}\binits{J.~W.}}
(\byear{2015}).
\btitle{{The random matrix regime of Maronna's M-estimator with elliptically
  distributed samples}}.
\bjournal{Journal of Multivariate Analysis}
\bvolume{139}
\bpages{56--78}.
\end{barticle}
\endbibitem

\bibitem[\protect\citeauthoryear{Giryes, Sapiro and Bronstein}{2015}]{GIR15}
\begin{barticle}[author]
\bauthor{\bsnm{Giryes},~\bfnm{Raja}\binits{R.}},
  \bauthor{\bsnm{Sapiro},~\bfnm{Guillermo}\binits{G.}} \AND
  \bauthor{\bsnm{Bronstein},~\bfnm{Alex~M.}\binits{A.~M.}}
(\byear{2015}).
\btitle{Deep Neural Networks with Random Gaussian Weights: A Universal
  Classification Strategy?}
\bjournal{{IEEE} Transactions on Signal Processing}
\bvolume{64}
\bpages{3444-3457}.
\end{barticle}
\endbibitem

\bibitem[\protect\citeauthoryear{Hornik, Stinchcombe and White}{1989}]{HOR89}
\begin{barticle}[author]
\bauthor{\bsnm{Hornik},~\bfnm{Kurt}\binits{K.}},
  \bauthor{\bsnm{Stinchcombe},~\bfnm{Maxwell}\binits{M.}} \AND
  \bauthor{\bsnm{White},~\bfnm{Halbert}\binits{H.}}
(\byear{1989}).
\btitle{Multilayer feedforward networks are universal approximators}.
\bjournal{Neural networks}
\bvolume{2}
\bpages{359--366}.
\end{barticle}
\endbibitem

\bibitem[\protect\citeauthoryear{Hoydis, Couillet and Debbah}{2013}]{COU11d}
\begin{barticle}[author]
\bauthor{\bsnm{Hoydis},~\bfnm{J.}\binits{J.}},
  \bauthor{\bsnm{Couillet},~\bfnm{R.}\binits{R.}} \AND
  \bauthor{\bsnm{Debbah},~\bfnm{M.}\binits{M.}}
(\byear{2013}).
\btitle{{Random beamforming over quasi-static and fading channels: a
  deterministic equivalent approach}}.
\bjournal{{IEEE} Transactions on Information Theory}
\bvolume{58}
\bpages{6392-6425}.
\end{barticle}
\endbibitem

\bibitem[\protect\citeauthoryear{Huang, Zhu and Siew}{2006}]{HUA06}
\begin{barticle}[author]
\bauthor{\bsnm{Huang},~\bfnm{Guang-Bin}\binits{G.-B.}},
  \bauthor{\bsnm{Zhu},~\bfnm{Qin-Yu}\binits{Q.-Y.}} \AND
  \bauthor{\bsnm{Siew},~\bfnm{Chee-Kheong}\binits{C.-K.}}
(\byear{2006}).
\btitle{Extreme learning machine: theory and applications}.
\bjournal{Neurocomputing}
\bvolume{70}
\bpages{489--501}.
\end{barticle}
\endbibitem

\bibitem[\protect\citeauthoryear{Huang et~al.}{2012}]{HUA12}
\begin{barticle}[author]
\bauthor{\bsnm{Huang},~\bfnm{Guang-Bin}\binits{G.-B.}},
  \bauthor{\bsnm{Zhou},~\bfnm{Hongming}\binits{H.}},
  \bauthor{\bsnm{Ding},~\bfnm{Xiaojian}\binits{X.}} \AND
  \bauthor{\bsnm{Zhang},~\bfnm{Rui}\binits{R.}}
(\byear{2012}).
\btitle{Extreme learning machine for regression and multiclass classification}.
\bjournal{Systems, Man, and Cybernetics, Part B: Cybernetics, IEEE Transactions
  on}
\bvolume{42}
\bpages{513--529}.
\end{barticle}
\endbibitem

\bibitem[\protect\citeauthoryear{Jaeger and Haas}{2004}]{JAE04}
\begin{barticle}[author]
\bauthor{\bsnm{Jaeger},~\bfnm{H.}\binits{H.}} \AND
  \bauthor{\bsnm{Haas},~\bfnm{H.}\binits{H.}}
(\byear{2004}).
\btitle{Harnessing nonlinearity: Predicting chaotic systems and saving energy
  in wireless communication}.
\bjournal{Science}
\bvolume{304}
\bpages{78--80}.
\end{barticle}
\endbibitem

\bibitem[\protect\citeauthoryear{Kammoun et~al.}{2009}]{KAM09}
\begin{barticle}[author]
\bauthor{\bsnm{Kammoun},~\bfnm{A.}\binits{A.}},
  \bauthor{\bsnm{Kharouf},~\bfnm{M.}\binits{M.}},
  \bauthor{\bsnm{Hachem},~\bfnm{W.}\binits{W.}} \AND
  \bauthor{\bsnm{Najim},~\bfnm{J.}\binits{J.}}
(\byear{2009}).
\btitle{A central limit theorem for the sinr at the lmmse estimator output for
  large-dimensional signals}.
\bjournal{{IEEE} Transactions on Information Theory}
\bvolume{55}
\bpages{5048--5063}.
\end{barticle}
\endbibitem

\bibitem[\protect\citeauthoryear{{El Karoui}}{2009}]{ELK09}
\begin{barticle}[author]
\bauthor{\bsnm{{El Karoui}},~\bfnm{N.}\binits{N.}}
(\byear{2009}).
\btitle{Concentration of measure and spectra of random matrices: applications
  to correlation matrices, elliptical distributions and beyond}.
\bjournal{The Annals of Applied Probability}
\bvolume{19}
\bpages{2362--2405}.
\end{barticle}
\endbibitem

\bibitem[\protect\citeauthoryear{{El Karoui}}{2010}]{ELK10}
\begin{barticle}[author]
\bauthor{\bsnm{{El Karoui}},~\bfnm{N.}\binits{N.}}
(\byear{2010}).
\btitle{The spectrum of kernel random matrices}.
\bjournal{The Annals of Statistics}
\bvolume{38}
\bpages{1--50}.
\end{barticle}
\endbibitem

\bibitem[\protect\citeauthoryear{{El Karoui}}{2013}]{KAR13}
\begin{barticle}[author]
\bauthor{\bsnm{{El Karoui}},~\bfnm{N}\binits{N.}}
(\byear{2013}).
\btitle{Asymptotic behavior of unregularized and ridge-regularized
  high-dimensional robust regression estimators: rigorous results}.
\bjournal{arXiv preprint arXiv:1311.2445}.
\end{barticle}
\endbibitem

\bibitem[\protect\citeauthoryear{Krizhevsky, Sutskever and
  Hinton}{2012}]{KRI12}
\begin{binproceedings}[author]
\bauthor{\bsnm{Krizhevsky},~\bfnm{Alex}\binits{A.}},
  \bauthor{\bsnm{Sutskever},~\bfnm{Ilya}\binits{I.}} \AND
  \bauthor{\bsnm{Hinton},~\bfnm{Geoffrey~E.}\binits{G.~E.}}
(\byear{2012}).
\btitle{Imagenet classification with deep convolutional neural networks}.
In \bbooktitle{Advances in neural information processing systems}
\bpages{1097--1105}.
\end{binproceedings}
\endbibitem

\bibitem[\protect\citeauthoryear{LeCun, Cortes and Burges}{1998}]{MNIST}
\begin{bmisc}[author]
\bauthor{\bsnm{LeCun},~\bfnm{Y.}\binits{Y.}},
  \bauthor{\bsnm{Cortes},~\bfnm{C.}\binits{C.}} \AND
  \bauthor{\bsnm{Burges},~\bfnm{C.}\binits{C.}}
(\byear{1998}).
\btitle{{The MNIST database of handwritten digits}}.
\end{bmisc}
\endbibitem

\bibitem[\protect\citeauthoryear{Ledoux}{2005}]{LED05}
\begin{bbook}[author]
\bauthor{\bsnm{Ledoux},~\bfnm{Michel}\binits{M.}}
(\byear{2005}).
\btitle{The concentration of measure phenomenon}
\bvolume{89}.
\bpublisher{American Mathematical Soc.}
\end{bbook}
\endbibitem

\bibitem[\protect\citeauthoryear{Loubaton and Vallet}{2010}]{LOU10b}
\begin{barticle}[author]
\bauthor{\bsnm{Loubaton},~\bfnm{P.}\binits{P.}} \AND
  \bauthor{\bsnm{Vallet},~\bfnm{P.}\binits{P.}}
(\byear{2010}).
\btitle{{Almost sure localization of the eigenvalues in a Gaussian information
  plus noise model. Application to the spiked models}}.
\bjournal{Electronic Journal of Probability}
\bvolume{16}
\bpages{1934--1959}.
\end{barticle}
\endbibitem

\bibitem[\protect\citeauthoryear{Mai and Couillet}{2017}]{MAI17}
\begin{binproceedings}[author]
\bauthor{\bsnm{Mai},~\bfnm{Xiaoyi}\binits{X.}} \AND
  \bauthor{\bsnm{Couillet},~\bfnm{Romain}\binits{R.}}
(\byear{2017}).
\btitle{The counterintuitive mechanism of graph-based semi-supervised learning
  in the big data regime}.
In \bbooktitle{IEEE International Conference on Acoustics, Speech and Signal
  Processing (ICASSP'17)}.
\end{binproceedings}
\endbibitem

\bibitem[\protect\citeauthoryear{Mar\u{c}enko and Pastur}{1967}]{MAR67}
\begin{barticle}[author]
\bauthor{\bsnm{Mar\u{c}enko},~\bfnm{V.~A.}\binits{V.~A.}} \AND
  \bauthor{\bsnm{Pastur},~\bfnm{L.~A.}\binits{L.~A.}}
(\byear{1967}).
\btitle{{Distribution of eigenvalues for some sets of random matrices}}.
\bjournal{Math USSR-Sbornik}
\bvolume{1}
\bpages{457-483}.
\end{barticle}
\endbibitem

\bibitem[\protect\citeauthoryear{Pastur and {\^S}erbina}{2011}]{PAS11}
\begin{bbook}[author]
\bauthor{\bsnm{Pastur},~\bfnm{L.}\binits{L.}} \AND
  \bauthor{\bsnm{{\^S}erbina},~\bfnm{M.}\binits{M.}}
(\byear{2011}).
\btitle{Eigenvalue distribution of large random matrices}.
\bpublisher{American Mathematical Society}.
\end{bbook}
\endbibitem

\bibitem[\protect\citeauthoryear{Rahimi and Recht}{2007}]{RAH07}
\begin{binproceedings}[author]
\bauthor{\bsnm{Rahimi},~\bfnm{Ali}\binits{A.}} \AND
  \bauthor{\bsnm{Recht},~\bfnm{Benjamin}\binits{B.}}
(\byear{2007}).
\btitle{Random features for large-scale kernel machines}.
In \bbooktitle{Advances in neural information processing systems}
\bpages{1177--1184}.
\end{binproceedings}
\endbibitem

\bibitem[\protect\citeauthoryear{Rosenblatt}{1958}]{ROS58}
\begin{barticle}[author]
\bauthor{\bsnm{Rosenblatt},~\bfnm{Frank}\binits{F.}}
(\byear{1958}).
\btitle{The perceptron: a probabilistic model for information storage and
  organization in the brain.}
\bjournal{Psychological review}
\bvolume{65}
\bpages{386}.
\end{barticle}
\endbibitem

\bibitem[\protect\citeauthoryear{Rudelson et~al.}{2013}]{RUD13}
\begin{barticle}[author]
\bauthor{\bsnm{Rudelson},~\bfnm{Mark}\binits{M.}},
  \bauthor{\bsnm{Vershynin},~\bfnm{Roman}\binits{R.}} \betal{et~al.}
(\byear{2013}).
\btitle{{Hanson-Wright inequality and sub-Gaussian concentration}}.
\bjournal{Electron. Commun. Probab}
\bvolume{18}
\bpages{1--9}.
\end{barticle}
\endbibitem

\bibitem[\protect\citeauthoryear{Saxe et~al.}{2011}]{SAX11}
\begin{binproceedings}[author]
\bauthor{\bsnm{Saxe},~\bfnm{Andrew}\binits{A.}},
  \bauthor{\bsnm{Koh},~\bfnm{Pang~W}\binits{P.~W.}},
  \bauthor{\bsnm{Chen},~\bfnm{Zhenghao}\binits{Z.}},
  \bauthor{\bsnm{Bhand},~\bfnm{Maneesh}\binits{M.}},
  \bauthor{\bsnm{Suresh},~\bfnm{Bipin}\binits{B.}} \AND
  \bauthor{\bsnm{Ng},~\bfnm{Andrew~Y}\binits{A.~Y.}}
(\byear{2011}).
\btitle{On random weights and unsupervised feature learning}.
In \bbooktitle{Proceedings of the 28th international conference on machine
  learning (ICML-11)}
\bpages{1089--1096}.
\end{binproceedings}
\endbibitem

\bibitem[\protect\citeauthoryear{Schmidhuber}{2015}]{SCH15}
\begin{barticle}[author]
\bauthor{\bsnm{Schmidhuber},~\bfnm{J{\"u}rgen}\binits{J.}}
(\byear{2015}).
\btitle{Deep learning in neural networks: An overview}.
\bjournal{Neural Networks}
\bvolume{61}
\bpages{85--117}.
\end{barticle}
\endbibitem

\bibitem[\protect\citeauthoryear{Silverstein and Bai}{1995}]{SIL95}
\begin{barticle}[author]
\bauthor{\bsnm{Silverstein},~\bfnm{J.~W.}\binits{J.~W.}} \AND
  \bauthor{\bsnm{Bai},~\bfnm{Z.~D.}\binits{Z.~D.}}
(\byear{1995}).
\btitle{{On the empirical distribution of eigenvalues of a class of large
  dimensional random matrices}}.
\bjournal{Journal of Multivariate Analysis}
\bvolume{54}
\bpages{175-192}.
\end{barticle}
\endbibitem

\bibitem[\protect\citeauthoryear{Silverstein and Choi}{1995}]{CHO95}
\begin{barticle}[author]
\bauthor{\bsnm{Silverstein},~\bfnm{J.~W.}\binits{J.~W.}} \AND
  \bauthor{\bsnm{Choi},~\bfnm{S.}\binits{S.}}
(\byear{1995}).
\btitle{{Analysis of the limiting spectral distribution of large dimensional
  random matrices}}.
\bjournal{Journal of Multivariate Analysis}
\bvolume{54}
\bpages{295-309}.
\end{barticle}
\endbibitem

\bibitem[\protect\citeauthoryear{Tao}{2012}]{TAO12}
\begin{bbook}[author]
\bauthor{\bsnm{Tao},~\bfnm{Terence}\binits{T.}}
(\byear{2012}).
\btitle{Topics in random matrix theory}
\bvolume{132}.
\bpublisher{American Mathematical Soc.}
\end{bbook}
\endbibitem

\bibitem[\protect\citeauthoryear{Titchmarsh}{1939}]{TIT39}
\begin{bbook}[author]
\bauthor{\bsnm{Titchmarsh},~\bfnm{E.~C.}\binits{E.~C.}}
(\byear{1939}).
\btitle{{The Theory of Functions}}.
\bpublisher{Oxford University Press}, \baddress{New York, NY, USA}.
\end{bbook}
\endbibitem

\bibitem[\protect\citeauthoryear{Vershynin}{2012}]{VER12}
\begin{barticle}[author]
\bauthor{\bsnm{Vershynin},~\bfnm{Roman}\binits{R.}}
(\byear{2012}).
\btitle{Introduction to the non-asymptotic analysis of random matrices}.
\bjournal{in Compressed Sensing, 210--268, Cambridge University Press}.
\end{barticle}
\endbibitem

\bibitem[\protect\citeauthoryear{Williams}{1998}]{WIL98}
\begin{barticle}[author]
\bauthor{\bsnm{Williams},~\bfnm{Christopher K.~I.}\binits{C.~K.~I.}}
(\byear{1998}).
\btitle{Computation with infinite neural networks}.
\bjournal{Neural Computation}
\bvolume{10}
\bpages{1203--1216}.
\end{barticle}
\endbibitem

\bibitem[\protect\citeauthoryear{Yates}{1995}]{YAT95}
\begin{barticle}[author]
\bauthor{\bsnm{Yates},~\bfnm{R.~D.}\binits{R.~D.}}
(\byear{1995}).
\btitle{{A framework for uplink power control in cellular radio systems}}.
\bjournal{{IEEE} Journal on Selected Areas in Communications}
\bvolume{13}
\bpages{1341-1347}.
\end{barticle}
\endbibitem

\bibitem[\protect\citeauthoryear{Zhang, Cheng and Singer}{2014}]{ZHA14}
\begin{barticle}[author]
\bauthor{\bsnm{Zhang},~\bfnm{T.}\binits{T.}},
  \bauthor{\bsnm{Cheng},~\bfnm{X.}\binits{X.}} \AND
  \bauthor{\bsnm{Singer},~\bfnm{A.}\binits{A.}}
(\byear{2014}).
\btitle{{Marchenko-Pastur Law for Tyler's and Maronna's M-estimators}}.
\bjournal{http://arxiv.org/abs/1401.3424}.
\end{barticle}
\endbibitem

\bibitem[\protect\citeauthoryear{Zhenyu~Liao}{2017}]{LIA17b}
\begin{barticle}[author]
\bauthor{\bsnm{Zhenyu~Liao},~\bfnm{Romain~Couillet}\binits{R.~C.}}
(\byear{2017}).
\btitle{A Large Dimensional Analysis of Least Squares Support Vector Machines}.
\bjournal{(submitted to) Journal of Machine Learning Research}.
\end{barticle}
\endbibitem

\end{thebibliography}

\end{document}